\documentclass[11pt, reqno]{article}
\usepackage[T1]{fontenc}

\newcommand{\papernumber}{I}
\renewcommand{\thesection}{\papernumber.\arabic{section}}
\usepackage{amsmath,amssymb,amsfonts,amsthm}
\usepackage{color}
\usepackage{xr-hyper}
\usepackage[colorlinks=true,linkcolor=blue,citecolor=blue,urlcolor=blue,pdfborder={0 0 0}]{hyperref}
\usepackage{cleveref}

\usepackage{caption}
\usepackage{thmtools}

\usepackage{mathrsfs}

\crefname{theorem}{Theorem}{Theorems}
\crefname{thm}{Theorem}{Theorems}
\crefname{lemma}{Lemma}{Lemmas}
\crefname{lem}{Lemma}{Lemmas}
\crefname{remark}{Remark}{Remarks}
\crefname{prop}{Proposition}{Propositions}
\crefname{defn}{Definition}{Definitions}
\crefname{corollary}{Corollary}{Corollaries}
\crefname{conjecture}{Conjecture}{Conjectures}
\crefname{question}{Question}{Questions}
\crefname{chapter}{Chapter}{Chapters}
\crefname{section}{Section}{Sections}
\crefname{figure}{Figure}{Figures}
\crefname{example}{Example}{Examples}

\theoremstyle{plain}
\newtheorem{thm}{Theorem}[section]

\newtheorem{lemma}[thm]{Lemma}
\newtheorem{theorem}[thm]{Theorem}

\newtheorem{corollary}[thm]{Corollary}

\newtheorem{prop}[thm]{Proposition}

\theoremstyle{definition}
\newtheorem{defn}[thm]{Definition}

\theoremstyle{remark}
\newtheorem{remark}[thm]{Remark}

\numberwithin{equation}{section}

\renewcommand{\P}{\mathbb P}
\newcommand{\E}{\mathbb E}

\newcommand{\R}{\mathbb R}
\newcommand{\Z}{\mathbb Z}
\newcommand{\N}{\mathbb N}

\newcommand{\eqd} {\overset{d}{=}}

\newcommand{\cE}{\mathcal E}

\newcommand{\cG}{\mathcal G}
\newcommand{\cH}{\mathcal H}

\newcommand{\sC}{\mathscr C}

\newcommand{\bbK}{\mathbb K}

\newcommand{\bbN}{\mathbb N}

\newcommand{\bbT}{\mathbb T}

\def\P{\mathbb{P}}

\newcommand{\1}{{\rm 1\hspace*{-0.4ex}%
\rule{0.1ex}{1.52ex}\hspace*{0.2ex}}}

\DeclareMathSymbol{\leqslant}{\mathalpha}{AMSa}{"36} 
\DeclareMathSymbol{\geqslant}{\mathalpha}{AMSa}{"3E} 
\DeclareMathSymbol{\eset}{\mathalpha}{AMSb}{"3F}     

\makeatletter
\newcommand{\hathat}[1]{%
\begingroup%
  \let\macc@kerna\z@%
  \let\macc@kernb\z@%
  \let\macc@nucleus\@empty%
  \hat{\mathchoice%
    {\raisebox{.2ex}{\vphantom{\ensuremath{\displaystyle #1}}}}%
    {\raisebox{.2ex}{\vphantom{\ensuremath{\textstyle #1}}}}%
    {\raisebox{.16ex}{\vphantom{\ensuremath{\scriptstyle #1}}}}%
    {\raisebox{.14ex}{\vphantom{\ensuremath{\scriptscriptstyle #1}}}}%
    \smash{\hat{#1}}}%
\endgroup%
}
\makeatother

\renewcommand{\epsilon}{\varepsilon}

\usepackage[margin=1in]{geometry}
\usepackage{tikz}
\usepackage{pgfplots}
\usepackage{extarrows}
\usepackage{titling}
\usepackage{commath}
\usepackage{wasysym}
\usepackage{mathtools}
\usepackage{titlesec}
\newcommand{\eps}{\varepsilon}
\usepackage{bbm}
\usepackage{setspace}
\usepackage{enumitem}
\usepackage{booktabs}
\usetikzlibrary{calc}

\usepackage[font=footnotesize,labelfont=bf]{caption}

\usepackage[textsize=tiny]{todonotes}

\usepackage{cite}

\makeatletter
\pgfdeclareshape{slash underlined}
{
  \inheritsavedanchors[from=rectangle] 
  \inheritanchorborder[from=rectangle]
  \inheritanchor[from=rectangle]{north}
  \inheritanchor[from=rectangle]{north west}
  \inheritanchor[from=rectangle]{north east}
  \inheritanchor[from=rectangle]{center}
  \inheritanchor[from=rectangle]{west}
  \inheritanchor[from=rectangle]{east}
  \inheritanchor[from=rectangle]{mid}
  \inheritanchor[from=rectangle]{mid west}
  \inheritanchor[from=rectangle]{mid east}
  \inheritanchor[from=rectangle]{base}
  \inheritanchor[from=rectangle]{base west}
  \inheritanchor[from=rectangle]{base east}
  \inheritanchor[from=rectangle]{south}
  \inheritanchor[from=rectangle]{south west}
  \inheritanchor[from=rectangle]{south east}
  \inheritanchorborder[from=rectangle]
  \foregroundpath{
    \southwest \pgf@xa=\pgf@x \pgf@ya=\pgf@y
    \northeast \pgf@xb=\pgf@x \pgf@yb=\pgf@y
    \pgf@xc=\pgf@xa
    \advance\pgf@xc by .5\pgf@xb
    \pgf@yc=\pgf@ya
    \advance\pgf@xc by -1.3pt
    \advance\pgf@yc by -1.8pt
    \pgfpathmoveto{\pgfqpoint{\pgf@xc}{\pgf@yc}}
    \advance\pgf@xc by  2.6pt
    \advance\pgf@yc by  3.6pt
    \pgfpathlineto{\pgfqpoint{\pgf@xc}{\pgf@yc}}
    \pgfpathmoveto{\pgfqpoint{\pgf@xa}{\pgf@ya}}
    \pgfpathlineto{\pgfqpoint{\pgf@xb}{\pgf@ya}}
 }
}
\tikzset{nomorepostaction/.code=\let\tikz@postactions\pgfutil@empty}

\makeatother

\pretitle{\begin{flushleft}\Large}
\posttitle{\par\end{flushleft}\vskip 0.5em
}
\preauthor{\begin{flushleft}}
\postauthor{
\\
\vspace{0.5em}
\footnotesize{The Division of Physics, Mathematics and Astronomy, California Institute of Technology\\ Email: 
\href{mailto:t.hutchcroft@caltech.edu}{t.hutchcroft@caltech.edu}}
\end{flushleft}}
\predate{\begin{flushleft}}
\postdate{\par\end{flushleft}}
\title{{\bf Critical long-range percolation I: High effective dimension}}

\renewenvironment{abstract}
 {\par\noindent\textbf{\abstractname.}\ \ignorespaces}
 {\par\medskip}

\titleformat{\section}
  {\normalfont\large \bf}{\thesection}{0.4em}{}

\author{{\bf Tom Hutchcroft}}

\makeatletter
\newcommand\mytag[2][]{%
  \def\@currentlabel{#2}%
  (#2)\label{#1} 
}
\makeatother

\begin{document}

\date{\small{\today}}

\maketitle

\begin{abstract}
In long-range Bernoulli bond percolation on the $d$-dimensional lattice $\Z^d$, each pair of points $x$ and $y$ are connected by an edge with probability $1-\exp(-\beta\|x-y\|^{-d-\alpha})$, where $\alpha>0$ is fixed, $\beta \geq 0$ is the parameter that is varied to induce a phase transition, and $\|\cdot\|$ is a norm. As $d$ and $\alpha$ are varied, the model is conjectured to exhibit eight qualitatively different forms of second-order critical behaviour, with a transition between a mean-field regime and a low-dimensional regime satisfying the  hyperscaling relations when $d=\min\{6,3\alpha\}$, a transition between effectively long- and short-range regimes at a crossover value $\alpha=\alpha_c(d)$, and with various logarithmic corrections to these behaviours occurring at the boundaries between these regimes. 

This is the first of a series\footnote{An expository account of the series over 12 hours of lectures is available on the YouTube channel of the Fondation Hadamard, see \url{https://www.youtube.com/playlist?list=PLbq-TeAWSXhPQt8MA9_GAdNgtEbvswsfS}.}
 of three papers developing a rigorous and detailed theory of the model's critical behaviour in five of these eight regimes, including all effectively long-range regimes and all effectively high-dimensional regimes.
In this paper, we introduce our non-perturbative real-space renormalization group method and use it to analyze the high-effective-dimensional regime $d>\min\{6,3\alpha\}$. In particular, we compute the tail of the cluster volume and establish the superprocess scaling limits of the model, which transition between super-L\'evy and super-Brownian behaviour when $\alpha=2$. All our results hold unconditionally for $d> 3\alpha$, without any perturbative assumptions on the model; beyond this regime, when $d> 6$ and $\alpha \geq d/3$, they hold under the assumption that appropriate two-point function estimates hold as provided for spread-out models by the lace expansion. Our results on scaling limits also hold (with possible slowly-varying corrections to scaling) in the critical-dimensional regime with $d=3\alpha<6$ subject to a marginal-triviality condition we call the \emph{hydrodynamic condition}; this condition is verified in the third paper in this series, in which we also compute the precise logarithmic corrections to mean-field scaling when $d=3\alpha<6$.
\end{abstract}

\newpage

\setcounter{tocdepth}{2}

\tableofcontents

\setstretch{1.1}

\newpage

\section{Introduction}
\label{sec:introduction}

\subsection{Critical phenomena in long-range percolation}

A fundamental objective of statistical mechanics  is to understand \emph{critical phenomena}, that is, the rich, fractal-like behaviour exhibited by many large, complex systems at and near their points of phase transition.  Within the context of lattice models,
such critical phenomena can be studied either through the language of \emph{critical exponents}, which govern the power-law growth or decay of various quantities at or near criticality, or through \emph{scaling limits} (a.k.a.\ \emph{continuum limits}), in which one attempts to take a limit of the ``entire system'' as the lattice spacing goes to zero, yielding a more complete picture of large-scale behaviour than is provided by critical exponents alone.

Two of the most fascinating (and still largely conjectural) aspects of critical phenomena are \emph{universality} and \emph{dimension dependence}: the large-scale behaviour of many models at and near criticality are predicted to be independent of various ``small-scale details'' of the model, such as the precise choice of $d$-dimensional lattice, but to depend strongly on the dimension in which the model is defined. In particular, many models are either known or predicted to have an  \emph{upper critical dimension} $d_c$ such that if $d>d_c$ then the model has relatively simple, ``mean-field'' critical behaviour, while if $d<d_c$ then the model has more complex critical behaviour with both quantitative and qualitative distinctions from the mean-field regime. For models with \emph{long-range} (meaning infinite-range) \emph{interactions}, there is predicted to be a further qualitative distinction between \emph{effectively long-range} and \emph{effectively short-range} regimes as the decay rate of the long-range interaction is varied. This leads to (at least) four qualitatively distinct regimes of critical behaviour: short-range high-dimensional, long-range high-dimensional, short-range low-dimensional and long-range low-dimensional. (More accurately the word ``effectively'' should appear before the name of each regime, with long-range models able to be ``effectively high-dimensional'' even for $d=1$ as we discuss below.) Further, subtly distinct behaviours occur on the boundaries between these regimes,
 leading to at least eight qualitatively distinct behaviours in total. (See \cref{fig:cartoon}.)

While the basic features of the theory have been understood heuristically in the theoretical physics literature since Wilson's seminal introduction of his renormalization group method in the 1970s \cite{wilson1975renormalization} (which was first applied to models with long-range interactions by Fisher, Ma, and Nickel \cite{fisher1972critical} and Sak \cite{sak1973recursion}; see also the more recent works \cite{behan2017scaling,brezin2014crossover}),
 the development of a corresponding rigorous theory has proven extremely challenging.
  In the context of finite-range Bernoulli percolation, we now have a very good understanding in high dimensions \cite{MR1043524,MR2748397,MR762034,MR1127713,MR1959796,fitzner2015nearest,MR2551766} (see also \cite{MR2430773,MR3306002,MR4032873,liu2025high} for related results for long-range percolation) and a fairly good understanding in two dimensions  \cite{smirnov2001critical2,smirnov2001critical,lawler2002one,MR879034,duminil2020rotational,camia2024conformal}, although the best two-dimensional results remain limited to \emph{site percolation on the triangular lattice} and universality remains a major challenge. For intermediate dimensions $2<d<6$ and the critical dimension $d=6$ there appears to be a complete lack of tools to adequately address the problem. Moreover, for $2<d<6$ there are no conjectures regarding exact values for critical exponents nor any known reason to expect that closed-form expressions for these exponents should exist. (For $d=6$ we have conjectures \cite{essam1978percolation,ruiz1998logarithmic} but not yet proofs, as we discuss in more detail in \cref{III-sec:introduction}.) Even in the high-dimensional setting, all existing methods for finite-range models (most of which build on the \emph{lace expansion} \cite{MR782962,MR1043524})  are \emph{perturbative} in the sense that they require a ``small parameter'' to work and cannot be used to prove mean-field critical behaviour under the minimal assumption that $d>d_c=6$. (Instead one must either take $d$ very large or replace the standard hypercubic lattice with, say, a ``spread-out'' lattice in which $x\neq y\in \Z^d$ are considered adjacent if they have distance at most some large constant $L$.)


Over the past few years, significant progress has been made understanding critical phenomena in \emph{long-range percolation} \cite{hutchcroft2020power,hutchcroft2022sharp,baumler2022isoperimetric,hutchcroft2024pointwise} beyond the high-dimensional regime. This progress has focused on understanding the critical \emph{two-point function}\footnote{That is, the probability $\tau_{\beta_c}(x,y)$ that two points belong to the same critical cluster.}, which was predicted\footnote{Sak worked with the $O(n)$ model, but his prediction also applies similarly to percolation (see \cite{brezin2014crossover,gori2017one}). Sak's prediction built on the earlier, less precise work of Fisher, Ma, and Nickel \cite{fisher1972critical}; see also \cite{suzuki1972wilson}.} by Sak \cite{sak1973recursion} to have simple scaling exponents throughout the entire effectively long-range regime that is insensitive to the distinction between high and low effective dimension. Following these works, Sak's prediction has been completely verified in dimension $1$, with significant partial progress made in general dimensions. On the other hand, both the methods and results of the papers \cite{hutchcroft2022sharp,hutchcroft2024pointwise,baumler2022isoperimetric,hutchcroft2020power} are blind to the distinction between low and high effective dimensions and seem unsuitable for the study of more delicate properties of critical long-range percolation where this distinction is important.

\begin{figure}
\centering
\includegraphics[scale=0.9]{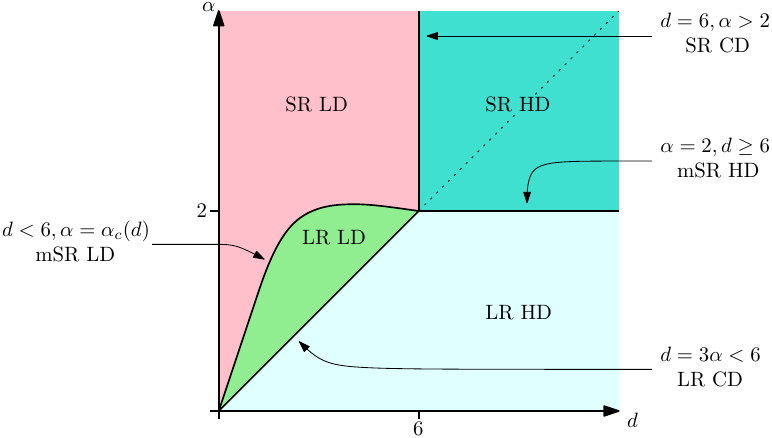}\hspace{1cm}
\caption{Schematic illustration of the different regimes of critical behaviour for long-range percolation. 
 LR, SR, HD, LD, and CD stand for (effectively) ``Long Range'', ``Short Range'', ``High Dimensional'', ``Low Dimensional'', and ``Critical Dimensional'' respectively, while mSR stands for ``marginally Short Range''. Here we ignore the special behaviours occuring when $d=1$ (where there is either no phase transition when $\alpha>1$ or a discontinuous phase transition when $\alpha=1$ \cite{MR868738,duminil2020long}) to avoid clutter. In this paper we study critical behaviour in the effectively high-dimensional regimes LR HD, SR HD, mSR HD, and the effectively long-range critical-dimensional regime LR CD; our results concerning the critical line LR CD hold conditionally on the validity of the \emph{hydrodynamic condition}, which is verified in the third paper of the series. Within the regions SR HD and mSR HD, our results are proven with no further assumptions below the dotted line $d=3\alpha$ but are proven only under the numerical assumptions needed to implement the lace expansion on and above this line; this is not expected to correspond to any true change in behaviour of the model.}
\label{fig:cartoon}
\end{figure}


This is the first in a series of three papers \cite{LRPpaper2,LRPpaper3} in which we develop a rigorous and detailed theory of the transition between low and high-dimensional critical behaviour for long-range percolation on $\Z^d$, going well beyond previously obtained results on two-point functions and treating the (effectively) long-range low-dimensional, long-range high-dimensional, and short-range high-dimensional regimes as well as the boundaries between these regimes. (Equivalently, everything other than the effectively short-range low-dimensional regime and its boundaries.)  This also leads to a more complete resolution of Sak's prediction as discussed in the second paper of this series. 
The theory we develop closely parallels that which we previously developed in the simpler setting of \emph{hierarchical percolation} in our earlier works \cite{hutchcrofthierarchical,hutchcroft2022critical} (but does not assume familiarity with these works). Extending this theory to the Euclidean setting entails both significant technical challenges and genuinely new phenomena that need to be understood, since hierarchical models are always ``effectively long-range'' and do not exhibit the transition to an ``effectively short-range'' regime seen in Euclidean models. Moreover, the results in this series of papers are in some ways significantly stronger than their hierarchical predecessors, including e.g.\ results on scaling limits and $k$-point functions that were not yet established in the hierarchical case.


In this paper, we introduce the core principles of our method, which can be thought of as a kind of \emph{real-space renormalization group} (RG) \emph{method}, and apply this method to analyze both critical exponents and scaling limits in the high effective dimensional regime. The development of a form of RG that can be applied directly to ``geometric'' statistical mechanics models of probability theory such as percolation and self-avoiding walk (rather than to ``spin systems'' like the $\varphi^4$ model that can be written as ``Gaussian field with interactions'' as in more traditional rigorous approaches to RG) is a central contribution of the paper; further discussion of how our methods compare to these traditional approaches is given in \cref{subsec:RG}. 
We also prove our scaling limit results for effectively long-range models \emph{at the critical dimension} subject to a technical ``marginal triviality'' condition we call the \emph{hydrodynamic condition} that is verified in the third paper in the series. 
 The second and third papers concern effectively long-range models in low and critical effective dimension respectively. 


While our results are arguably most interesting in low dimensions and at the critical dimension since the high-dimensional case can also be treated using existing methods such as the lace expansion \cite{MR2430773,MR3306002,MR4032873,liu2025high}, we stress that the high-dimensional applications of our methods developed in this paper still have significant advantages over the lace expansion since (a) they apply \emph{non-perturbatively} in the regime where the model is ``mean-field because the kernel is sufficiently heavy tailed'' (i.e., when $d>3\alpha$ in the notation we introduce below), without needing any extraneous numerical ``small parameter'' assumptions;
 (b) they yield results not just on critical exponents but on the \emph{full scaling limit} of the model, which have not yet been established via other methods; (c) they lead directly to our analysis of the low- and critical-dimensional regimes. 

The remainder of this introduction is organized as follows: In \cref{subsec:definitions} we give precise definitions of the class of models we consider and outline the rudiments of the method we use to study them. We state our main results concerning the high-effective-dimensional regime in \cref{subsec:statement_of_high_dimensional_results} and their (conditional) extensions to the critical dimension in \cref{subsec:scaling_limits_in_the_critical_dimension_under_the_hydrodynamic_condition}. We then give a brief overview of how our methods compare to other approaches to RG in \cref{subsec:RG}.

\medskip\noindent \textbf{Asymptotic notation.}
We write $\asymp$, $\preceq$, and $\succeq$ for equalities and inequalities that hold to within positive multiplicative constants depending on $d\geq 1$, $\alpha>0$, and the kernel $J$
but not on any other parameters. If implicit constants depend on an additional parameter (such as the index of a moment), this will be indicated using a subscript. Landau's asymptotic notation is used similarly, so that if $f$ is a non-negative function then ``$f(n)=O(n)$ for every $n\geq 1$'' and ``$f(n) \preceq n$ for every $n\geq 1$'' both mean that there exists a positive constant $C$ such that $f(n)\leq C n$ for every $n\geq 1$. We also write $f(n)=o(g(n))$ to mean that $f(n)/g(n)\to 0$ as $n\to\infty$ and write $f(n)\sim g(n)$ to mean that $f(n)/g(n)\to 1$ as $n\to\infty$. The quantities we represent implicitly using big-$O$ and little-$o$ notation are always taken to be non-negative, and we denote quantities of uncertain sign using $\pm O$ or $\pm o$ as appropriate.








\subsection{Definition of the model and the real-space RG framework}
\label{subsec:definitions}


In \textbf{long-range percolation} on $\Z^d$, we define a random graph with vertex set $\Z^d$ and where each pair of vertices $x\neq y$ forms an edge $\{x,y\}$ with probability $1-e^{-\beta J(x,y)}$, independently of all other pairs, where $\beta \geq 0$ is a parameter and $J:\Z^d\times \Z^d \to [0,\infty)$ is a kernel that is symmetric and translation-invariant in the sense that $J(x,y)=J(y,x)=J(0,y-x)$. We will be most interested in the case that
\begin{equation}
\label{eq:weak_kernel_assumption}
  J(x,y)\sim \text{const.}\,\|x-y\|^{-d-\alpha}
\end{equation}
as $\|x-y\|\to \infty$ for some $\alpha>0$ and norm $\|\cdot\|$. (In fact we will work under slightly stronger assumptions on our kernel described in \cref{def:model_def} below.) 
 Edges that are included in the configuration are referred to as \textbf{open}, edges not included in the configuration are referred to as \textbf{closed}, and connected components of the resulting random graph are referred to as \textbf{clusters}. We write $\P_\beta$ and $\E_\beta$ for probabilities and expectations taken with respect to this random subgraph, and define the critical point $\beta_c$ to be
\[
  \beta_c := \inf \{\beta\geq 0: \text{an infinite cluster exists $\P_\beta$-a.s.}\},
\]
which satisfies $0<\beta_c<\infty$ whenever $d\geq 2$ and $\alpha>0$ or $d=1$ and $0<\alpha \leq 1$ \cite{newman1986one,schulman1983long,MR436850}. It is a theorem of Berger \cite{MR1896880} (see also \cite{hutchcroft2020power,hutchcroft2022sharp}) that the phase transition is \emph{continuous} in the sense that there are no infinite clusters at $\beta_c$ whenever $0<\alpha <d$. While it is conjectured that this is true for all $\alpha>0$ when $d\geq 2$, the model has a \emph{discontinuous} phase transition (meaning that an infinite cluster exists at $\beta=\beta_c$) when $d=\alpha=1$ \cite{MR868738,duminil2020long}.

\begin{remark}
While our focus in these papers is entirely on the study of \emph{critical phenomena} in long-range percolation, there are many other interesting features of long-range percolation that have been studied in the literature such as the geometry of the infinite supercritical cluster \cite{biskup2021arithmetic,ding2023uniqueness,baumler2023distances}, and we refer the reader to e.g.\ \cite{heydenreich2015progress} for a broader introduction.
\end{remark}

\begin{remark}
As mentioned above, traditional approaches to analyzing high-dimensional percolation via the lace expansion \cite{MR2430773,MR3306002,MR4032873} often require one to work with \emph{spread-out models}. For long-range models, this means working with kernels that are approximately constant up to some large constant scale, such as $J(x,y) = (\max \{L^{-1}\|x-y\|,1\})^{-d-\alpha}$. The precise conditions needed for these perturbative methods to work can be stated more generally in terms of numerical assumptions on $J$ and its Fourier transform, see e.g.\ \cite[Assumption 1.1]{MR3306002}.
\end{remark}

We will work throughout this series of papers with the following slightly restricted class of kernels satisfying \eqref{eq:weak_kernel_assumption} as described in the following definition. The notation and normalization conventions introduced in this definition will be in force throughout the rest of the paper unless specified otherwise.

\begin{defn}[Assumptions on the kernel]
\label{def:model_def}
Let $d\geq 1$, let $\|\cdot\|$ be a norm on $\R^d$ and let $\alpha>0$. We consider translation-invariant kernels in which $J(x,y)$ is a decreasing, differentiable function of $\|x-y\|$ when $x\neq y$ (and vanishing otherwise), so that we can write
\begin{equation}
J(x,y) = \mathbbm{1}(x\neq y) J(\|x-y\|) = \mathbbm{1}(x\neq y) \int_{\|x-y\|}^\infty |J'(s)| \dif s.
\label{eq:Kernel_definition}
\end{equation}
We will assume moreover that
\[
|J'(r)| = C(1+\delta_r) r^{-d-\alpha-1}
\]
for some constant $C>0$ and measurable, logarithmically integrable error function\footnote{This means that $\delta_r\to 0$ as $r\to \infty$, $\delta_r$ is locally integrable on some interval of the form $[r_0,\infty)$, and $\int_{r_0}^\infty |\delta_r|\frac{\dif r}{r}<\infty$. Our terminology regarding error estimates and regularity is discussed in more detail in \cref{subsec:regular_variation_and_logarithmic_integrability}.} $\delta_r$, as is the case when $J(r)=r^{-d-\alpha}$ and $\delta_r\equiv 0$. \textbf{We may assume without loss of generality that the unit ball in $\|\cdot\|$ has unit Lebesgue measure and that $C=1$, so that \begin{equation}
\label{eq:normalization_conventions}
\tag{$*$}
|J'(r)|\sim r^{-d-\alpha-1} \qquad \text{ and } \qquad |B_r|:=|\{x\in \Z^d:\|x\|\leq r\}|\sim r^d,\end{equation}and will do this throughout the paper.} (Any other norm and kernel of the form we consider can be normalized to satisfy \eqref{eq:normalization_conventions}, which changes the value of $\beta_c$ but does not change other features of the model at criticality.)
\end{defn}

\begin{remark}
All of our results remain true if \eqref{eq:Kernel_definition} is only assumed to hold for $\|x-y\|$ sufficiently large, so that $J(r)$ could e.g.\ be non-differentiable for small $r$. (Indeed, the small-scale kernel need not even be a function of the norm.) This allows one to treat kernels such as $J(x,y)=\max\{L^{-1}\|x-y\|,1\}^{-d-\alpha}$ that are popular in the lace expansion literature. We do not pursue such generalizations here as they would increase the notational complexity of our proofs without leading to any additional phenomena at large scales. (If one wants to consider spread-out models within our framework, one can just use a kernel with a smoother cut-off such as $J(x,y)=(1+L^{-1}\|x-y\|)^{-d-\alpha}$.) 
\end{remark}

As alluded to above, one interesting feature of long-range percolation (and other long-range models) is that varying $\alpha$ is in some ways closely analogous to varying the dimension, but can sensibly be done in a continuous fashion. This phenomenon can be described in part via the \textbf{effective dimension}, which is defined to be
\[d_\mathrm{eff} = \max\left\{d, \frac{2d}{\alpha}\right\}.\]
(Equivalently, the effective dimension is the \emph{spectral dimension} of the random walk with jump kernel proportional to $\|x-y\|^{-d-\alpha}$.) While the effective dimension has its uses, we caution the reader not to rely on it too strongly: It gives the correct prediction that the model is mean-field when $d_\mathrm{eff}>d_c=6$ and satisfies the hyperscaling relations when $d_\mathrm{eff}<d_c=6$, but is not predicted to fully determine the critical behaviour of the model. 
(In particular, when $d<6$ it is \emph{not} predicted that the the long-range model begins to have the same critical behaviour as the short-range model at the point $\alpha=2$ where $d_\mathrm{eff}$ becomes equal to $d$: it is instead predicted that this already happens at the slightly smaller value $43/24\approx 1.79\ldots$ when $d=2$ and does not happen until values of $\alpha$ slightly larger than $2$ when $d=3,4,5$; see \cref{fig:cartoon}.)


\begin{defn}[The cut-off kernel]
Given a kernel $J$ as above and a parameter $r>0$, we write $J_r$ for the \textbf{cut-off kernel}
\[
J_r(x,y) = \mathbbm{1}(x\neq y, \|x-y\|\leq r) \int_{\|x-y\|}^r |J'(s)| \dif s.
\]
and write $\P_{\beta,r},\E_{\beta,r}$ for probabilities and expectations taken with respect to the law of long-range percolation on $\Z^d$ with kernel $J_r$ at parameter $\beta$. We will always write $\beta_c$ for the critical parameter associated to the original kernel $J$, so that the cut-off measure $\P_{\beta_c,r}$ is subcritical at $\beta_c$ when $r<\infty$. (Indeed, for every $r<\infty$ there exists $\eps=\eps(r)>0$ such that $J_r(x,y)\leq (1-\eps)J(x,y)$ for every $x\neq y$, so that $\P_{\beta_c,r}$ is stochastically dominated by the subcritical measure $\P_{(1-\eps)\beta_c}$.)
\end{defn}

Our analysis of critical long-range percolation will be based on understanding the asymptotic blow-up of various observables under the measure $\E_{\beta_c,r}$ as $r\to \infty$, with the most important being the cluster volume moments $\E_{\beta_c,r}|K|^p$, the typical size of the largest cluster in the box of radius $r$, denoted $M_{\beta_c,r}$ and defined in detail in \cref{subsec:correlation_inequalities_and_the_universal_tightness_theorem}, and spatial moments such as $\E_{\beta_c,r}\sum_{x\in K}\|x\|_2^p$ and, say, $\E_{\beta_c,r}\sum_{x,y\in K}\|x\|_2^p\|y\|_2^q$, where $K$ always denotes the cluster of the origin. This basic framework of studying the asymptotics of these observables as a function of the cut-off parameter $r$ should be thought of as a \textbf{real-space renormalization group} (RG) \textbf{method}; we give an overview of how this compares to other RG methods in \cref{subsec:RG}. (We do not assume that the reader has any previous familiarity with RG; since our implementation is so different from other approaches this does not have much effect on the exposition anyway.) 
 When $d_\mathrm{eff} > d_c=6$ or $d_\mathrm{eff}=d_c=6$ and $\alpha\leq 2$, we study these quantities by proving that they satisfy a certain system of ODEs asymptotically as $r\to \infty$, which for the cluster volume moments can be written explicitly as
\begin{equation}
\label{eq:volume_moments_ODE_intro}
  \frac{d}{dr}\E_{\beta_c,r}|K|^p \sim \beta_c r^{-\alpha-1} \sum_{\ell=0}^{p-1}\binom{p}{\ell} \E_{\beta_c,r}|K|^{\ell+1}\E_{\beta_c,r}|K|^{p-\ell}.
\end{equation}
(The leading constants appearing here are determined by our normalization assumptions \eqref{eq:normalization_conventions} on the kernel $J$ and the norm $\|\cdot\|$.)
This infinite system of limiting ODEs is triangular in the sense that the derivative of the $p$th moment can be expressed to first order in terms of the moments of order $p$ and smaller, leading to a simple exact solution to the system in the limit. In the case $d=3\alpha<6$ the second-order corrections to these limiting ODEs are only smaller than the leading-order contributions by a polylogarithmic factor, so that they potentially change the first-order asymptotics of the solutions and lead to logarithmic corrections to scaling. 

\medskip

 The transition between the effectively long- and short-range regimes can also be understood  in high effective dimension straightforwardly through the  asymptotic ODEs for the spatially weighted moments: The simplest example of such an asymptotic ODE is
\begin{equation}
\frac{d}{dr} \E_{\beta_c,r} \left[\sum_{x\in K} \|x\|_2^2 \right]
\sim  \frac{2 \alpha}{r} \E_{\beta_c,r} \left[\sum_{x\in K} \|x\|_2^2 \right] + \frac{1}{r}\left(\frac{\alpha^2}{\beta_c} \int_{B} \|y\|_2^2 \dif y\right) r^{2+\alpha},
\label{eq:x12_derivative_asymptotic2_intro}
\end{equation}
where $B=\{x\in \R^d:\|x\|\leq 1\}$.
Solutions to this ODE have different asymptotic forms according to whether $\alpha$ is bigger than, smaller than, or equal to $2$, due to the competition between the coefficient of $2\alpha$ in the self-referential term and the exponent $2+\alpha$ in the driving term: the term of order $r^{2+\alpha-1}$ is smaller than the first term by a power when $\alpha>2$, the same order as the first term when $\alpha<2$, and leads to logarithmic corrections in the marginally short-range case $\alpha=2$ (mSR HD in \cref{fig:cartoon}); see \cref{lem:ODE_with_driving_term,lem:ODE_with_possibly_negligible_driving_term,lem:ODE_with_critical_driving_term} for the relevant general ODE lemmas.

\medskip

It is instructive to compare this situation with that holding for nearest-neighbour models. For spread-out models in high-effective dimensions (which may be long-range or short-range), it is predicted that there exists a constant $A$ such that the $\beta$ derivative of the susceptibility satisfies
\[
  \frac{d}{d\beta} \E_\beta |K| \sim A (\E_\beta |K|)^2
\]
as $\beta \uparrow \beta_c$. The constant $A$ is a complicated, non-universal quantity that encodes the microscopic structure of the model. The lace expansion, when convergent, should give an expression for $A$ as an infinite series (although as far as we can tell this computation has not yet appeared in the literature); see \cite[Eq.\ 5.71]{MR2239599} for an analogous asymptotic estimate for self-avoiding walk. In contrast, our asymptotic ODE for the $r$ derivative of the same quantity when $d>d_\mathrm{eff}$ is just
\begin{equation}
\label{eq:ODE_first_moment_intro}
  \frac{d}{dr} \E_{\beta_c,r} |K| \sim \beta_c r^{-\alpha-1} (\E_{\beta_c,r} |K|)^2,
\end{equation}
with all relevant constants taking simple explicit values. Moreover, it is conjectured that for nearest-neighbour percolation in the critical dimension $d=6$, the asymptotic relationship above must be modified to
\[
  \frac{d}{d\beta} \E_\beta |K| \sim A_\beta (\E_\beta |K|)^2
\]
where $A_\beta$ is a slowly varying function of order $(\log |\beta-\beta_c|)^{-2/7}$ as $\beta \uparrow \beta_c$ \cite{essam1978percolation}. Meanwhile, as we establish in \cref{lem:first_moment_critical_dim} and \cref{III-thm:critical_dim_hydro}, our simple asymptotic ODE \eqref{eq:ODE_first_moment_intro} \emph{continues to hold} when $d_\mathrm{eff}=6$ and $\alpha\leq 2$, with the logarithmic corrections to scaling in this case entering only via the \emph{second order} corrections to the limiting ODEs.
 Intuitively, the simple asymptotic form of the $r$ derivatives owes to the exact asymptotic independence (in high effective dimensions) of the two clusters at the endpoints of a \emph{long} pivotal edge, which can hold even at the critical dimension; for the $\beta$-derivative we must look at \emph{short} pivotal edges where asymptotic independence is too much to ask for and there is at best a ``small constant'' interaction between the two clusters. 

\medskip

Of course, to implement this analysis we must actually \emph{prove} the validity of the asymptotic ODEs such as \eqref{eq:volume_moments_ODE_intro} in high effective dimension. In the case $d>3\alpha$, where the model has mean-field behaviour ``because the kernel is heavy tailed enough'', we give a short and direct proof of this via a bootstrapping argument that does not rely on any previous results on long-range percolation (see \cref{lem:first_moment_high_dimensional}). This argument is simple enough that it can be applied \emph{mutatis mutandis} to a wide range of other (long-range) models including the Ising model, self-avoiding walk, and lattice trees (which will be covered in future work) and is one of the primary contributions of this paper. In the marginal case $d=3\alpha$ these limiting ODEs are much more difficult to justify, and indeed their proof in the third paper of this series constitutes the most technical part of this series of papers. When $d_\mathrm{eff}<d_c=6$ and the model is effectively long-range (i.e., in the regime LR LD of \cref{fig:cartoon}) we show in the second paper in this series that the asymptotic ODE \eqref{eq:volume_moments_ODE_intro} \emph{cannot} hold, which we then argue forces the \emph{hyperscaling relations} to hold, allowing for the computation of various non-mean-field critical exponents in this case.

\subsection{Statement of high-dimensional results}
\label{subsec:statement_of_high_dimensional_results}

We now state our main theorems concerning the model in the high effective dimension. All our results in this regime will be established under one of the three following hypotheses, which we refer to collectively as \mytag[HD]{{\color{blue}HD}}\!\!:
 \begin{enumerate}
 \item $d> 3\alpha$.
 \item $d>6$ and 
  $\P_{\beta_c}(x\leftrightarrow y) \preceq \|x-y\|^{-d+2}$ for all $x\neq y$.
 \item $d=6$, $\alpha=2$, and 
  $\P_{\beta_c}(x\leftrightarrow y) \preceq (\log \|x-y\|)^{-1} \|x-y\|^{-d+2}$ for all $x\neq y$.
\end{enumerate}
The first case ($d>3\alpha$) is arguably the most interesting as our results in this case are entirely non-perturbative (i.e., do not need any small parameter assumptions as required by the lace expansion); this region includes the entire effectively long range high-dimensional regime, the part of the effectively \emph{short range} high-dimensional regime in which $2< \alpha <d/3$, and every point in the \emph{marginally short range}  high-dimensional  curve $\alpha=2$, $d\geq 6$ other than the doubly-marginal point $d=6$, $\alpha=2$. Hypotheses 2 and 3 of \eqref{HD} cover \emph{spread-out} models in the remaining high-dimensional regime and the doubly-marginal point $d=6$, $\alpha=2$, where the relevant two-point estimates have been established under perturbative assumptions using the lace expansion \cite{MR2430773,MR4032873,MR3306002}.

 Note that mean-field critical behaviour was previously established  non-perturbatively (at the level of exponents) under the stronger assumption that $\alpha< \min\{d/3,1\}$ \cite{hutchcroft2024pointwise}, without the precise first-order asymptotics on the volume tail or the scaling limit results proven here.

\medskip

\noindent \textbf{The volume distribution.}
Our first main result on the high-dimensional model establishes first-order asymptotics of the cluster volume moments $\E_{\beta_c,r}|K|^p$ as $r\to \infty$ and applies these to establish the first-order asymptotics of the volume tail $\P_{\beta_c}(|K|\geq n)$ as $n\to \infty$. 
  Note that the constant $\alpha/\beta_c$ appearing in this theorem depends on our normalization conventions \eqref{eq:normalization_conventions} for $J$ and $\|\cdot\|$, while the combinatorial factor $(2p-3)!!$ (the product of all odd integers between $1$ and $2p-3$) does not depend on this choice of normalization.

\begin{theorem}[High-dimensional cluster moments]
\label{thm:hd_moments_main}
 Suppose that at least one of the three hypotheses of \eqref{HD} holds.
 There exists a positive constant $A$ such that
\begin{equation}
\label{eq:hd_moments_main}
\E_{\beta_c,r} |K|^p \sim  (2p-3)!! A^{p-1} \frac{\alpha}{\beta_c} r^{(2p-1)\alpha}
\end{equation}
as $r\to\infty$ for each integer $p\geq 1$
and
\begin{equation}
\label{eq:hd_volume_tail_main}
\P_{\beta_c}(|K|\geq n) \sim \frac{\alpha}{\beta_c} \sqrt{\frac{2}{\pi A}} \cdot \frac{1}{\sqrt{n}}
\end{equation}
as $n\to \infty$.
\end{theorem}

The deduction of \eqref{eq:hd_volume_tail_main} from \eqref{eq:hd_moments_main} is carried via a Tauberian analysis in \cref{sec:volume_tail}.

\begin{remark}
The combinatorial factor $(2p-3)!!$ appearing in \eqref{eq:hd_moments_main} encodes the fact that the size-biased distribution of $|K|$ under $\P_{\beta_c,r}$, normalized to have mean $1$, converges to a chi-squared distribution with one degree of freedom (i.e., the square of a standard Gaussian) as $r\to \infty$. 
 The same chi-squared limit distribution also appears in various other contexts including slightly subcritical branching processes and slightly subcritical Erd\H{o}s-R\'enyi random graphs as explained in detail in \cite[Remark 1.14]{hutchcroft2022critical}. The factor $(2p-3)!!$ also arises combinatorially  as the number of plane trees rooted at a leaf with $p$ labelled non-root leaves and all other vertices of degree $3$.
\end{remark}

\begin{remark}
The constant $A$ appearing in \cref{thm:hd_moments_main} 
 is non-universal and depends on the small-scale details of the kernel $J$. Indeed, it arises as the exponential of a certain convergent integral of errors 
 which is determined primarily by the values of the error at small scales. 
\end{remark}

\medskip

\noindent \textbf{Scaling limits and the transition between the effectively long- and short-range regimes.}
Our next main result concerns the \emph{superprocess scaling limits} of large clusters in the high-effective-dimensional regime. In contrast to \cref{thm:hd_moments_main}, this theorem reveals a clear distinction between the effectively short-range and long-range high-dimensional regimes.
Here, \textbf{superprocesses} are universal limit objects describing the scaling limits of branching random walks, and are known or conjectured to describe the scaling limits of a large number of statistical mechanics models in high dimensions \cite{hara1998incipient,van2003convergence,cox1999rescaled,borgs1999mean,holmes2020range,cabezas2025random,holmes2008convergence,watanabe1968limit,bramson2001super,van2013survival}; see \cite{slade2002scaling} for a survey. Following Le Gall's \emph{Brownian snake} construction \cite{le1999spatial}, they may be viewed as (variants of) of Aldous's \emph{continuum random tree} \cite{CRT1,aldous1993continuum,aldous1991continuum} that are embedded into space using 
either Brownian motion or L\'evy processes as appropriate in the different regimes we consider.
 We do not assume that the reader has any prior knowledge of L\'evy processes or superprocesses and give further background on these objects as it becomes relevant throughout the course of \cref{sec:superprocesses}; comprehensive introductions can be found in e.g.\ \cite{perkins2002part,dynkin1994introduction,le1999spatial}.


 We phrase our scaling limit result in terms of the \emph{canonical measure}\footnote{In the Brownian snake framework, this is the measure in which the underlying measure on trees is constructed from the It\^o excursion measure on Brownian excursions using the continuum analogue of the bijection between plane trees and Dyck paths. In contrast, Aldous's continuum random tree is obtained by taking the Brownian excursion conditioned to run for time $1$ (and then scaling this excursion by a factor $2$), and should be relevant for scaling limits of clusters conditioned to have a fixed large volume. We warn the reader that there are several different scaling conventions for these objects used in the literature resulting in different factors of $2$ or $\sqrt{2}$ appearing in various formulae, and one must be careful about this when comparing results from different references.} $\mathbb{N}$ of the integrated superprocess excursion, which is a $\sigma$-finite measure on non-zero measures on $\R^d$ satisfying 
 \[\bbN\Bigl(\mu(\R^d)\geq \lambda\Bigr)\propto \lambda^{-1/2} <\infty\] for every $\lambda>0$. The canonical measure describes the limit of \emph{rescaled measures without conditioning}; scaling limit results for e.g.\ the law of the cluster conditioned to have size at least $n$ (where the limit is a probability measure) can easily be extracted from this more general theorem as explained in \cref{remark:scaling_limit_volume_formulation}. The word ``integrated'' appearing here means that we are considering these processes as \emph{random measures} only, without comparing e.g.\ the genealogical structure of the branching processes to the intrinsic metric structure of clusters. 
The transition between the effectively short-range and long-range regimes when $\alpha=2$ is marked both by a transition between super-Brownian and super-L\'evy scaling limits and by a change in the asymptotics of the relevant scaling functions, with logarithmic corrections to super-Brownian scaling when $\alpha=2$. 



We write $\delta_x$ for the Dirac delta measure at $x\in \R^d$.


\begin{theorem}[Superprocess limits]
\label{thm:superprocess_main}
Suppose that at least one of the three hypotheses of \eqref{HD} holds.
  There exist functions $\zeta(R)\to\infty$ and $\eta(R)\to 0$ such that
\[
\frac{1}{\eta(R)}\P_{\beta_c}\left(|K|\geq \lambda \zeta(R),\; \frac{1}{\zeta(R)}\sum_{x \in K} \delta_{x/R}\in \;\cdot\;  \right) \to \mathbb{N}\Bigl(\mu(\R^d)\geq \lambda, \mu\in \cdot\Bigr)
\]
weakly as $R\to \infty$ for each $\lambda>0$, where $\mathbb{N}$ denotes either the canonical measure of the integrated symmetric $\alpha$-stable L\'evy superprocess excursion associated to the L\'evy measure with density $\frac{\alpha}{d+\alpha}\|x\|^{-d-\alpha}$ if $\alpha <2$ or the canonical measure of integrated super-Brownian excursion with some invertible, unit-trace covariance matrix $\Sigma$ if $\alpha \geq 2$. Moreover, the normalization factors $\zeta$ and $\eta$ are given asymptotically by
\[
\zeta(R) \sim \text{const.}\;\begin{cases}
 R^{4} \\
 R^{4} (\log R)^{-2} \\
R^{2\alpha} \\
\end{cases}
\text{ and } \quad
\eta(R) \sim \text{const.}\;\begin{cases}
 R^{-2} &\; \alpha>2, d>6 \qquad\phantom{\alpha} \text{\emph{(SR HD)}}\\
 R^{-2} \log R &\; \alpha = 2, d\geq 6 \qquad \phantom{\alpha} \text{\emph{(mSR HD)}}\\
R^{-\alpha} &\; \alpha < 2, d > 3\alpha \qquad \text{\emph{(LR HD)}}\\
\end{cases}
\]
as $R\to \infty$.
\end{theorem}


We also establish in \cref{cor:scaling_limit_cut_off} a similar scaling limit theorem for the size-biased cut-off measure $\hat \E_{\beta_c,r}$, which is more relevant to future applications later in the series and which is used in the proof of \cref{thm:superprocess_main}. (Note that the cluster scale $R$ and the cut-off scale $r$ have a non-trivial relationship described by \eqref{eq:radius_of_gyration_intro}.)
A similar result for nearest-neighbour models (which was conjectured in \cite{hara1998incipient,van2003convergence} and was the subject of substantial partial progress in \cite{MR1773141,hara2000scaling}) is proven via completely different methods in forthcoming joint work with Blanc-Renaudie \cite{HutBR_superprocesses}. 

\begin{remark}
\label{remark:renormalized_covariance_matrix_symmetry}
If $\alpha>2$, the renormalized covariance matrix $\Sigma$ is, in general, non-universal and depends on the small-scale details of the kernel $J$ in a complicated way. If, however, the kernel $J$ is invariant under permutation of coordinates then $\Sigma$ must be the identity matrix, so that the scaling limit is standard super-Brownian motion. In the marginally short-range case $\alpha=2$, it follows from \cref{prop:radius_of_gyration} that $\Sigma$ is universal in the sense that it is determined explicitly by the norm $\|\cdot\|$ via
\begin{equation}
  \Sigma(i,j) = \frac{1}{\int_B \|y\|_2^2\dif y} \int_B y_iy_j \dif y
\end{equation}
where $B$ denotes the unit ball of $\|\cdot\|$ and $y_i$ denotes the $i$ coordinate of $y$.
\end{remark}

\begin{remark}[Reformulations of the scaling limit]
\label{remark:scaling_limit_volume_formulation}
If desired, \cref{thm:superprocess_main} can be rephrased equivalently as a convergence statement 
\[
\frac{1}{\eta(R)}\P_{\beta_c}\left( \frac{1}{\zeta(R)}\sum_{x \in K} \delta_{x/R}\in \;\cdot\;  \right) \to \mathbb{N}
\]
for an appropriate topology on measures on non-zero finite measures on $\R^d$ giving finite mass to the set $\{\mu:\mu(\R^d)\geq \eps\}$ for every $\eps>0$, without the need to explicitly introduce the regularization parameter~$\lambda$.
\cref{thm:superprocess_main} also implies that
\begin{equation}
\label{eq:scaling_limit_volume_version_conditional3}
\P_{\beta_c}\left( \frac{1}{\zeta(R)}\sum_{x \in K} \delta_{x/R}\in \cdot \;\Bigg|\; |K| \geq \lambda \zeta(R) \right) \to \mathbb{N}\bigl( \mu \in \cdot \;|\; \mu(\R^d) \geq \lambda \bigr)
\end{equation}
weakly as $R\to \infty$ for each fixed $\lambda>0$; this is a weaker statement than \cref{thm:superprocess_main} since it does not identify the probability of the event being conditioned on. Indeed, as the reader can check, \cref{thm:superprocess_main} implies the volume tail estimate of
\eqref{eq:hd_volume_tail_main} (which is used in its proof), and is equivalent to the conditional limit
law \eqref{eq:scaling_limit_volume_version_conditional3} and the volume-tail estimate \eqref{eq:hd_volume_tail_main} since $\N(\{\mu : \mu(\R^d) \geq \lambda\}) \propto \lambda^{-1/2}$ and $\eta(R) \sim \mathrm{const.}\; \zeta(R)^{-2}$ under the hypotheses of the theorem.
 Rewriting \eqref{eq:scaling_limit_volume_version_conditional3} in terms of  the volume rather than the length-scale, we also have that 
\begin{equation}
\label{eq:scaling_limit_volume_version_conditional2}
\P_{\beta_c}\left( \frac{1}{n}\sum_{x \in K} \delta_{x/\zeta^{-1}(n)}\in \cdot \;\Bigg|\; |K| \geq \lambda n \right) \to \mathbb{N}\bigl( \mu \in \cdot \;|\; \mu(\R^d) \geq \lambda \bigr)
\end{equation}
weakly as $n\to \infty$ for each fixed $\lambda>0$. 
\end{remark}

\begin{remark}
\label{remark:support}
One might wonder whether the cluster $K$ also converges \emph{as a set} to the support of the relevant integrated superprocess excursion. 
 When $\alpha<2$, this is a rather degenerate statement since the relevant integrated super-L\'evy excursion has support equal to the entire space $\R^d$ almost surely \cite{perkins1990polar}. On the other hand, although the integrated super-Brownian excursion has compact support a.s.\ \cite{iscoe1988supports}, it seems likely based on similar phenomena studied in \cite{MR3418547,heydenreich2014high,MR3206998} that the $K$ converges to this support as a set only under the stronger assumption that $\alpha>4$, with the cluster containing enough ``hairs'' (i.e., subgraphs that have negligible volume but macroscopic diameter) to make it converge as a set to all of $\R^d$ when $\alpha < 4$. We do not pursue this issue further in this paper, but refer the reader to \cite{holmes2020range} for a broad discussion of questions of this form including sufficient conditions for the convergence of the range in a large family of high-dimensional models.
\end{remark}

We will prove \cref{thm:superprocess_main} using 
using the
following concrete description of the moment measures of the canonical measure $\N$ in terms of the
diagrammatic expansion:
\begin{multline}
\label{eq:canonical_measure_diagrams}
  \N\left[\idotsint \prod_{i=1}^n \varphi_i(x_i) \dif \mu(x_1) \cdots \dif \mu(x_n) \right]
  \\=
  2^{2n-2}\sum_{T\in \mathbb{T}_n} \idotsint \prod_{\substack{i<j \\ i\sim j}} G(x_i-x_j) \prod_{i=1}^n \varphi_i(x_i)\dif x_1 \cdots \dif x_{2n-1},
\end{multline}
where $\mu$ denotes the ``random'' measure with ``law'' $\N$ (which is not really a probability measure),  $\mathbb{T}_n$ denotes a set of isomorphism class representatives of trees with leaves labelled $0,1,\ldots,n$ and unlabelled non-leaf vertices all of degree $3$, $G$ denotes the Green's function of the relevant spatial motion (defined formally in \cref{subsec:scaling_limits_without_cut_off}), $\varphi_1,\ldots,\varphi_n$ are non-negative compactly supported continuous functions, and we take $x_0=0$.
(For notational purposes we take our trees to have vertex set $\{1,\ldots,2n-1\}$ even though we regard the internal vertices $\{n+1,\ldots,2n-1\}$ as unlabelled; isomorphisms are not required to preserve the labels of the internal vertices.)
For the purposes of this paper the formula \eqref{eq:canonical_measure_diagrams} can essentially be taken as the \emph{definition} of $\N$, which is shown to satisfy this\footnote{In \cite[Chapter IV, Proposition 2 and Theorem 4]{le1999spatial} a slightly different version of this formula is stated where the points are restricted to lie in clockwise order on the contour of the tree starting from the root and the sum is over (unlabelled) rooted plane trees, and the constant coefficient on the right hand side is $2^{n-1}$. Removing the restriction concerning the order multiplies both sides by a factor of $n!$, while changing the type of trees summed over on the right hand side gives a $2^{n-1}/n!$ factor (to cancel the effect of adding the labels and then forgetting the planar structure), leading to the $2^{2n-2}$ prefactor written here.} formula in  \cite[Chapter IV, Proposition 2 and Theorem 4]{le1999spatial}. Indeed, when $d>\min\{4,2\alpha\}$ (as is the case in all situations we consider), Carleman's criterion implies that the canonical measure $\bbN$ is in fact \emph{characterised} by the moment formula \eqref{eq:canonical_measure_diagrams} as we explain in \cref{subsec:scaling_limits_without_cut_off} (see in particular the moment bound of \cref{lem:canonical_measure_moment_upper_bound}). (When $d\leq \min\{2,\alpha\}$ all the moments described by \eqref{eq:canonical_measure_diagrams} are infinite and the formula is no longer useful as written for describing the law of the superprocess.)

\begin{table}[t]
\centering
\renewcommand{\arraystretch}{1.6} 
\begin{tabular}{|c|c|c|}
\hline
 & \# Large Clusters & Typical size \\ 
\hline
LR HD & $R^{d-3\alpha}$ & $R^{2\alpha}$  \\ 
\hline
mSR HD & $R^{d-6}(\log R)^3$ & $R^4 (\log R)^{-2}$ \\ 
\hline
SR HD & $R^{d-6}$ & $R^4$  \\ 
\hline
LR CD & $\log R$ & $R^{2\alpha} (\log R)^{-1/2}$  \\ 
\hline
LR LD & $O(1)$ & $R^{(d+\alpha)/2}$ \\ 
\hline
\end{tabular}
\caption{Summary of results concerning the ``typical size of a large cluster'' and the ``number of typical large clusters'' in a box of radius $R$ for long-range percolation on $\Z^d$. Precise statements are given in \cref{thm:superprocess_main}, \cref{II-thm:hyperscaling} and \cref{III-cor:superprocess_main_CD}. (See also \eqref{eq:number_of_large_clusters} and \eqref{III-eq:number_of_large_clusters}.) The most important qualitative distinction is between the LR LD regime studied in the second paper of this series, where the number of large clusters is $O(1)$, and the other regimes, studied in the first and third papers, in which the number of large clusters on scale $R$ diverges as $R\to\infty$. Be careful to note the non-trivial relationship between the cluster scale $R$ and the cut-off scale $r$ outside of the effectively long-range regime as described in \eqref{eq:radius_of_gyration_intro}.}
\label{table:large_clusters}
\end{table}

\medskip

\textbf{Interpretations of the scaling factors $\zeta$ and $\eta$.}
Roughly speaking, one should interpret $\zeta(R)$ as being the ``typical size of a large cluster on scale $R$'' and $\eta(R)$ as being the probability that the origin belongs to such a ``typical large cluster". As such, the ratio
\begin{equation}
\label{eq:number_of_large_clusters}
  N(R) := \frac{R^d \eta(R)}{\zeta(R)} \sim \text{const.}\; \begin{cases} R^{d-6} & \alpha>2, d>6 \qquad\phantom{\alpha} \text{(SR HD)}\\
R^{d-6} (\log r)^3 & \alpha = 2, d\geq 6 \qquad\phantom{\alpha} \text{(mSR HD)}\\
R^{d-3\alpha} & \alpha<2, d>3\alpha \qquad \text{(LR HD)}\\
   \end{cases}
\end{equation}
can be interpreted heuristically as the \emph{number} of ``typical large clusters'' on scale $R$. See \cite{MR1431856} for similar asymptotic estimates regarding the size and number of macroscopic clusters in a box for nearest-neighbour percolation in high dimensions. The names for these quantities can be justified by considering e.g.\ the sum of $p$th powers of cluster sizes in a ball, which for $p\geq 2$ can be shown to satisfy
\begin{equation}
\label{eq:N_and_zeta_moments}
  \sum_{C \in \mathscr{C}} |C\cap B_R|^p \asymp \E_{\beta_c} \sum_{C \in \mathscr{C}} |C\cap B_R|^p \asymp N(R) \zeta(R)^p
\end{equation}
with high probability under $\P_{\beta_c}$, where $\mathscr{C}$ denotes the collection of all clusters in the percolation configuration; although the \emph{largest} cluster in the ball should be larger than $\zeta(R)$ by a factor of order $\log N(R)$ due to random fluctuations, this largest cluster is not significant to the macroscopic geometry of the configuration and it is instead the ``soup'' of roughly $N(R)$ clusters with volume of order $\zeta(R)$ that are important. Note  that we see only $(\log R)^3$ large clusters on scale $R$ in the doubly-marginal case $d=6$, $\alpha=2$ even though several other aspects of the model have exact mean-field scaling with no logarithmic corrections in this case (one example of this being the relation $\eta(R)\asymp \zeta(R)^{-1/2}$ which encodes the $n^{-1/2}$ volume tail as discussed in \cref{remark:scaling_limit_volume_formulation}). For effectively long-range, \emph{critical-dimensional} models with $d=3\alpha<6$, we show in \cref{III-cor:superprocess_main_CD} that the corresponding functions scale as
\[
\zeta(R) \sim \text{const.}\;R^{2\alpha} (\log R)^{-1/2},
\quad
\eta(R) \sim \text{const.}\; R^{-\alpha} (\log R)^{1/2}, \quad \text{ and } \quad N(R) \sim \text{const.}\; \log R
\]
as $R\to \infty$; see \Cref{table:large_clusters} for a summary.


 An important step in the proof of \cref{thm:superprocess_main} is to establish the asymptotics of the \textbf{radius of gyration} (a.k.a. the correlation length of order $2$)
\begin{equation}
\label{eq:radius_of_gyration_intro}
\xi_{2}(r):= \sqrt{\frac{\E_{\beta_c,r}\left[\sum_{x\in K} \|x\|_2^2 \right]}{\E_{\beta_c,r}|K|}} \sim 
\text{const.}\;\begin{cases}
 r^{\alpha/2} &\qquad \alpha>2, d>6  \qquad \phantom{\alpha}\text{(SR HD)}\\
 r \sqrt{\log r} &\qquad \alpha = 2, d\geq 6 \qquad \phantom{\alpha}\text{(mSR HD)}\\
r &\qquad \alpha < 2, d > 3\alpha \qquad \text{(LR HD)}\\
\end{cases}
\end{equation}
(see \cref{prop:radius_of_gyration}), 
which is a measure of the distance between the origin and a typical point
 in the cluster of the origin when this cluster has its ``typical large size'' under the measure $\P_{\beta_c,r}$ (i.e., is drawn from the size-biased distribution $\hat \P_{\beta_c,r}$); it is convenient to use the Euclidean norm $\|\cdot\|_2$ in these definitions for algebraic reasons even though our model may be defined with respect to a different norm $\|\cdot\|$. This quantity is used to relate the cluster scale $R$ appearing in e.g.\ \cref{thm:superprocess_main} with the cut-off scale $r$ used in our RG analysis.
  Note that the effectively long-range regime is characterised by the radius of gyration $\xi_{2}(r) \sim \text{const.}\,r$ having the same order as the cut-off scale $r$; this will motivate our \emph{definition} of the effectively long-range regime for models with (possibly) low effective dimension in \cref{II-def:CL}.

\begin{remark}
It is interesting to note that we prove all our results about critical exponents, scaling limits, etc.\ in the regime $d>3\alpha$ without ever proving pointwise estimates on the two-point function or that the triangle condition holds. In \cref{II-cor:high_dimensional_two_point_pointwise} we apply the results of this paper to prove that the pointwise two-point function bound
\begin{equation}
\label{eq:two_point_paper_II}
  \P_{\beta_c}(x\leftrightarrow y) \asymp \|x-y\|^{-d+\alpha}
\end{equation}
holds when $d\geq 3\alpha$ and $\alpha<2$ (the $d=3\alpha$ case relying also on the results of the third paper, namely \cref{III-thm:critical_dim_hydro}); see also Propositions \ref{III-prop:three_point_upper} and \ref{III-prop:three_point_lower} and \cref{II-thm:Fisher_relation} for corresponding up-to-constants estimates on the three-point and \emph{slightly subcritical} two-point functions respectively. The estimate \eqref{eq:two_point_paper_II} implies in particular that the triangle condition holds throughout the effectively long-range high-dimensional regime (but we do not use this fact in our proofs since it is not established until the very end of our analysis).
\end{remark}

\subsection{Conditional results in the critical dimension}
\label{subsec:scaling_limits_in_the_critical_dimension_under_the_hydrodynamic_condition}

Although the primary focus of this paper is on the high-effective-dimensional regime, our methods also apply \emph{at the critical dimension} subject to a technical ``marginal triviality'' assumption we call the \emph{hydrodynamic condition} that is verified for $d=3\alpha$ in the third paper in this series. In this section we introduce this condition and state our results on scaling limits and the cluster volume distribution in the critical dimension that hold subject to this condition.
The proofs of these theorems are presented in this paper, rather than deferred to after the hydrodynamic condition is proven in the third paper, since they overlap heavily with the proofs of \cref{thm:hd_moments_main,thm:superprocess_main} and it is convenient to carry out the two proofs in parallel to one another.


We now introduce the hydrodynamic condition, which was previously introduced for the hierarchical model in \cite{hutchcroft2022critical}.
For each $r\geq 1$, we define the quantity
\[
  M_r =\min\left\{n\geq 0: \P_{\beta_c,r}\left(\max_{x\in \Z^d} |K_x \cap B_r| \geq n\right)\leq e^{-1}\right\},
\]
which we think of as the typical size of the largest cluster in the box $B_r$ under the cut-off measure $\P_{\beta_c,r}$. The fact that this is a \emph{useful} quantity to study owes largely to the \emph{universal tightness theorem} of \cite{hutchcroft2020power} as explained in \cref{subsec:correlation_inequalities_and_the_universal_tightness_theorem}.
As we explain in detail in \cref{sub:previous_results_on_long_range_percolation}, it follows from the results of \cite{hutchcroft2022sharp} that $M_r$ always satisfies
 \[M_r=O(r^{(d+\alpha)/2})\] as $r\to \infty$ when $0<\alpha<d$. The geometric significance of this bound is that if $A_1$ and $A_2$ are disjoint subsets of the ball $B_r$ then the expected number of edges between $A_1$ and $A_2$ that are open in the configuration associated to the measure $\P_{\beta_c,2r}$ but not $\P_{\beta_c,r}$ (in the standard monotone coupling) is of order $|A_1||A_2|r^{-d-\alpha}$, so that the probability such an open edge exists is high when $|A_1|,|A_2| \gg r^{(d+\alpha)/2}$, low when $|A_1|,|A_2| \ll r^{(d+\alpha)/2}$, and bounded away from $0$ and $1$ when $|A_1|$ and $|A_2|$ are both of order $r^{(d+\alpha)/2}$. Thus, the bound $M_r =O(r^{(d+\alpha)/2})$ can be interpreted as stating that the largest clusters in two adjacent balls of radius $r$ cannot have a \emph{high} probability of merging via the direct addition of a single edge when we pass from scale $r$ to scale $2r$. As we shall see throughout this series, the qualitative distinction between the effectively low- and high-dimensional long-range regimes can be understood in terms of whether or not this bound is sharp: In low dimensions large clusters have a good probability to directly merge with each other on each scale, which leads to non-tree-like geometry and the validity of the so-called \emph{hyperscaling relations}, while in high-dimensions different clusters interact with each other only very weakly and mean-field approximations such as \eqref{eq:ODE_first_moment_intro} and \eqref{eq:x12_derivative_asymptotic2_intro} are valid. (This is closely related to whether the ``number of typical large clusters'' $N(r)$ considered in \eqref{eq:number_of_large_clusters} diverges or remains bounded as $r\to \infty$.) 



\begin{defn}[Hydro]
\label{def:Hydro}
 We say that the \textbf{hydrodynamic condition} \mytag[Hydro]{Hydro}\!\!\! holds if \[M_r=o(r^{(d+\alpha)/2})\] as $r\to\infty$. 
\end{defn}

This terminology was inspired by the notion of \emph{hydrodynamic limits}, where trajectories of certain Markov chains (e.g.\ interacting particle systems) converge to deterministic dynamical systems described by ODEs or PDEs \cite{demasi2006mathematical}. In our context, the hydrodynamic condition implies roughly that the ensemble of all clusters in a large box has approximately deterministic geometry in various senses; see \cite[Section 4]{hutchcroft2022critical} for further discussion in the hierarchical setting.

\medskip



The fact that the hydrodynamic condition holds when $d>3\alpha$ will be established as a corollary to the main results of this paper (which proceed via a different approach that does not invoke this condition directly); see \cref{cor:HD_hydro}. The fact that it also holds when $d=3\alpha$ is much more difficult to prove and is established in the third paper in this series (\cref{III-thm:critical_dim_hydro}). Finally, the fact that the hydrodynamic condition does \emph{not} hold in the effectively long-range, low-dimensional regime (LR LD in \cref{fig:cartoon}) is established in the second paper of the series (\cref{II-prop:low_dim_not_hydro,II-thm:max_cluster_size_LD}) and plays an important role in our analysis of that regime.

\medskip

\noindent
\textbf{Scaling limits in the critical dimension under the hydrodynamic condition.}
We now state our main results describing critical behaviour at the upper critical dimension subject to the hydrodynamic condition.
%
 In both of the following two theorems, the large scale asymptotics of the model are determined in terms of certain slowly varying 
  corrections to scaling (see in \cref{subsec:regular_variation_and_logarithmic_integrability} for the definition of slow variation) that we do not determine in this paper but are computed for $d=3\alpha<6$ in the third paper in this series (\cref{III-thm:critical_dim_moments_main}).

\begin{theorem}[Cluster volumes in the critical dimension]
\label{thm:critical_dim_moments_main_slowly_varying}
Suppose that $d=3\alpha$ and the hydrodynamic condition holds. There exist bounded, slowly varying functions $A_r$ and $A_n^*$ with $r\mapsto A_r r^{2\alpha}$ and $n\mapsto (A^*_n)^{-1/2\alpha} n^{1/2\alpha}$ mutually inverse such that
\begin{equation}
\label{eq:critical_dim_moments_main}
\E_{\beta_c,r} |K|^p \sim  (2p-3)!! A^{p-1}_r \frac{\alpha}{\beta_c} r^{(2p-1)\alpha}
\end{equation}
as $r\to\infty$ for each integer $p\geq 1$
and
\begin{equation}
\label{eq:critical_dim_volume_tail_main}
\P_{\beta_c}(|K|\geq n) \sim \frac{\alpha}{\beta_c} \sqrt{\frac{2}{\pi A^*_n}} \cdot \frac{1}{\sqrt{n}}
\end{equation}
as $n\to \infty$. 
\end{theorem}

We will see that the asymptotic formula \eqref{eq:critical_dim_moments_main} is an analytic consequence of the asymptotic ODE \eqref{eq:volume_moments_ODE_intro}. Indeed, the two estimates \eqref{eq:critical_dim_moments_main} and \eqref{eq:volume_moments_ODE_intro} are justified in parallel over the course of \cref{sec:analysis_of_moments}. As in the high-dimensional case, the volume tail asymptotics of \eqref{eq:critical_dim_volume_tail_main} will be deduced from the moment asymptotics of \eqref{eq:critical_dim_moments_main} with the aid of Karamata's Tauberian theorem in \cref{sec:volume_tail}.

\begin{remark}
When $d=3\alpha \geq 6$ and the model satisfies a two-point function estimate as verified for spread-out models via the lace expansion \cite{MR3306002,MR4032873}, it follows from \cref{thm:hd_moments_main} that $A_r$ and $A^*_n$ converge to positive constants as $r,n\to\infty$. On the other hand, for $d=3\alpha<6$ we show in \cref{III-thm:critical_dim_moments_main} that $A_r \sim A (\log r)^{-1/2}$ 
and $A_n^*\sim (2\alpha)^{1/2} A (\log n)^{1/2}$ 
for some constant $A>0$. 
%
%
\end{remark}


For our next theorem we restrict attention to the case $d=3\alpha<6$ where the model is effectively long-range and critical dimensional. (Note that the case $d=3\alpha\geq 6$ was already treated in \cref{thm:superprocess_main} for spread-out models satisfying the conclusions of the lace expansion \cite{MR2430773,MR3306002,MR4032873}.) We drop the distinction between the cluster scale $R$ and the cut-off scale $r$ from this theorem since the two are equivalent for the effectively long-range models we treat here.

\begin{theorem}[Superprocess limits in the critical dimension]
\label{thm:superprocess_main_regularly_varying}
 Suppose that $d=3\alpha<6$ and the hydrodynamic condition holds.
  If we define 
 \[
  \eta(r) = 4 \cdot \frac{(\E_{\beta_c,r}|K|)^2}{\E_{\beta_c,r}|K|^2}
   \qquad \text{ and } \qquad \zeta(r) = \frac{1}{4} \cdot \frac{\E_{\beta_c,r}|K|^2}{\E_{\beta_c,r}|K|}
\]
then $\zeta(r)$ and $\eta(r)$ are regularly varying of index $2\alpha$ and $-\alpha$ respectively and
\[
\frac{1}{\eta(r)}\P_{\beta_c}\left(|K|\geq \lambda \zeta(r),\; \frac{1}{\zeta(r)}\sum_{x \in K} \delta_{x/r}\in \;\cdot\;\right) \to \mathbb{N}\Bigl(\mu(\R^d)\geq \lambda, \mu \in \;\cdot\;\Bigr)
\]
weakly as $r\to\infty$ for each fixed $\lambda>0$, where $\mathbb{N}$ denotes the canonical measure of the integrated  superprocess whose spatial motion is the symmetric $\alpha$-stable L\'evy jump process whose L\'evy measure has density $\frac{\alpha}{\alpha+d}\|x\|^{-d-\alpha}$. 

\end{theorem}

As in \cref{remark:scaling_limit_volume_formulation}, the canonical measure formulation of the scaling limit presented in \cref{thm:superprocess_main} also implies that
\begin{equation}
\P_{\beta_c}\left( \frac{1}{n}\sum_{x \in K} \delta_{x/\zeta^{-1}(n)}\in \cdot \;\Bigg|\; |K| \geq n \right) \to \mathbb{N}\bigl( \mu \in \cdot \;|\; \mu(\R^d) \geq 1 \bigr)
\end{equation}
weakly as $n\to \infty$. See also \cref{cor:scaling_limit_cut_off} for a related scaling limit theorem for the size-biased cut-off measures $\hat \P_{\beta_c,r}$.

\begin{remark}
It is an interesting feature of our proof strategy that we must establish the full superprocess scaling limit of the model when $d=3\alpha$, with undetermined slowly varying corrections to scaling, \emph{before} computing the asymptotics of these corrections to scaling. (In particular, the full superprocess scaling limit theorems are needed even to compute the corrections to scaling of \cref{thm:critical_dim_moments_main_slowly_varying} that pertain only to the volume distribution.)
Roughly speaking, this strategy is viable because \cref{thm:critical_dim_moments_main_slowly_varying,thm:superprocess_main_regularly_varying} can be deduced from \emph{first-order} asymptotics of derivatives (of the form described in \eqref{eq:volume_moments_ODE_intro} and \eqref{eq:x12_derivative_asymptotic2_intro}), while computing logarithmic corrections to scaling requires analyzing the same derivatives to \emph{second order}. In \cref{III-sec:logarithmic_corrections_at_the_critical_dimension} we compute the asymptotics of these second-order corrections using a method that relies on already knowing the scaling limit of the model. More concretely, we obtain (modulo some additional logarithmically integrable errors; see \cref{III-prop:critical_dim_second_order}) that if $d=3\alpha<6$ then there exists a positive constant $C$, determined by the superprocess scaling limit of the model, such that
\begin{equation}
\label{eq:log_correction_paper1}
  \frac{d}{dr} \log \E_{\beta_c,r}|K|^2 = \frac{3\alpha}{r} - \frac{C}{r} \frac{(\E_{\beta_c,r}|K|^2)^2}{r^{d}(\E_{\beta_c,r}|K|)^3} \pm \cdots
%
\end{equation}
as $r\to \infty$. (As can be seen from this estimate, it is perhaps more accurate to say that we find an asymptotic relationship between the second-order correction and the quantity we are differentiating, rather than ``computing'' this second-order correction directly.) In light of the simple asymptotic formula $\E_{\beta_c,r}|K|\sim \frac{\alpha}{\beta_c}r^\alpha$ for the first moment, the estimate \eqref{eq:log_correction_paper1} implies by elementary analysis that $A_r \sim A (\log r)^{-1/2}$ as $r\to \infty$ for some positive constant $A$ (see \cref{III-lem:ODE_log_correction}). 
\end{remark}

\subsection{Relation to other approaches to the renormalization group}
\label{subsec:RG}

We now give a brief overview of how our methods differ from other rigorous renormalization group (RG) approaches to the critical behaviour of lattice models\footnote{For an informal but informative overview of RG, we highly recommend Roland Bauerschmidt's lecture at the meeting ``100 years of the Ising Model'', available at \url{https://www.youtube.com/watch?v=l1iZbsEiO6c}}. We stress that we do not assume the reader has any previous familiarity with such approaches, and that the rest of this section can be safely ignored without compromising the intelligibility of the rest of the paper.

\begin{remark}[Relation with other approaches to RG]
While the methods of this paper, particularly those used to treat the regime $d_\mathrm{eff}\geq d_c=6$, can be thought of as coming under the umbrella of RG methods, they differ from other rigorous approaches to RG in various important ways. First, these methods are typically applied to spin systems such as the $\varphi^4$ model that can be constructed\footnote{Some other models, such as continuous time weakly self-avoiding walk (WSAW), can be treated within this framework via equivalences with models of this form. In the case of WSAW the equivalence is with a \emph{supersymmetric} variant of the $\varphi^4$ model that corresponds formally the $\varphi^4$ model with $0$-dimensional spins \cite[Chapter 11]{MR3969983} and leads one outside the domain of probability theory. See \cite{swan2021superprobability} for a broader overview of supersymmetric methods in probability theory.
In the physics literature one often studies percolation via (non-rigorous) relationships to a putative ``$0$-component $\varphi^3$ model'' or ``$1$-state Potts model'' \cite{essam1978percolation,ruiz1998logarithmic,de1981critical}, but these have relationships have yet to see precise interpretations in terms of e.g.\ an isomorphism theorem relating percolation to a supersymmetric $\varphi^3$ model.
} by adding an interaction to the law of the Gaussian free field, i.e., the Gaussian field whose covariances are given by the Green's function $G=\Delta^{-1}$, where the Laplacian is defined by $\Delta=D-J$ with $D$ the diagonal matrix $D(x,x)=\sum_y J(x,y)$. One of the most successful rigorous versions of this method is that of Bauerschmidt, Brydges, and Slade \cite{MR3339164,bauerschmidt2015critical,bauerschmidt2017finite} (see also e.g.\ \cite{MR3772040,MR3723429}), which works by replacing the Gaussian free field with a finitely-dependent Gaussian field with covariances $G_r$, where $G_r$ is a well-chosen ``cut-off'' correlation matrix that converges to $G$ as $r\to \infty$, and computing the asymptotics of various observables as $r\to \infty$. Thus, one can think this approach as closely analogous (or even ``dual'') to our method, which applies a cut-off to the Laplacian rather than its inverse. 
This is related in part to the distinction in the physics literature between \textbf{real-space RG} (often attributed to Kadanoff) and \textbf{momentum-space RG} (different flavours of which are attributed to Wilson and Polchinski); let us stress however that our method is distinct from more standard ``block-spin'' versions of real-space RG as have been applied to the rigorous study of critical phenomena in e.g. \cite{bleher2010critical,MR709462,abdesselam2013rigorous,MR1882398,gawedzki1985massless,hara1987rigorousa,hara1987rigorousb}.

There are advantages and disadvantages to both approaches. An obvious major advantage of the inverse-Laplacian-based methods over ours is that it applies also to short-range models. The main advantages of our method are:
\begin{enumerate}
  \item It can be applied to models such as percolation which are not known to have exact representations in terms of a self-interacting Gaussian free field.
  \item The intermediate models we obtain by applying a cut-off to the kernel $J$ are of the same form as the model we start with: in particular, when studying percolation, the cut-off measures $\P_{\beta_c,r}$ all describe Bernoulli bond percolation on an appropriately chosen transitive weighted graph. This is in stark contrast to methods based on applying a cut-off to the \emph{inverse} Laplacian, where if one starts with, say, a $\varphi^4$ model then the intermediate models obtained by applying a cut-off to the Greens function are \emph{not} of the same form\footnote{More precisely, the intermediate models one obtains in this framework can be defined with respect to interaction kernels (obtained by inverting the cut-off Greens function) that are positive definite but may fail to be positive \emph{pointwise}. For $\varphi^4$-type models this means that the intermediate models are not completely ferromagnetic and one loses access to various relevant correlation inequalities; for percolation it is unclear how to interpret an intermediate model of this form at all since some edges would have negative inclusion probability.}. As such, in our approach\footnote{Another approach with this feature is simply to consider $\beta_c-\beta$ as a kind of indirect scale parameter and study the asymptotics of observables as $\beta \uparrow \beta_c$. For the $\varphi^4$ model this is equivalent to introducing a mass to the GFF and is known in theoretical physics as Pauli-Villars regularization \cite{pauli1949invariant}. While this preserves access to correlation inequalities, it has much worse analytic properties than our approach as discussed in \cref{subsec:definitions}.} one keeps access to all the correlation inequalities and other theorems that are available for the original model (percolation in this paper) and can argue in a more probabilistic manner, while in the inverse-Laplacian approach one typically forgoes the use of such features and works in a more analytic way. (On the plus side, this means that the inverse-Laplacian methods can be applied without significant further effort to models such as the spin $O(n)$ model with $n\geq 3$ which are not known to obey any non-trivial correlation inequalities.)
  \item Our approach is able to prove theorems applying non-perturbatively, without any small parameter assumptions, in contrast to the RG framework of \cite{MR3339164,bauerschmidt2015critical,bauerschmidt2017finite}, the traditional approach to high-dimensional percolation via the lace expansion \cite{MR1043524,MR2430773}, and the recently introduced lace expansion alternative of \cite{duminil2024alternative,duminil2024alternative}. The fact that we have access to correlation inequalities throughout our analysis is very important for this. (On the other hand, one can think of the perturbative analysis in the inverse-Laplacian approach as giving stronger \emph{universality} results than our method, as it can be used to establish results concerning models that are small perturbations of e.g.\ the $\varphi^4$ model in a fairly arbitrary sense, see e.g.\ \cite{bauerschmidt2017four}.)
\end{enumerate}
Finally, perhaps the most remarkable feature of our method is that it can be used to analyze the \emph{low}-effective-dimension regime in a non-perturbative way, as explained in the second paper in this series \cite{LRPpaper2}. As far as we are aware these results have no counterpart in the other parts of the RG literature (where low dimensions are typically treated by expanding perturbatively around the critical dimension \cite{wilson1972critical,MR3772040,MR3723429,abdesselam2013rigorous,MR709462,bleher2010critical}), and indeed mark the largest departure of our approach from the standard RG framework. 
\end{remark}

\section{Background}
\label{sec:background}

In this section we briefly overview the relevant background that we will need for the rest of the paper, including precise statements of the results of the papers \cite{hutchcroft2020power,hutchcroft2022sharp,baumler2022isoperimetric,hutchcroft2024pointwise} that will be used in our analysis.

\subsection{Correlation inequalities}
\label{subsec:correlation_inequalities_and_the_universal_tightness_theorem}

We now overview the main correlation inequalities that will be used throughout the series. While much of this section will be familiar to experts, the inequalities discussed here also include some more recent additions to the percolation literature including the universal tightness theorem \cite{hutchcroft2020power} and the Gladkov inequality \cite{gladkov2024percolation}. We will also take the opportunity to write down some simple consequences of the basic correlation inequalities that will be used repeatedly throughout the paper.

\medskip

\noindent \textbf{The Harris-FGK and BK inequalities.} The most basic correlation inequality for percolation is the Harris-FKG inequality \cite[Section 2.2]{grimmett2010percolation}, which states that if $F(\omega)$ and $G(\omega)$ are increasing functions of the percolation configuration then they are non-negatively correlated:
\[
  \E F(\omega)G(\omega) \geq \E F(\omega) \E G(\omega).
\]
In particular, if $A$ and $B$ are two increasing \emph{events} (meaning that their indicator functions are increasing functions of the configuration) then
\[
  \P(A \cap B) \geq \P(A)\P(B).
\]
We next introduce the BK (van den Berg and Kesten) inequality \cite[Section 2.3]{grimmett2010percolation}, which is complementary to Harris-FKG and captures the sense in which Bernoulli bond percolation is \emph{negatively} dependent.
Given an event $A$ and a configuration $\omega \in A$, we say that a set of edges $W$ is a \textbf{witness} for $\omega \in A$ if every configuration $\omega'$ with $\omega'|_W=\omega|_W$ also belongs to $A$. Given two events $A$ and $B$, the \textbf{disjoint occurence} $A\circ B$ of $A$ and $B$ is defined by
\[A\circ B = \{\omega: \omega\in A\cap B \text{ and there exist disjoint witnesses for $\omega\in A$ and $\omega\in B$}\}.\]
The \textbf{BK inequality} states that if $A$ and $B$ are increasing events depending on only finitely many edges then
\[
  \P(A \circ B) \leq \P(A)\P(B).
\]
In fact the same inequality holds \emph{without} the assumption that $A$ and $B$ are increasing, a fact known as \emph{Reimer's inequality} \cite{MR1751301}. (The special case of Reimer's inequality in which each of $A$ and $B$ can be written as the intersection of an increasing set and a decreasing set is significantly easier to prove, and is sometimes known as the van den Berg--Kesten--Fiebig inequality \cite{MR877608}; this special case will suffice for all our applications.) It is usually possible to apply the BK inequality or Reimer's inequality to events depending on infinitely many edges via simple limiting arguments\footnote{In particular, both inequalities are always valid for the set $A\odot B$ of configurations in which there exist \emph{finite} disjoint witnesses for $A$ and $B$ \cite[Section 7]{arratia2015van}, which coincides with the set $A\circ B$ in most use cases. In fact the BK inequality is always valid for any increasing Borel sets $A$ and $B$ as we will prove in our forthcoming textbook on critical percolation.}, and we will do this without further comment when no subtleties arise.

\medskip

We now take the opportunity to make note of 
the following
useful consequence
 of the Harris-FKG and BK inequalities that will be used repeatedly throughout our analysis.



\begin{lemma}[Two-cluster correlation inequalities]
Consider Bernoulli bond percolation on a weighted graph $G=(V,E,J)$ with some $\beta\geq 0$. The inequality
\label{lem:BK_disjoint_clusters_covariance}
\[
\E_\beta\left[F(K_x)G(K_y)\mathbbm{1}(x\nleftrightarrow y)\right] \leq \E_\beta \left[F(K_x)\right] \E_\beta \left[G(K_y)\right]
\]
holds for every $x,y\in V$ and every pair of increasing non-negative functions $F$ and $G$. Moreover, the inequalities 
\[
\E_\beta\left[F(K_x)G(K_y)\mathbbm{1}(x\nleftrightarrow y)\right] \geq \E_\beta \left[F(K_x)\right] \E_\beta \left[G(K_y)\right] - \E_\beta\left[F(K_x)G(K_y)\mathbbm{1}(x\leftrightarrow y)\right]
\]
and
\[
  0\leq \E_\beta [F(K_x)G(K_y)] -  \E_\beta \left[F(K_x)\right] \E_\beta \left[G(K_y)\right] \leq \E_\beta\left[F(K_x)G(K_y)\mathbbm{1}(x\leftrightarrow y)\right]
\]
hold for every $x,y\in V$ and every pair of increasing non-negative functions $F$ and $G$ such that $\E_\beta F(K_x),\E_\beta G(K_y)<\infty$.
\end{lemma}

\begin{proof}[Proof of \cref{lem:disjoint_connections}]
The first inequality is an immediate consequence of the BK inequality. The second inequality follows follows by writing
$\mathbbm{1}(x\nleftrightarrow y)=1-\mathbbm{1}(x\leftrightarrow y)$ and using Harris-FKG to bound the resulting term $\E_\beta[F(K_x)G(K_y)] \geq \E_\beta F(K_x) \E_\beta G(K_y)$. The final inequality follows from the first two using again that $1=\mathbbm{1}(x\leftrightarrow y)+\mathbbm{1}(x\nleftrightarrow y)$.
\end{proof}

The following lemma is a special case of \cref{lem:BK_disjoint_clusters_covariance} and generalizes \cite[Lemma 4]{MR1431856}.

\begin{lemma}
\label{lem:disjoint_connections} 
Consider Bernoulli bond percolation on a weighted graph $G=(V,E,J)$ with some $\beta\geq 0$. The estimate
\begin{equation*}
0 \leq\P_\beta(x\leftrightarrow y)\P_\beta(z \leftrightarrow w) - \P_\beta(x\leftrightarrow y \nleftrightarrow z \leftrightarrow w)   \leq \P_\beta(x\leftrightarrow y \leftrightarrow z \leftrightarrow w)
\end{equation*}
holds for every $x,y,z,w\in V$.
\end{lemma}


\noindent \textbf{The tree-graph inequalities.} One of the most important consequences\footnote{The original proof of the tree-graph inequalities did not rely on the BK inequality, which was proven later.} of the BK inequality are the so-called \textbf{tree-graph inequalities} of Aizenman and Newman \cite{MR762034}, which give upper bounds on $k$-point connectivity functions in terms of sums of products of two-point connection probabilities arranged in the structure of a labelled binary tree. The first of these inequalities is
\begin{equation}
\label{eq:tree_graph_3point}
  \P_\beta(x\leftrightarrow y \leftrightarrow z) \leq \sum_{w\in V} \P_\beta(x \leftrightarrow w)\P_\beta(w \leftrightarrow y)\P_\beta(w \leftrightarrow z),
\end{equation}
which holds for Bernoulli bond percolation on any weighted graph and for any choice of $x,y,z \in V$. 
For higher $(k+1)$-point functions, the corresponding bound is
\begin{equation}
\label{tree_graph_general}
\P_{\beta}(x_0,x_1,\ldots,x_k \text{ all connected}) \leq 
\sum_{T\in \bbT_k} \sum_{x_{k+1},\ldots,x_{2k-1}} \prod_{\substack{i<j \\ i\sim j}} \P_\beta(x_i \leftrightarrow x_j),
\end{equation}
 where the sum is taken over (representatives of) isomorphism classes of trees with leaves labelled $0,1,\ldots,k$ and unlabelled non-leaf vertices all of degree $3$ as in \eqref{eq:canonical_measure_diagrams}. 
These inequalities lead to relatively simple upper bounds on cluster moments for transitive weighted graphs, namely
\begin{equation}
\label{eq:tree_graph_pth_moment}
  \E_\beta|K|^p \leq (2p-3)!! (\E_\beta |K|)^{2p-1}
\end{equation}
for each integer $p\geq 1$, where the combinatorial factor $(2p-3)!!$ counts the number of trees with $p+1$ leaves labelled $\{0,\ldots,p\}$ and with all other vertices unlabelled and of degree $3$. (This combinatorial factor is the same appearing in \cref{thm:hd_moments_main,III-thm:critical_dim_moments_main}.)
Theorems \ref{thm:hd_moments_main}, \ref{II-thm:main_low_dim}, and \ref{III-thm:critical_dim_moments_main} show that the tree-graph bounds are of optimal order in high effective dimensions but not in low effective dimensions (including when $d_\mathrm{eff}=d_c=6$ and $d<6$, in which case they are suboptimal by a polylogarithmic factor); see \cref{III-sub:relations_between_the_beta_derivative_and_the_second_moment} for further discussion.

\medskip

\noindent \textbf{The universal tightness theorem.} We now state the universal tightness theorem of \cite[Theorem 2.2]{hutchcroft2020power}. Let $G=(V,E,J)$ be a weighted graph and let $\Lambda \subseteq V$ be finite and non-empty. Given a percolation configuration $\omega$ on $G$, consider the random partition of $\Lambda$ given by $\sC=\sC(\omega,\Lambda)=\{C \cap \Lambda : C$ a cluster of $\omega\}$, where the clusters of $\omega$ are computed with respect to the entire graph $G$ and may involve connections leaving the distinguished set $\Lambda$.
Let $|K_\mathrm{max}(\Lambda)|$ be the random variable defined by
\[
|K_\mathrm{max}(\Lambda)|= \max\{|K_x \cap \Lambda| : x\in V\}=\max\{|K_x \cap \Lambda| : x\in \Lambda\} 
\]
  and for each $\beta \geq 0$ define the \textbf{edian}\footnote{This pun on ``median'' was suggested to me independently by Michael Aizenman and Louigi Addario-Berry.} $M_\beta(\Lambda):=\min \{n \geq 0 : \P_\beta(|K_\mathrm{max}(\Lambda)| \geq n)\leq e^{-1} \}$. The universal tightness theorem, which is also a consequence of the BK inequality, states that $|K_\mathrm{max}(\Lambda)|$ is unlikely to deviate from its edian by a large constant factor, and moreover that the distribution of the intersection of the cluster of a specific vertex with $\Lambda$ has an exponential tail above the edian size of the largest cluster.

\begin{theorem}[Universal tightness of the maximum cluster size] 
\label{thm:universal_tightness}
Let $G=(V,E,J)$ be a weighted graph and let $\Lambda \subseteq V$ be finite and non-empty. Then the inequalities
\begin{align}
\P_\beta\Bigl(|K_\mathrm{max}(\Lambda)| \geq \lambda M_\beta(\Lambda)\Bigr) &\leq \exp\left(-\frac{1}{9}\lambda \right)
\label{eq:BigClusterUnrooted}
\\
\text{and} \qquad \P_\beta\Bigl(|K_\mathrm{max}(\Lambda)| < \eps M_\beta(\Lambda) \Bigr) &\leq 27 \eps \qquad \phantom{\text{and}} \qquad
\label{eq:SmallMaximum}
\end{align}
hold for every $\beta\geq 0$, $\lambda \geq 1$, and $0<\eps \leq 1$. Moreover, the inequality
\begin{equation}
\label{eq:BigClusterRooted}
\P_\beta\Bigl(|K_u \cap \Lambda| \geq \lambda M_\beta(\Lambda)\Bigr) \leq e \cdot \P_\beta\Bigl(|K_u \cap \Lambda| \geq  M_\beta(\Lambda)\Bigr) \exp\left(-\frac{1}{9}\lambda \right)
\end{equation}
holds for every $\beta \geq 0$, $\lambda \geq 1$, and $u \in V$.
\end{theorem}

\begin{remark}
In \cite[Section 3]{hutchcroft2022critical} we established generalizations of the universal tightneness theorem from the $\ell^\infty$ norm of the partition of $\Lambda$ into clusters to the $\ell^p$ norm. These generalizations will not be used in this series of papers.
\end{remark}

Returning to our primary setting of long-range percolation on $\Z^d$, the most important feature of the universal tightness theorem is that it lets us bound the ratio of moments of $|K\cap B_r|$ in terms of the quantity $M_{\beta,r}$, which is defined to be the edian of $|K_\mathrm{max}(B_r)|$ under the measure $\P_{\beta,r}$.
We record this application in the following general corollary of \cref{thm:universal_tightness}.

\begin{corollary}
\label{cor:universal_tightness_moments}
Let $G=(V,E,J)$ be a weighted graph.
The estimate 
\[\E_{\beta}|K_x \cap W|^{p+q} \preceq_{p,q} M_{\beta}(W)^q\E_{\beta}|K_x \cap W|^p\] holds for every $q, p \geq 0$, $\beta\geq 0$, $x\in V$,  and $W\subseteq V$ finite.
\end{corollary}


\begin{proof}[Proof of \cref{cor:universal_tightness_moments}]
Write $M=M_\beta(W)$ and $K=K_x$. We can use integration by parts and the inequality \eqref{eq:BigClusterRooted} to bound
\begin{align*}
\E_{\beta}|K \cap W|^{p+q} &\asymp_{p,q} \int_0^\infty t^{p+q-1} \P(|K\cap W| \geq t) \dif t\\
&\leq  \int_0^M t^{p+q-1} \P(|K\cap W| \geq t) \dif t +  \P(|K\cap W| \geq M) \int_M^\infty t^{p+q-1} \exp\left[1- \frac{t}{9M}\right] \dif t.
\end{align*}
The first term is bounded by $M^q \int_0^\infty t^{p-1} \P(|K\cap W| \geq t) \dif t \asymp M^q \E_{\beta}|K \cap W|^{p}$. For the second term, we can use the change of variables $s=t/9M$ to compute
\[
  \int_M^\infty t^{p+q-1} \exp\left[1- \frac{t}{9M}\right] \dif t = \int_{1/9}^\infty (9Ms)^{p+q-1} \exp\left[1- s\right] 9M \dif s \asymp_{p,q} M^{p+q}
\]
so that
\[
   \P(|K\cap W| \geq M) \int_M^\infty t^{p+q-1} \exp\left[1- \frac{t}{9M}\right] \dif t \preceq_{p,q} \frac{\E_{\beta}|K \cap W|^{p}}{M^p} M^{p+q} \asymp_{p,q} M^q \E_{\beta}|K \cap W|^{p}
\]
as required.
\end{proof}

As in \cref{subsec:scaling_limits_in_the_critical_dimension_under_the_hydrodynamic_condition}, we also write $M_r=M_{\beta_c,r}$ for the quantity used to define the hydrodynamic condition. This quantity will be very important throughout our analysis in the regime $d_\mathrm{eff}\leq 6$.  (It is less important in high effective dimension where one can usually accomplish the same things more easily using the tree-graph inequalities.)
The following simple lower bound on cluster moments in terms of $M_r$ complements the upper bound of \cref{cor:universal_tightness_moments}. 

\begin{lemma}
\label{lem:moments_bounded_below_by_M}
The inequality
\[
\E_{\beta,r}|K|^p \geq \E_{\beta,r}|K \cap B_{2r}|^p \geq (M_{\beta,r}-1)^p \P_{\beta,r} \left(|K\cap B_{2r}|\geq M_{\beta,r}-1\right) \geq \frac{(M_{\beta,r}-1)^{p+1}}{e |B_r|}
\]
holds for every $\beta,r\geq 0$ and $p\geq 1$.
\end{lemma}

\begin{proof}[Proof of \cref{lem:moments_bounded_below_by_M}]
The first two inequalities are trivial; we focus on the third. 
Let $B_r(x)$ denote the set $\{y\in \Z^d:\|x-y\|\leq r\}$ for each $x\in \Z^d$ and $r\geq 0$. If there exists a cluster $C$ having intersection of size at least $M_{\beta,r}-1$ with $B_r$ then every $x\in B_r$ belonging to this cluster has $|K \cap B_{2r}(x)| \geq M_{\beta,r}-1$. Since such a cluster exists with probability at least $1/e$ by definition of $M_{\beta,r}$, it follows by transitivity that
\[
  \P_{\beta,r}(|K \cap B_{2r}| \geq M_{\beta,r}-1) = \frac{1}{|B_r|}\E_{\beta,r} \sum_{x \in B_r} \mathbbm{1}(|K \cap B_{2r}(x)| \geq M_{\beta,r}-1) \geq \frac{M_\beta-1}{e |B_r|}
\]
as claimed.
\end{proof}

\medskip

\noindent \textbf{The Gladkov inequality.}
The last correlation inequality we make note of was established in the recent work of 
\cite{gladkov2024percolation}, and is once again a (non-obvious) consequence of the BK inequality.

\begin{thm}
\label{thm:Gladkov}
Let $G=(V,E,J)$ be a weighted graph and consider Bernoulli bond percolation on $G$. The three-point connectivity function satisfies
\[
  \P_\beta(x \leftrightarrow y \leftrightarrow z)^2 \leq 8 \P_\beta(x \leftrightarrow y)\P_\beta(y \leftrightarrow z)\P_\beta(z\leftrightarrow x)
\]
for every $\beta \geq 0$ and $x,y,z\in V$.
\end{thm}

An extension of this inequality to higher $k$-point functions is given in \cref{II-thm:higher_Gladkov}. While the Gladkov inequality is obviously useful for estimating $k$-point functions in low effective dimension, where it is sharp at criticality, it will also be useful for proving that certain error terms are small in the critical case $d=3\alpha$. (See e.g.\ the proof of \cref{lem:gyration_derivative}, Case 2.) When applied in this context, the Gladkov inequality can be used to derive bounds that are similar to, but slightly different than, those obtained from the universal tightness theorem; the following corollary is a protypical example of a bound that can be obtained from Gladkov in this way.

\begin{corollary}
\label{cor:Gladkov_moments}
Let $G=(V,E,J)$ be a transitive weighted graph and consider Bernoulli bond percolation on $G$. The inequality
\[
  \E_\beta\left[|K_x| |K_x \cap W|\right] \leq 2^{3/2} \E_\beta |K| \sqrt{ |W| \E_\beta |K_x \cap W|}
\]
holds for every $\beta \geq 0$, $x\in V$, and $W \subseteq V$ finite.
\end{corollary}

\begin{proof}[Proof of \cref{cor:Gladkov_moments}]
We can use the Gladkov inequality to write
\begin{multline*}
  \E_\beta\left[|K_x| |K_x \cap W|\right] = \sum_{y\in V, z\in W} \P_\beta(x\leftrightarrow y \leftrightarrow z) 
  \\\leq 2^{3/2} \sum_{y\in V, z\in W} \P_\beta(x\leftrightarrow y)^{1/2}\P_\beta(y\leftrightarrow z)^{1/2}\P_\beta(z\leftrightarrow x)^{1/2}.
\end{multline*}
The Cauchy-Schwarz inequality lets us bound the sum over $y$ by
\[
  \sum_{y\in V} \P_\beta(x\leftrightarrow y)^{1/2}\P_\beta(y\leftrightarrow z)^{1/2} \leq \sqrt{\sum_{y\in V} \P_\beta(x\leftrightarrow y) \sum_{y\in V} \P_\beta(y\leftrightarrow z)} = \E_\beta |K|
\]
for each $z\in W$, and
a second application of Cauchy-Schwarz yields that
\begin{equation*}
  \E_\beta\left[|K_x| |K_x \cap W|\right] \leq 2^{3/2} \E_\beta |K| \sum_{z\in W} \P_\beta(z\leftrightarrow x)^{1/2} \leq 2^{3/2} \E_\beta |K| \sqrt{ |W| \E_\beta |K \cap W|}
  \end{equation*}
  as claimed.
\end{proof}

\begin{remark}
Note that the similar bound 
\[
  \E_\beta\left[|K_x \cap W|^2\right] \preceq \E_\beta |K_x \cap W| \sqrt{|W| \sup_y \E|K_y \cap W|}
\]
is an easy consequence of the universal tightness theorem, with the bound 
\[M_\beta(W) = O\left(\sqrt{|W|\sup_y \E_\beta |K_y \cap W|}\right)\] holding by a similar argument to the proof of \cref{lem:moments_bounded_below_by_M}. Compared to the universal tightness theorem, the advantage of the Gladkov inequality is that it lets us show that $|K \cap W|$ is unlikely to exceed $\sqrt{|W| \E|K \cap W|}$ even after size-biasing or conditioning on certain rare events such as $0$ being connected to a distant point or to a low-intensity ghost field. (In particular, this upper bound on $|K\cap W|$ holds with high probability even under the IIC measure when it is well-defined.)
\end{remark}

\subsection{Regular variation, doubling, and logarithmic integrability}
\label{subsec:regular_variation_and_logarithmic_integrability}

We now introduce the definition and basic properties of (measurable) regularly varying functions, referring the reader to \cite{bingham1989regular} and \cite[Chapters VIII and XIII]{fellerII} for further background.
Recall that a function $f:(0,\infty)\to \R$ is said to be \textbf{regularly varying} if is eventually non-vanishing and the limit $\lim_{x\to \infty} f(\lambda x)/f(x)$ exists for every every $\lambda>0$. If this holds and $f$ is measurable then in fact there must exist $\alpha\in \R$, known as the \textbf{index} of $f$, such that $\lim_{x\to \infty} f(\lambda x)/f(x)=\lambda^\alpha$ for every $\lambda>0$, and moreover this convergence must hold uniformly for $\lambda$ in compact subsets of $(0,\infty)$ \cite[Theorem 1.2.1]{bingham1989regular}. Having regular variation of index $\alpha$ is a strictly stronger property than satisfying an asymptotic estimate of the form $f(x)=x^{\alpha \pm o(1)}$ as $x\to \infty$ (for example, the function $(2+\sin(x))x^\alpha$ satisfies an estimate of this form but is not regularly varying). The function $f$ is said to be \textbf{slowly varying} if it is regularly varying of index $0$, or equivalently if it is eventually non-vanishing and $f(\lambda x)\sim f(x)$ as $x\to \infty$ for each fixed $\lambda>0$. Every regularly varying function $f$ of index $\alpha$ can be written $f(x)=x^\alpha \ell(x)$ for some slowly-varying function $\ell$. Regular variation is preserved under first-order asymptotic equivalence in the sense that if $f\sim g$ and $g$ is regularly varying then $f$ is regularly varying of the same index. For the remainder of the paper regularly varying functions will always be taken to be measurable.

\medskip

An important property of regularly varying functions of index other than $-1$ is that it is easy to give first-order asymptotic formulae for their integrals:

\begin{lemma}[{\!\!\cite[Theorem VIII.9.1]{fellerII}}]
\label{lem:regular_variation_integration} If $f:(0,\infty) \to \R$ is regularly varying of index $\alpha>-1$ then 
\[\int_1^r f(s) \dif s \sim \frac{rf(r)}{\alpha+1} \]
as $r\to \infty$.
If $f$ is regularly varying of index $\alpha<-1$ then $\int_r^\infty f(s) \dif s <\infty$ for every $r>0$ and 
\[\int_r^\infty f(s) \dif s \sim \frac{rf(r)}{1-\alpha} \]
as $r\to \infty$.
\end{lemma}

In \cref{sec:volume_tail} we will also make use of \emph{Karamata's Tauberian theorem}, stated here in its probabilistic formulation. This theorem generalizes the Hardy-Littlewood Tauberian theorem and was Karamata's original motivation for the introduction of the definition of regular variation.

\begin{theorem}
[{\!\cite[Theorem XIII.5.2]{fellerII}}]
\label{thm:Karamata}
If $f$ is regularly varying of index $-\alpha$ for some $0\leq \alpha <1$ and $X$ is a non-negative random variable then
\[
\E[1-e^{-h X}] \sim f(1/h) \text{ as $h\downarrow 0$} \qquad \text{ if and only if } \qquad \P(X\geq x) \sim \frac{f(x)}{\Gamma(1-\alpha)} \text{ as $x\to\infty$,}
\]
where $\Gamma(z)=\int_0^\infty t^{z-1} e^{-t}\dif t$ denotes the Gamma function.
\end{theorem}


\noindent \textbf{Logarithmic integrability}.
We say that a measurable function $h:(0,\infty)\to \R$ is \textbf{logarithmically integrable} if it is locally integrable on some interval of the form $[r_0,\infty)$ and
\[
\int_{r_0}^\infty |h(r)| \frac{dr}{r} = \int_{\log r_0}^\infty |h(e^t)| \dif t <\infty.
\]
We say that $h$ is a \textbf{logarithmically integrable error function} if it is logarithmically integrable and $h(r)\to 0$ as $r\to \infty$. Note that the logarithmically integrable error functions form an algebra in the sense that they are closed under products and linear combinations.
Regularly varying functions of negative index are always logarithmically integrable error functions and, more generally, regularly varying functions are logarithmically integrable error functions if and only if they are logarithmically integrable.  

\medskip

Regularly varying functions and logarithmically integrable error functions will arise naturally in our analysis via the following fact, which can be thought of as a partial converse to \cref{lem:regular_variation_integration}.

\begin{lemma}[$f'\sim a r^{-1}f$]
\label{lem:ODE_self_referential}
If $f:(0,\infty)\to(0,\infty)$ is differentiable, eventually non-vanishing, and satisfies $f'(r)\sim a r^{-1} f(r)$ as $r\to\infty$ for some $a>0$ then $f$ is regularly varying of index $a$. If moreover $f'(r)= a(1-h(r)) r^{-1} f(r)$ for some logarithmically integrable error function $h:(0,\infty)\to \R$  then there exists a constant $A>0$ such that $f(r)\sim A r^a$ as $r\to \infty$.
\end{lemma}

We include a brief proof of this standard fact since it will be instructive for the remainder of our analysis. It is the first of many ODE lemmas we will prove throughout this series of papers.

\begin{proof}[Proof of \cref{lem:ODE_self_referential}]
The estimate $f'(r)\sim a r^{-1} f(r)$ can be rewritten
\[
(\log f(r))' = a(1-h(r)) r^{-1} 
\]
where $|h(r)|\to 0$ as $r\to \infty$. Integrating this equation yields that
\[
f(R) = \exp\left[ - a\int_r^R \frac{h(s)}{s} \dif s\right] \left(\frac{R}{r}\right)^a f(r) 
\]
for every $R\geq r \geq 1$. This clearly implies the claim that $f$ is regularly varying of index $a$. Moreover, if $h$ is logarithmically integrable then, taking $r_0$ sufficiently large that $h$ is locally integrable on $[r_0,\infty)$,
\[
f(r) = \exp\left[ - a\int_{r_0}^r \frac{h(s)}{s} \dif s\right] (r/r_0)^a f(r_0)  \sim \exp\left[ - a\int_{r_0}^\infty \frac{h(s)}{s} \dif s\right] (r/r_0)^a f(r_0),
\]
so that the claimed estimate $f(r)\sim A r^a$ holds with $A= r_0^{-a}f(r_0)\exp[ - a\int_{r_0}^\infty \frac{h(s)}{s} \dif s]$.
\end{proof}

Finally, let us mention that an increasing function $f:[1,\infty)\to (0,\infty)$ is said to be \textbf{doubling} if there exists a constant $C$ such that $f(2x) \leq Cf(x)$ for every $x\geq 1$, so that increasing functions of regular variation are always doubling. An increasing function $f:[1,\infty)\to (0,\infty)$ is said to be \textbf{reverse doubling} if its inverse is doubling, or equivalently if there exists a constant $C$ such that $f(Cx)\geq 2f(x)$ for every $x \geq 1$; positive, increasing functions with regular variation of positive index are both doubling and reverse-doubling.


\subsection{Previous results on long-range percolation}
\label{sub:previous_results_on_long_range_percolation}

In this section we give formal statements of previous results for long-range percolation that will be used in our analysis.
The most important such result is the following spatially-averaged upper bound on the two-point function, proven in \cite{hutchcroft2022sharp} following the hierarchical analysis of \cite{hutchcrofthierarchical}. (The theorem holds vacuously when $\alpha \geq d$.)

\begin{thm}
\label{thm:two_point_spatial_average_upper} $\sum_{x\in B_r} \P_{\beta_c}(0 \leftrightarrow x) \preceq r^\alpha$ for every $r\geq 1$.
\end{thm}

Under the stronger assumption that $\alpha<1$, a matching lower bound on the spatially averaged two-point function was proven by B\"aumler and Burger \cite{baumler2022isoperimetric}. These results were then strengthened to \emph{pointwise} up-to-constants estimates on the two-point function when $\alpha<1$ in \cite{hutchcroft2024pointwise}. (We conjecture that the pointwise \emph{upper bound} holds for all $\alpha<d$.)

\begin{thm}
\label{thm:two_point_pointwise} If $\alpha<1$ then
 $\P_{\beta_c}(x \leftrightarrow y) \asymp \|x-y\|^{-d+\alpha}$
for every $x\neq y$ in $\Z^d$.
\end{thm}

\cref{thm:two_point_spatial_average_upper} and \cref{lem:moments_bounded_below_by_M} have the following important consequence regarding the quantity $M_r=M_{\beta_c,r}$ defined in \cref{subsec:correlation_inequalities_and_the_universal_tightness_theorem}. (As before, the theorem holds vacuously when $\alpha \geq d$.)

\begin{corollary}
\label{cor:M_upper_bound_general}
$M_r \preceq r^{(d+\alpha)/2}$ for every $r\geq 1$. In fact, the same bound holds for the edian of $|K_\mathrm{max}(B_r)|$ under the measure $\P_{\beta_c}$.
\end{corollary}

As discussed in \cref{subsec:scaling_limits_in_the_critical_dimension_under_the_hydrodynamic_condition}, we will eventually (over the course of the entire series) prove that this upper bound is of the correct order in the effectively long-range low-dimensional regime but that $M_r=o(r^{(d+\alpha)/2})$ when $d\geq 3\alpha$. 

\section{Consequences of volume moment asymptotics}
\label{sec:volume_tail}

We will spend much time accross this series of papers establishing asymptotic estimates on moments such as $\E_{\beta_c,r}|K|^p$ as a function of the cut-off scale $r$. 
Before we start doing this, we will first spend this section explaining how such estimates can be used to deduce analogous estimates on the volume tail 
$\P_{\beta_c}(|K|\geq n)$. 
 In \cref{subsec:volume_tail_up_to_constants} we give conditions under which the volume tail can be estimated up to constants, while in \cref{subsec:tauberian} we give conditions under which first-order asymptotics of the volume tail can be determined. \cref{thm:hd_moments_main} implies that the conditions needed for the stronger conclusions of \cref{subsec:tauberian} hold when $d> 3\alpha$ (or $d>6$ and appropriate two-point estimates hold), while \cref{thm:critical_dim_moments_main_slowly_varying} (and, later, \cref{III-thm:critical_dim_hydro})  show that this remains true when $d=3\alpha$. The weaker conclusions of
 \cref{subsec:volume_tail_up_to_constants} also apply to the low-effective-dimensional models studied in the second paper in the series as established in \cref{II-thm:main_low_dim,II-prop:negligibility_of_mesoscopic}. (We conjecture that the stronger results of \cref{subsec:tauberian} also apply in this case, but are currently unable to prove this; this is related to establishing the existence of the scaling limit in low dimensions.) We phrase all of the results of this section at a fairly high level of generality so that they can be applied across all the cases we consider.

\subsection{Up-to-constants estimates}
\label{subsec:volume_tail_up_to_constants}

The goal of this subsection is to prove the following proposition, which gives conditions on the law of $|K|$ under the cut-off measure $\P_{\beta_c,r}$ under which we can estimate the  volume tail $\P_{\beta_c}(|K|\geq n)$ (without cut-off) up to constants. Recall that we write $\hat \P_{\beta_c,r}$ for the size-biased measure, whose Radon-Nikodym derivative with respect to $\P_{\beta_c,r}$ is equal to $|K|/\E_{\beta_c,r}|K|$.

\begin{prop}[Up-to-constants volume-tail asymptotics from moment asymptotics]
\label{prop:up-to-constants_volume-tail}
Suppose that $\E_{\beta_c,r}|K| \asymp r^\alpha$ as $r\to \infty$ and let $f:(0,\infty)\to(0,\infty)$ be an increasing function whose inverse $f^{-1}$ is doubling.
\begin{enumerate}
  \item If there exist positive constants $c$ and $C$ such that $\hat \P_{\beta_c,r}(cf(r) \leq |K| \leq Cf(r))\geq c$ for every $r\geq 1$ then
  \[
\P_{\beta_c}(|K|\geq n) \succeq \frac{f^{-1}(n)^\alpha}{n}
  \]
  for every $n\geq 1$.
  \item If for each $\eps>0$ there exists $\delta>0$ such that $\hat \P_{\beta_c,r}(|K| \leq \delta f(r)) \leq \eps$ for every $r\geq 1$ then
  \[
\P_{\beta_c}(|K|\geq n) \preceq  \frac{f^{-1}(n)^\alpha}{n}
\]
for every $n\geq 1$.
\end{enumerate}
\end{prop}

The assumption that $f^{-1}$ is doubling holds in particular if $f$ is regularly varying of positive index. A sufficient condition for the hypotheses of both items to hold is that $\hat \E_r |K| \asymp f(r)$ and that the measures $\{Q_r: r\geq 1\}$ are tight in $(0,\infty)$, where $Q_r$ is defined to be the law of $|K|/\hat \E_{\beta_c,r}|K|$ under the size-biased measure $\hat \P_{\beta_c,r}$; this is the condition we will verify with $f(r)=r^{(d+\alpha)/2}$ in the effectively long-range, low-dimensional regime (see \cref{II-prop:negligibility_of_mesoscopic}). For the cases treated by \cref{thm:hd_moments_main,thm:critical_dim_moments_main_slowly_varying} we are also able to verify this condition for an appropriate choice of $f$, but in these cases we are also able to prove much more precise estimates as explained in the next subsection.

\begin{proof}[Proof of \cref{prop:up-to-constants_volume-tail}] We begin with the volume tail lower bound of item 1. We have by the definitions that
\[
\P_{\beta_c}(|K|\geq n) \geq \P_{\beta_c,r}(|K|\geq n) = \E_{\beta_c,r}|K| \cdot \hat \E_{\beta_c,r}\frac{\mathbbm{1}(|K|\geq n)}{|K|}.
\]
Thus, under the assumptions of item 1, we have that
\[
\P_{\beta_c}(|K|\geq cf(r)) \geq \E_{\beta_c,r}|K| \cdot \hat \E_{\beta_c,r}\frac{\mathbbm{1}(cf(r) \leq |K|\leq Cf(r))}{|K|} \asymp \frac{r^\alpha}{f(r)}
\]
for every $r\geq 1$. The claimed lower bound follows by taking $r=f^{-1}(n/c)$, which is of the same order as $f^{-1}(n)$ by assumption. 

\medskip

We now turn to the volume tail upper bound of item 2. Let $\cG$ be an independent \emph{ghost field} of intensity $h$ independent of the percolation configuration, i.e., a random subset of $\Z^d$ in which each vertex is included independently at random with inclusion probability $1-e^{-h}$. Connection probabilities to the ghost field encode the Laplace transform of the cluster volume distribution:
\[
\P_{\beta_c,r,h}(0\leftrightarrow \cG) = \E_{\beta_c,r}[1-e^{-h|K|}]
\]
for every $r,h>0$, where we write $\P_{\beta_c,r,h}$ for the joint law of the cut-off percolation configuration and the ghost field. If the origin is connected to the ghost field in the model without cut-off but not in the model with cut-off then, in the monotone coupling, there must exist an edge in the boundary of the cluster in the cut-off model that becomes open in the non-cut-off model and whose other endpoint is connected to the ghost by a path that is disjoint from the cluster of the origin in the cut-off model. Since exploring the cluster of the origin in the cut-off model reveals only negative information about connectivity to the ghost off of this cluster, we have by these considerations and a union bound that
\begin{equation}
\label{eq:volume_tail_union_bound}
\P_{\beta_c,\infty,h}(0 \leftrightarrow \cG) \leq 
 \P_{\beta_c,r,h}(0 \leftrightarrow \cG) + 
Cr^{-\alpha}  \E_{\beta_c,r}\left[|K|e^{-h|K|}\right]
 \P_{\beta_c,\infty,h}(0 \leftrightarrow \cG),
\end{equation}
where $\P_{\beta_c,\infty,h}$ denotes the joint law of the model without cut-off and the independent ghost field, $\E_{\beta_c,r}|K|e^{-h|K|}$ is the expected size of the cluster on the indicator that it is not connected to $\cG$, and $r^{-\alpha}$ is the order of the expected number of edges of the origin that flip from closed to open when passing from $\P_{\beta_c,r}$ to $\P_{\beta_c}$.
(See \cite[Section 5.2]{hutchcroft2022critical} for a longer-form explanation of how to derive this inequality in the hierarchical model; the details here are very similar.) Our hypotheses imply that there exists $\delta>0$ such that if $r\geq 1$ and $h \geq 1/(\delta f(r))$ then
\[
Cr^{-\alpha}  \E_{\beta_c,r}\left[|K|e^{-h|K|}\right] = C r^{-\alpha} \E_{\beta_c,r}|K| \hat \E_{\beta_c,r}[e^{-h|K|}] \leq \frac{1}{2}.
\]
(To see this, note that $r^{-\alpha}\E_{\beta_c,r}|K|$ is bounded by the assumption that $\E_{\beta_c,r}|K| \asymp r^\alpha$, while $ \hat \E_{\beta_c,r}[e^{-h|K|}]$ is small when $h \gg 1/f(r)$ by the hypothesis of item 2.)
 Thus, letting $h>0$ and taking $r=f^{-1}(1/\delta h)$, we obtain that
\[
\P_{\beta_c,\infty,h}(0 \leftrightarrow \cG) \leq 
2\P_{\beta_c,r,h}(0 \leftrightarrow \cG)
\]
when $r \geq f^{-1}(h/\delta)$. Using the bound $\P_{\beta_c,r,h}(0 \leftrightarrow \cG) =\E_{\beta_c,r}[1-e^{-h|K|}] \leq h \E_{\beta_c,r}|K| \asymp h r^\alpha$ we obtain that
\[
\P_{\beta_c,\infty,h}(0 \leftrightarrow \cG) \preceq f^{-1}(1/\delta h)^\alpha h \asymp f^{-1}(1/h)^\alpha h
\]
and, taking $h=1/n$, it follows that
\[
\P_{\beta_c}(|K|\geq m) = \frac{\P_{\beta_c,\infty,1/n}(|K|\geq n \text{ and } K \cap \cG \neq \emptyset)}{\P_{\beta_c,\infty,1/n}(K \cap \cG \neq \emptyset \mid |K|\geq n)} \leq \frac{\P_{\beta_c,\infty,1/n}(K \cap \cG \neq \emptyset)}{1-e^{-1}} \preceq \frac{f^{-1}(n)^\alpha}{n}
\]
as claimed.
\end{proof}

\subsection{First-order asymptotics via Tauberian theory}
\label{subsec:tauberian}

In this subsection we prove the following proposition, which gives conditions under which the first-order asymptotics of the volume tail $\P_{\beta_c}(|K|\geq n)$ can be computed from the asymptotics of the moments $\E_{\beta_c,r}|K|^p$. We then use this proposition to prove that the volume-tail asymptotics \eqref{eq:hd_volume_tail_main} and \eqref{eq:critical_dim_volume_tail_main} follow from the moment asymptotics \eqref{eq:hd_moments_main} and \eqref{eq:critical_dim_moments_main}. 

\begin{prop}[First-order volume-tail asymptotics from moment asymptotics]
\label{prop:first_order_volume_tail}
Suppose that the following conditions are satisfied:
\begin{enumerate}
  \item $\E_{\beta_c,r}|K| \sim A_1 r^\alpha$ as $r\to\infty$ for some $A_1 \in (0,\infty)$. 
  \item $\hat \E_{\beta_c,r}|K| = \E_{\beta_c,r}|K|^2/\E_{\beta_c,r}|K| \sim A_2 f(r)$ as $r\to\infty$ for some regularly varying function $f$ of index $a \cdot \alpha$ for some $a>0$ and some $A_2\in (0,\infty)$.
  \item The measure $Q_r$, defined to be the law of $|K|/\hat\E_{\beta_c,r}|K|$ under the size-biased measure $\hat \P_{\beta_c,r}$, converges weakly to a measure $Q_\infty$ as $r\to\infty$.
  \item The limit measure $Q_\infty$ does not have an atom at zero.
\end{enumerate}
Then
\begin{align}
\E_{\beta_c}[1-e^{-h|K|}] &\sim A_3 f^{-1}(1/h)^\alpha h && \text{as $h\downarrow 0$, and}\label{eq:Laplace_asymptotics_general}\\
\P_{\beta_c}(|K|\geq n) &\sim A_4 \frac{f^{-1}(n)^\alpha}{n} && \text{as $n\to \infty$,}\label{eq:volume_tail_asymptotics_general}
\end{align}
where the constants $A_3$ and $A_4$ are given by
\[A_3= \frac{A_1}{A_2^{1/a}} \lim_{\lambda\to \infty} \lambda^{-(a-1)/a} \int_0^\infty \frac{1-e^{-\lambda x}}{x} \dif Q_\infty(x) \qquad\text{ and } \qquad A_4=  \frac{A_3}{\Gamma(1/a)}.\]
\end{prop}

Here $\Gamma$ denotes the Gamma function, which has $\Gamma(1/2)=\sqrt{\pi}$, so that the two constants $A_3$ and $A_4$ differ by $A_3=\sqrt{\pi} A_4$ when $a=2$. It is an implicit conclusion of the proposition that the $\lambda \to \infty$ limit taken in the definition of $A_3$ is well-defined. Note that if $f$ is regularly varying of positive index then its inverse $f^{-1}$ is also regularly varying \cite[Theorem 1.5.12]{bingham1989regular}.

\begin{remark}
A similar proposition holds with the expectation $\E_{\beta_c,r}$ replaced by $\E_{\beta_c-\eps}$ and with the estimate $\E_{\beta_c,r}|K| \sim A_1 r^\alpha$ replaced by $\E_{\beta_c-\eps}|K| \sim A_1 \eps^{-1}$; this may be relevant for nearest-neighbor models in high dimensions. The argument given here could also be applied in the hierarchical model to improve the volume-tail estimates of \cite{hutchcroft2022critical} from up-to-constants estimates to first-order estimates when $d\geq 3\alpha$.
\end{remark}

\begin{remark}
Taking $n=\hat \E_{\beta_c,r}|K|$, \cref{prop:first_order_volume_tail} and its proof imply in particular that
\begin{equation}
\label{eq:volume_tail_without_inverses}
\P_{\beta_c}(|K|\geq \hat \E_{\beta_c,r}|K|) \asymp \P_{\beta_c,r}(|K|\geq \hat \E_{\beta_c,r}|K|) 
 \asymp \frac{r^\alpha}{\hat \E_{\beta_c,r}|K|} \asymp \frac{\E_{\beta_c,r}|K|}{\hat \E_{\beta_c,r}|K|}
\end{equation}
for every $r\geq 1$ when the hypotheses of \cref{prop:first_order_volume_tail} are satisfied, which holds in particular under the assumptions of either \cref{thm:hd_moments_main} or \cref{thm:critical_dim_moments_main_slowly_varying}.
\end{remark}


\begin{proof}[Proof of \cref{prop:first_order_volume_tail}]
We prove the estimate \eqref{eq:Laplace_asymptotics_general} on the moment generating function of $|K|$, with the estimate \eqref{eq:volume_tail_asymptotics_general} following from this together with Karamata's Tauberian theorem \cite[Theorem XIII.5.2]{fellerII} as stated in \cref{thm:Karamata}.
Let $\cG$ be an independent ghost field of intensity $h$ independent of the percolation configuration, so that
\[
\P_{\beta_c,r,h}(0\leftrightarrow \cG) = \E_{\beta_c,r}[1-e^{-h|K|}]
\]
for every $r,h>0$. We have as in the proof of \cref{prop:up-to-constants_volume-tail} that there exists a constant $C$ such that
\begin{equation}
\label{eq:Tauberian_union}
\P_{\beta_c,r,h}(0 \leftrightarrow \cG) \leq \P_{\beta_c,\infty,h}(0 \leftrightarrow \cG) \leq \P_{\beta_c,r,h}(0 \leftrightarrow \cG)+Cr^{-\alpha}  \E_{\beta_c,r}\left[|K|e^{-h|K|}\right]
 \P_{r,h}(0 \leftrightarrow \cG).
\end{equation}
We will apply this estimate with $r=f^{-1}(\lambda/A_2h)$, where we think of $\lambda$ being fixed and take $h\downarrow 0$.
With this choice of $r$, we have that
\begin{multline*}
\E_{\beta_c,r}\left[|K|e^{-h|K|}\right] = \E_{\beta_c,r}|K| \cdot \hat \E_{\beta_c,r}[e^{-h|K|}] =  \E_{\beta_c,r}|K| \cdot \int_0^\infty e^{-(h \hat \E_{\beta_c,r} |K|)x} \dif Q_r(x)
\\\sim A_1 r^{\alpha} \int_0^\infty e^{-\lambda x} \dif Q_\infty(x)
\end{multline*}
as $h\downarrow 0$ with $r=f^{-1}(\lambda/A_2h)$ for fixed $\lambda>0$. Since $Q_\infty$ does not have an atom at zero, the integral on the right hand side converges to zero as $\lambda \to \infty$, and it follows that for every $\eps>0$ there exists $\lambda_0<\infty$ and $h_0>0$ such that if $\lambda\geq \lambda_0$, $h\leq h_0$, and $r=f^{-1}(\lambda/A_2 h)$ then $Cr^{-\alpha}  \E_{r}[|K|e^{-h|K|}] \leq \eps$. Thus, it follows from \eqref{eq:Tauberian_union} that if $\lambda \geq \lambda_0$ and $h\leq h_0$ then
\begin{equation}
\label{eq:ghost_truncation_stability}
1 \leq \frac{\P_{\beta_c,\infty,h}(0 \leftrightarrow \cG)}{\P_{\beta_c,r,h}(0 \leftrightarrow \cG)} \leq (1-\eps)^{-1}.
\end{equation}
We now compute the asymptotics of $\P_{\beta_c,r,h}(0 \leftrightarrow \cG)$ for this choice of $r$.
Similarly to above, we can write
\begin{align*}
\P_{\beta_c,r,h}(0 \leftrightarrow \cG) &= \E_{\beta_c,r}\left[1-\exp[-h|K|]\right] = \hat\E_{\beta_c,r}\left[\frac{1-\exp[-h|K|]}{|K|}\right] \E_{\beta_c,r} |K|
\\
&=  \frac{\E_{\beta_c,r}|K|}{\hat \E_{\beta_c,r}|K|} \int_0^\infty \frac{1-\exp[-(h \hat \E_{\beta_c,r}|K|) x]}{x} \dif Q_r(x)
\\&\sim \frac{A_1 f^{-1}(\lambda/A_2 h)^\alpha h}{\lambda}\int_0^\infty \frac{1-\exp[-\lambda x]}{x} \dif Q_\infty(x),
\end{align*}
and since $f^{-1}$ is regularly varying of index $1/(a\alpha)$ we obtain that
\begin{align}
\label{eq:ghost_truncated_asymptotics}
\P_{r,h}(0 \leftrightarrow \cG) \sim \left[\frac{A_1}{A_2^{1/a}\lambda^{1-1/a}}\int_0^\infty \frac{1-\exp[-\lambda x]}{x} \dif Q_\infty(x) \right] f^{-1}(1/h)^\alpha h
\end{align}
when $\lambda>0$ is fixed, $r=f^{-1}(\lambda/A_2 h)$, and $h\downarrow 0$. The claim follows immediately from the two esimates \eqref{eq:ghost_truncation_stability} and \eqref{eq:ghost_truncated_asymptotics}. 
\end{proof}

We now explain how to deduce the volume tail estimates of \cref{thm:hd_moments_main,thm:critical_dim_moments_main_slowly_varying} follow from the moment asymptotics stated in the same theorems.

\begin{proof}[Proof of \eqref{eq:hd_volume_tail_main} and \eqref{eq:critical_dim_volume_tail_main} given \eqref{eq:hd_moments_main} and \eqref{eq:critical_dim_moments_main}]
In both cases the limit measure $Q_\infty$ is simply the law of a chi-squared random variable with one degree of freedom, with
\[
\E_r|K| \sim \frac{\alpha}{\beta_c}r^\alpha \qquad \text{ and } \qquad \hat \E_r |K| \sim \begin{cases} A r^{2\alpha} & d>3\alpha\\
A_r r^{2\alpha} & d=3\alpha.
\end{cases}
\]
Indeed, the fact that the measure $Q_r$ converges to the chi-squared random variable with one degree of freedom follows from \eqref{eq:hd_moments_main} or \eqref{eq:critical_dim_moments_main} as appropriate since the moments of $Q_r$ satisfy
\[
\int_0^\infty x^p \dif Q_r = \frac{\hat \E_{\beta_c,r}|K|^{p}}{(\hat \E_{\beta_c,r}|K|)^{p}} = \frac{\E_{\beta_c,r}|K|^{p+1}}{(\E_{\beta_c,r}|K|^2 / \E_{\beta_c,r}|K|)^{p}\E_{\beta_c,r}|K|}  \sim (2p-1)!!
\]
as $r\to \infty$. It follows by dominated convergence that any subsequential limit of the measures $Q_r$ has $p$th moment $(2p-1)!!$ for each integer $p\geq 1$ and hence by Carleman's criterion that $Q_r$ converges to a chi-squared random variable with one degree of freedom as claimed.
We can compute the relevant integral with respect to $Q_\infty$ exactly, obtaining that
\[
\int_0^\infty \frac{1-e^{-\lambda x}}{x} \dif Q_\infty(x) = \frac{1}{\sqrt{2\pi}} \int_0^\infty (1-e^{-\lambda x})e^{-x/2} x^{-3/2} \dif x = \sqrt{2\lambda+1}-1,
\]
and hence that
\[
\lim_{\lambda\to\infty} \frac{1}{\sqrt{\lambda}}\int_0^\infty \frac{1-e^{-\lambda x}}{x} \dif Q_\infty(x)= \sqrt{2}.
\]
Thus, applying \cref{prop:first_order_volume_tail}, we obtain that
%
\[
\E_{\beta_c}[1-e^{-h|K|}] \sim \frac{\alpha}{\beta_c} \sqrt{2} \begin{cases} A^{-1/2}h^{1/2} & d> 3\alpha\\
\tilde A_{1/h}^{-1/2} h^{1/2} & d=3\alpha
\end{cases}
\]
as $h\downarrow 0$
 and
\[
\P_{\beta_c}(|K|\geq n) \sim \frac{\alpha}{\beta_c} \sqrt{\frac{2}{\pi}} \begin{cases} A^{-1/2} n^{-1/2} & d> 3\alpha\\
\tilde A_n^{-1/2} n^{-1/2} & d=3\alpha
\end{cases}
\]
as $n\to\infty$ as claimed, where in the case $d=3\alpha$
we have 
expressed the inverse of $A_r r^{2\alpha}$ as $\tilde A_{s}^{-1/2\alpha} s^{1/2\alpha}$ with $\tilde A_s$ slowly varying.
\end{proof}


\section{Analysis of volume moments}

\label{sec:analysis_of_moments}

In this section, we analyze the asymptotics of the cluster moments $\E_{\beta_c,r}|K|^p$ as $r\to \infty$ under the hypotheses of \cref{thm:hd_moments_main,thm:critical_dim_moments_main_slowly_varying}. 
All our analysis will be based on the following formula for the derivative. Recall that we write $B_r=\{x\in \Z^d:\|x\|\leq r\}$ for the lattice points in the closed $\|\cdot\|$-ball of radius $r$.

\begin{lemma}
\label{lem:moment_derivative}
The equality
\begin{align}
\frac{\partial}{\partial r}\E_{\beta,r}|K|^p &= \beta |J'(r)| \sum_{\ell=0}^{p-1} \binom{p}{\ell}\E_{\beta,r}\left[|K|^{\ell+1} \sum_{y\in B_r} \mathbbm{1}(y\notin K) |K_y|^{p-\ell} \right]
\label{eq:pth_moment_derivative}
\end{align}
holds for every integer $p\geq 1$, every $0\leq \beta \leq \beta_c$ and every $r>0$.
\end{lemma}

The proof of this lemma will apply the \textbf{mass-transport principle}
 for $\Z^d$, which states that if $F:\Z^d\times \Z^d \to[0,\infty]$ satisfies $F(x+z,y+z)=F(x,y)$ for every $x,y,z\in \Z^d$ then
 \begin{equation}
 \label{eq:MTP}
   \sum_{x\in \Z^d} F(0,x)=\sum_{x\in \Z^d} F(0,-x)=\sum_{x\in \Z^d} F(x,0),
 \end{equation}
 where the first equality follows since $x\mapsto-x$ is a bijection and the second follows by the assumed translation-invariance of $F$.
 More generally, if $F:(\Z^d)^k \to[0,\infty]$ satisfies $F(x_1+z,\ldots,x_k+z)=F(x_1,\ldots,x_k)$ for every $x_1,\ldots,x_k, z\in \Z^d$ then
 \begin{multline}
 \label{eq:MTP_general}
   \sum_{x_2,\ldots,x_k \in \Z^d} F(0,x_2,\ldots,x_k)=\sum_{x_1,x_3,\ldots,x_k\in \Z^d} F(x_1,0,x_3,\ldots,x_k)\\=\cdots = \sum_{x_1,\ldots,x_{k-1}\in \Z^d} F(x_1,x_2,\ldots,x_{k-1},0),
 \end{multline}
 as follows from repeated application of \eqref{eq:MTP}.

\begin{proof}[Proof of \cref{lem:moment_derivative}]
We have by Russo's formula \cite[Section 2.4]{grimmett2010percolation} that
\begin{align*}
\frac{\partial}{\partial r}\E_{\beta,r}|K|^p &= \beta |J'(r)| \E_{\beta,r}\left[\sum_{x\in K} \sum_{y\in B_r(x)} \mathbbm{1}(y\notin K) [(|K|+|K_y|)^p-|K|^p]\right],
\end{align*}
where we interpret $\beta |J'(r)| \mathbbm{1}(\|x-y\|\leq r)$ as the rate at which $\{x,y\}$ flips from closed to open as $r$ is increased. (The fact that Russo's formula can be applied to these moments directly in infinite volume when $\beta \leq \beta_c$ and $r<\infty$ is a standard consequence of the sharpness of the phase transition.) We can use the mass-transport principle to rewrite this as
\begin{align*}
\frac{\partial}{\partial r}\E_{\beta,r}|K|^p
&=\beta |J'(r)| \E_{\beta,r}\left[|K| \sum_{y\in B_r} \mathbbm{1}(y\notin K) [(|K|+|K_y|)^p-|K|^p]\right]\\
&= \beta |J'(r)| \sum_{\ell=0}^{p-1} \binom{p}{\ell}\E_{\beta,r}\left[|K|^{\ell+1} \sum_{y\in B_r} \mathbbm{1}(y\notin K) |K_y|^{p-\ell} \right],
\end{align*}
where the final equality follows by expanding out the binomial.
\end{proof}

As explained in the introduction, one of our primary goals throughout this series of papers will be to understand when this formula admits the mean-field approximation
\begin{align}
\frac{d}{d r}\E_{\beta_c,r}|K|^p \sim  \beta_c r^{-\alpha-1} \sum_{\ell=0}^{p-1} \binom{p}{\ell}\E_{\beta_c,r}|K|^{\ell+1} \E_{\beta_c,r}|K|^{p-\ell},
\label{eq:pth_moment_derivative_simplified}
\end{align}
which is valid when the clusters at $0$ and a uniform point of $B_r$ are ``asymptotically independent'' in an appropriately strong sense.
Roughly speaking, we will see that this approximation is valid above the critical dimension and is \emph{not} valid below the critical dimension (we prove this in the effectively long-range, low-dimensional regime, but it is presumably also true in the effectively long-range low-dimensional regime.) When $d_\mathrm{eff}=d_c=6$ but $d<6$ (as in \cref{III-thm:critical_dim_moments_main}), this approximation is valid but the errors decay slowly, so that a second-order analysis (carried out in \cref{III-sec:logarithmic_corrections_at_the_critical_dimension}) is necessary to understand the asymptotics of solutions.

\medskip

\subsection{First moment}
\label{subsec:first_moment}

In this section we give conditions slightly more general than those covered by \cref{thm:hd_moments_main,thm:critical_dim_moments_main_slowly_varying} under which the $r\to\infty$ asymptotics of $\E_{\beta_c,r}|K|$ may be computed exactly. Recall that the \textbf{triangle diagram} is defined by $\nabla_\beta=\sum_{x,y\in \Z^d} \P_{\beta}(0\leftrightarrow x)\P_{\beta}(x\leftrightarrow y)\P_{\beta}(y\leftrightarrow 0)$ and that the \textbf{triangle condition} is said to hold if $\nabla_{\beta_c}<\infty$. We also define the \textbf{open triangle} $\nabla_\beta(0,x)=\sum_{y,z\in \Z^d} \P_{\beta}(0\leftrightarrow y)\P_{\beta}(y\leftrightarrow z)\P_{\beta}(z\leftrightarrow x)$ and write $\nabla_r=\sup_{\|x\|\geq r} \nabla_{\beta_c}(0,x)$, which goes to zero as $r\to\infty$ whenever the triangle condition holds \cite{MR2779397,MR1127713}.

\begin{prop}
\label{prop:first_moment}
Suppose that at least one of the following is true:
\begin{enumerate}
  \item $d>3\alpha$.
  \item $d=3\alpha$ and the hydrodynamic condition \eqref{Hydro} holds.
  \item The triangle condition holds.
\end{enumerate}
Then $\E_{\beta_c,r} |K|$ can be expressed asymptotically as
\[
\E_{\beta_c,r} |K| \sim \frac{\alpha}{\beta_c} r^\alpha.
\]
If either $d>3\alpha$ or the triangle condition holds and $\nabla_r$ is logarithmically integrable then the error in this approximation
\[\cH_r:=\frac{\beta_c}{\alpha} r^{-\alpha} \E_{\beta_c,r}|K|-1\]
 is logarithmically integrable. 
\end{prop}

\cref{prop:first_moment} applies under the hypotheses of \eqref{HD} by the following elementary lemma, whose proof is deferred to the end of the section.

\begin{lemma}
\label{lem:log_integrable_triangle}
If Cases 2 or 3 of \eqref{HD} hold then $\nabla_r$ is logarithmically integrable.
\end{lemma}

\begin{remark}
It will follow from the computations in \cref{III-sec:logarithmic_corrections_at_the_critical_dimension} that the error term $\cH_r$ is \emph{not} logarithmically integrable when $d=3\alpha<6$.
\end{remark}

To begin working towards the proof of \cref{prop:first_moment}, we first rewrite the $p=1$ case of the equality \eqref{lem:moment_derivative} as
\begin{equation}
\label{eq:first_moment_ODE_error_def}
\frac{d}{dr}\E_{\beta_c,r}|K| = \beta_c |J'(r)| \E_{\beta_c,r}\left[\sum_{y\in B_r}\mathbbm{1}(y\notin K) |K||K_y| \right]
= (1-\cE_{1,r}) \beta_c |J'(r)| |B_r| (\E_{\beta_c,r}|K|)^2
\end{equation}
where the error term $\cE_{1,r}$ is defined by
\begin{multline*}
\cE_{1,r} := 
\frac{|B_r|(\E_{\beta_c,r}|K|)^2 - \E_{\beta_c,r} \left[|K| \sum_{y\in B_r} \mathbbm{1}(y\notin K) |K_y|\right]}{|B_r| (\E_{\beta_c,r}|K|)^2}
\\
=\frac{|B_r|(\E_{\beta_c,r}|K|)^2 - \E_{\beta_c,r} \left[|K| \sum_{y\in B_r} |K_y|\right]+\E_{\beta_c,r} \left[|K|^2 |K\cap B_r|\right] }{|B_r| (\E_{\beta_c,r}|K|)^2}.
\end{multline*}
We also define the errors $\cE_{0,r}$ and $\overline{\cE}_{1,r}$ by
\[
(1-\cE_{0,r}):=\frac{|J'(r)||B_r|}{r^{-\alpha-1}} \text{ and } (1-\overline{\cE}_{1,r}) = (1-\cE_{1,r})\frac{|J'(r)||B_r|}{r^{-\alpha-1}} = (1-\cE_{1,r})(1-\cE_{0,r}),
\]
so that
\begin{equation}
\label{eq:barE11_def_ODE}
\frac{d}{dr}\E_{\beta_c,r}|K| 
= (1-\overline{\cE}_{1,r}) \beta_c r^{-\alpha-1} (\E_{\beta_c,r}|K|)^2.
\end{equation}
Our aim is to show that the error term $\overline{\cE}_{1,r}$ appearing here is small under the hypotheses of \cref{prop:first_moment} and to deduce the first-order asymptotic estimate $\E_{\beta_c,r}|K|\sim \frac{\alpha}{\beta_c}r^\alpha$.

\medskip

Our analysis of this equation will rely on the following elementary ODE lemma.
 (Be careful to note that this lemma includes an implicit assumption that $f$ does not blow up in finite time, which is important for items 1 and 3.)

\begin{lemma}[The asymptotic ODE $f'\sim A r^{-\alpha-1}f^2$]
\label{lem:f'=f^2}
Let $f:(0,\infty)\to (0,\infty)$ be differentiable and let $A$, $\alpha$, and $r_0$ be positive constants.
\begin{enumerate}
  \item If $f'(r) \geq A r^{-\alpha-1} f(r)^2$ for every $r\geq r_0$ then $f(r) \leq A^{-1}\alpha r^\alpha$ for every $r\geq r_0$.
  \item If $f(r)\to\infty$ as $r\to \infty$ and $f'(r) \leq A r^{-\alpha-1} f(r)^2$ for every $r\geq r_0$  then $f(r) \geq A^{-1}\alpha r^\alpha$ for every $r\geq r_0$.
  \item If $f(r)\to\infty$ as $r\to\infty$ and $f'(r)=A(1-h(r))r^{-\alpha-1}f(r)^2$ for some bounded measurable function $h:(0,\infty)\to \R$ then 
  \[
f(r) = \frac{\alpha}{A}r^\alpha\left(1- \alpha r^\alpha \int_r^\infty h(s)s^{-\alpha-1}\dif s\right)^{-1}
  \]
  for every $r>0$.
\end{enumerate}
\end{lemma}

\begin{proof}[Proof of \cref{lem:f'=f^2}]
All three claims follow from the fact that, by the fundamental theorem of calculus and the chain rule,
\[
\frac{1}{f(r)} = \frac{1}{f(R)} + \int_r^R \frac{f'(s)}{f(s)^2} \dif s
\]
for every $r \leq R < \infty$;
 simply take $R\to\infty$ and use that $\liminf_{R\to \infty} 1/f(R) \geq 0$ (in the first claim) or $\lim_{R\to \infty} 1/f(R) = 0$ (in the second two claims) as appropriate.
\end{proof}

Since $\E_{\beta_c}|K|=\infty$, it follows from \cref{lem:f'=f^2} and \eqref{eq:barE11_def_ODE} that the identity
\begin{equation}\label{eq:first_moment_E1_formula}
\E_{\beta_c,r}|K|=\frac{\alpha}{\beta_c} r^\alpha \left(1-\alpha r^\alpha \int_r^\infty \overline{\mathcal{E}}_{1,s} s^{-\alpha-1} \dif s\right)^{-1}
\end{equation}
holds for every $r>0$. In particular, $\E_{\beta_c,r}|K|\sim \frac{\alpha}{\beta_c}r^\alpha$ if and only if $\overline{\cE}_{1,r}\to 0$ as $r\to \infty$. The following standard lemma together with our standing assumptions \eqref{eq:normalization_conventions} on the kernel $J$ ensure that the error $\cE_{0,r}$ is always logarithmically integrable, so that it suffices to bound the error $\cE_{1,r}$.

\begin{lemma}
\label{lem:ball_regularity}
$|B_r| = r^d \pm O(r^{d-1})$ as $r\to \infty$.
\end{lemma}

\begin{proof}[Proof of \cref{lem:ball_regularity}]
 Consider partitioning $\R^d$ into the boxes $x+[-1/2,1/2]^d$ with $x\in \Z^d$, which are disjoint up to sets of measure zero. Letting $C$ be the diameter of the hypercube $[0,1]^d$ under the norm $\|\cdot\|$, we have that every point of $B_r+[-1/2,1/2]^d$ has norm at most $r+C$. Meanwhile, every point with norm at most $r-C$ is contained in one of the cubes $x+[-1/2,1/2]^d$ whose center is a lattice point of norm at most $r$, so that every point of norm at most $r-C$ is contained in the set $B_r+[-1/2,1/2]^d$. Since this set has Lebesgue measure $|B_r|$ while the set of points with norm at most $r$ has Lebesgue measure $r^d$, it follows that
$(r-C)^d \leq B_r \leq (r+C)^d$
for every $r\geq 1$.
\end{proof}

We now explain how the error term $\cE_{1,r}$ can be bounded in the first two cases of \cref{prop:first_moment}. Bounding $\cE_{1,r}$ in terms of the triangle diagram requires a slightly different treatment, which we will return to afterwards.

\begin{lemma}\label{lem:E1_1} The error term $\cE_{1,r}$ satisfies
\[
0\leq \cE_{1,r} \leq  \frac{\E_{\beta_c,r} \left[|K|^2 |K\cap B_r|\right]}{|B_r| (\E_{\beta_c,r}|K|)^2}
\]
for every $r>0$.
\end{lemma}

\begin{proof}[Proof of \cref{lem:E1_1}]
This is an immediate consequence of the covariance bound
\[
0 \leq \E_{\beta_c,r} \left[|K| \sum_{y\in B_r} |K_y|\right]-|B_r|(\E_{\beta_c,r}|K|)^2 \leq \E_{\beta_c,r}\left[|K|^2|K\cap B_r|\right],
\]
which is a special case of \cref{lem:BK_disjoint_clusters_covariance}.
\end{proof}

The fact that $\cE_{1,r}$ is non-negative and $\cE_{0,r}\to 0$ ensures that the following lower bound always holds, without any assumptions on $d$ and $\alpha$.

\begin{corollary} 
\label{cor:mean_lower_bound}
$\liminf_{r\to\infty}r^{-\alpha}\E_{\beta_c,r}|K| \geq \frac{\alpha}{\beta_c}$.
\end{corollary}

We now consider the case $d>3\alpha$. The main step in this case is to prove an upper bound of the correct order on $\E_{\beta_c,r}|K|$, which will be done via a bootstrapping argument. Interestingly, this argument can be carried out from scratch, and does not rely on any of the previous results on the two-point function established in the earlier works \cite{hutchcroft2022sharp,hutchcroft2020power,baumler2022isoperimetric}. 

\begin{lemma}
\label{lem:first_moment_high_dimensional}
If $d > 3\alpha$ then there exists a constant $r_0<\infty$ such that
\[
\E_{\beta_c,r}|K| \leq \frac{2\alpha}{\beta_c} r^\alpha \qquad \text{ and } \qquad \cE_{1,r} \leq \frac{16\alpha^3}{\beta_c^3} r^{3\alpha-d}
\]
for every $r\geq r_0$. In particular, $\overline{\cE}_{1,r}$ is logarithmically integrable.
\end{lemma}

 The basic idea behind the following proof is that if $d>3\alpha$ and a bound of the correct order holds at some large scale $R$, then this information is sufficient to analyze the ODE even at much smaller scales $r\ll R$, since the bound is much more than sufficient to ensure that the error $\cE_{1,r}$ is small. This lets us prove that the claimed bounds hold at all scales via a boostrapping argument.

\begin{proof}[Proof of \cref{lem:first_moment_high_dimensional}]
The claimed bound on $\cE_{1,r}$ follows immediately from the claimed bound on $\E_{\beta_c,r}|K|$ together with \cref{lem:E1_1} and the tree-graph inequality, so we focus on bounding $\E_{\beta_c,r}|K|$. 
We first note that the derivation of \eqref{eq:barE11_def} above can be carried out at other values of $\beta \leq \beta_c$ to obtain that
\[
\frac{\partial}{\partial r} \E_{\beta,r}|K| = (1-\bar \cE_{1,r,\beta}) \beta r^{-\alpha-1}(\E_{\beta,r}|K|)^2
\]
for every $0\leq \beta \leq \beta_c$ and $r>0$, where 
\[
(1-\overline{\cE}_{1,r,\beta}) = \frac{|J'(r)||B_r|}{r^{-\alpha-1}} (1-\cE_{1,r,\beta})=(1-\cE_{0,r})(1-\cE_{1,r,\beta})
\]
and
\[
\cE_{1,r,\beta}:=\frac{|B_r|(\E_{\beta,r}|K|)^2 - \E_{\beta,r} \left[|K| \sum_{y\in B_r} |K_y|\right]+\E_{\beta,r} \left[|K|^2 |K\cap B_r|\right] }{|B_r| (\E_{\beta,r}|K|)^2} \leq \frac{\E_{\beta,r} \left[|K|^2 |K\cap B_r|\right] }{|B_r| (\E_{\beta,r}|K|)^2}.
\]
Using the tree-graph inequality, we can bound 
\begin{equation}
\label{eq:tree-graph-cE_1-bound}
\cE_{1,r,\beta} \leq \frac{\E_{\beta,r} \left[|K|^3\right] }{|B_r| (\E_{\beta,r}|K|)^2} \leq \frac{3 (\E_{\beta,r} |K|)^3 }{|B_r|} \preceq r^{3\alpha-d} \left(\frac{\E_{\beta,r} |K|}{\frac{\alpha}{\beta}r^{\alpha}}\right)^3
\end{equation}
for all $\beta_c/2\leq \beta \leq \beta_c$, so that there exists a constant $C$ such that
\begin{equation*}
\overline{\cE}_{1,r,\beta} \leq C r^{3\alpha-d} \left(\frac{\E_{\beta,r} |K|}{\frac{\alpha}{\beta}r^{\alpha}}\right)^3
\end{equation*}
for every $r\geq 1$ and $\beta_c/2\leq \beta \leq \beta_c$. (In the remainder of the proof we will only really use that the prefactor $Cr^{3\alpha-d}$ goes to zero as $r\to \infty$ independently of the choice of $\beta$ and that the second term is a function of the ratio of $\E_{\beta,r}|K|$ to $r^\alpha$.)
%
%
%
%
For each $\beta_c/2 \leq \beta < \beta_c$, let $r_0(\beta)$ be minimal such that 
\[
\E_{\beta,r}|K| \leq \frac{2\alpha}{\beta} r^\alpha
\]
for every $r\geq r_0(\beta)$, noting that $r_0(\beta)$ is finite for every $\beta_c/2\leq \beta<\beta_c$ by sharpness of the phase transition. 
Since $\E_{\beta,r} |K|$ is monotone in $r$, we can bound
\[
\overline{\cE}_{1,r,\beta} \leq C r^{3\alpha-d} \left(\frac{\E_{\beta,\max\{r,r_0\}} |K|}{\frac{\alpha}{\beta}r^{\alpha}}\right)^3 \leq 8C r^{3\alpha-d} \left(\frac{\max\{r,r_0\}}{r}\right)^{3\alpha},
\]
so that if $r\geq (32 C)^{1/d} \max\{r,r_0\}^{3\alpha/d}$ then $\overline{\cE}_{1,r,\beta} \leq 1/4$. Applying \cref{lem:f'=f^2} yields that $\E_{\beta,r}|K| \leq \frac{4\alpha}{3\beta} r^\alpha \leq \frac{2\alpha}{\beta} r^\alpha$ for every $r\geq (32 C)^{1/d} \max\{r,r_0(\beta)\}^{3\alpha/d}$, and hence by definition of $r_0$ that any such $r$ must be at least $r_0(\beta)$. This implies that $r_0(\beta)\geq (32 C)^{1/d} r_0(\beta)^{3\alpha/d}$ and hence that 
\[
r_0(\beta)\leq (32C)^{1/(d-3\alpha)}
\]
for every $\beta_c/2\leq \beta <\beta_c$. Taking $\beta \uparrow \beta_c$, it follows by monotone convergence that $\E_{\beta_c,r}|K| \leq \frac{2\alpha}{\beta} r^\alpha$ for all $r\geq (32C)^{1/(d-3\alpha)}$.
\end{proof}

We next extend this analysis to the critical dimension under the assumption that the hydrodynamic condition holds. 


\begin{lemma}
\label{lem:first_moment_critical_dim}
If $d=3\alpha$ and the hydrodynamic condition holds then $\cE_{1,r}\to 0$.
\end{lemma}

\begin{proof}[Proof of \cref{lem:first_moment_critical_dim}]
We have by Cauchy-Schwarz that
\begin{equation}
\E_{\beta_c,r}[|K|^2|K\cap B_r|] \leq \sqrt{\E_{\beta,r}[|K|^4] \E_{\beta_c,r}[|K \cap B_r|^2]} \preceq  \sqrt{(\E_{\beta,r}|K|)^7r^\alpha M_r},
\label{eq:E1_from_M}
\end{equation}
where in the second inequality we used the universal tightness theorem via \cref{cor:universal_tightness_moments} and the two-point upper bound of \cref{thm:two_point_spatial_average_upper} to bound $\E_{\beta_c,r}|K \cap B_r|^2 \preceq r^\alpha M_r$ and the tree-graph inequality to bound $\E_{\beta_c,r}[|K|^4]  \preceq (\E_{\beta_c,r}|K|)^7$. Since the hydrodynamic condition holds, we deduce that there exists a continuous decreasing function $\eta:(0,\infty)\to (0,\infty)$ with $\eta(r)\to0$ as $r\to \infty$ such that
\begin{equation}
\label{eq:critical_dim_E1rbeta}
\overline{\cE}_{1,r,\beta} \leq \frac{\eta(r) r^{\alpha/2} r^{(d+\alpha)/4} (\E_{\beta,r}|K|)^{7/2}}{r^d (\E_{\beta,r}|K|)^{2}} =
 \eta(r) r^{-\frac{3}{2}\alpha} (\E_{\beta,r}|K|)^{3/2}
\end{equation}
for every $0\leq \beta \leq \beta_c$ and $r>0$. (Here we have absorbed all implicit constants into the definition of $\eta(r)$.) We can now proceed in much the same way as in the proof of \cref{lem:first_moment_high_dimensional}: If we define $r_0(\beta)$ to be minimal such that $\E_{\beta,r}|K|\leq \frac{2\alpha}{\beta}r^\alpha$ then it follows by the same ODE analysis carried out in the proof of that lemma that there exists a universal constant $c>0$ such that if $r\leq r_0(\beta)$ then
\[
\Bigl[\eta(r) \left(\frac{r_0(\beta)}{r}\right)^{3\alpha/2} \leq c \Bigr]\Rightarrow \Bigl[\overline{\cE}_{1,s,\beta} \leq 2^{3/2} c \text{ for all $s\geq r$}\Bigr]\Rightarrow \Bigl[r_0(\beta) \leq r\Bigr].
\]
Setting $\psi(r)=(c\eta(r))^{-2/3\alpha} r$, it follows that $r_0(\beta)\leq \psi^{-1}(r_0(\beta))$ for every $\beta<\beta_c$. Since $\eta$ is continuous, decreasing, and converges to zero, $\psi^{-1}(r)/r$ is also continuous, decreasing, and converges to zero, so that $r_0(\beta)$ is bounded by the supremal solution to $\psi^{-1}(r)=r$. Since this quantity does not depend on $\beta$, it follows that $\E_{\beta_c,r}|K|\leq \frac{2\alpha}{\beta_c} r^\alpha$ for all sufficiently large $r$, and hence by \eqref{eq:critical_dim_E1rbeta} that $\overline{\cE}_{1,r}\to 0$ as $r\to \infty$.
\end{proof}

It remains to treat the first moment under the triangle condition.

\begin{lemma}
\label{lem:E1_triangle}
The estimate $\cE_{1,r} \leq |B_r|^{-1} \sum_{x\in B_r} \nabla(0,x)$ holds for every $r>0$. 
\end{lemma}

We begin by proving the following lemma, which is very similar to the estimates used to derive differential inequalities on the susceptibility from the triangle condition in \cite{MR762034}.

\begin{lemma}
\label{lem:disjoint_connections_triangle} The estimate
\begin{multline*}
\P_{\beta_c,r}(0\leftrightarrow x)\P_{\beta_c,r}(y \leftrightarrow z) - \P_{\beta_c,r}(0\leftrightarrow x \nleftrightarrow y \leftrightarrow z) 
\\  \leq \sum_{a,b\in \Z^d} \P_{\beta_c,r}(0\leftrightarrow a)\P_{\beta_c,r}(a\leftrightarrow x)\P_{\beta_c,r}(a\leftrightarrow b)\P_{\beta_c,r}(y\leftrightarrow b)\P_{\beta_c,r}(b\leftrightarrow z)
\end{multline*}
holds for every $x,y,z\in \Z^d$ and every $\beta,r\geq 0$.
\end{lemma}

(Note that this inequality is slightly stronger than that obtained by applying the tree-graph inequality to bound the right hand side of \cref{lem:disjoint_connections}.)

\begin{proof}[Proof of \cref{lem:disjoint_connections_triangle}]

Let $\omega$ and $\omega'$ be independent percolation configurations, and let $\tilde \omega$ be defined by setting $\tilde \omega(e) = \omega(e)$ for every edge $e$ touching the cluster $K_0$ of $0$ in $\omega$ and $\tilde \omega(e)=\omega'(e)$ for every other edge $e$, so that $\tilde \omega$ has the same law as $\omega$. If we let $K_y$ and $\tilde K_y$ denote the clusters of $y$ in $\omega$ and $\tilde \omega$ respectively then, on the event that $y\notin K_0$, we have that $K_y \subseteq \tilde K_y$ and that a point $z\in \tilde K_y$ belongs to $K_y$ if and only if there is an open path connecting $y$ to $z$ in $\tilde K_y$ that does not visit $K_0$. Thus, if $z \in \tilde K_y \setminus K_y$ then there must exist $b\in K_0$ such that $b$ is connected to $y$ and $z$ by disjoint paths in $\omega'$. Writing $\P$ for the joint law of the random variables we have just introduced, it follows by a union bound and the BK inequality that
\begin{align*}
\P(x \in K_0, z \in \tilde K_y \setminus K_y) &\leq \sum_{b\in\Z^d} \P(x,b \in K_0 \text{ and $b$ lies on a simple path connecting $0$ to $z$ in $\tilde K_y$}) \\&\leq \sum_{b\in \Z^d} \P_{\beta,r}(0\leftrightarrow x, 0 \leftrightarrow b)\P_{\beta,r}(y \leftrightarrow b) \P_{\beta,r}(b \leftrightarrow z) 
\\&\leq 
\sum_{a,b\in \Z^d} \P_{\beta,r}(0\leftrightarrow a)\P_{\beta,r}(a\leftrightarrow x)\P_{\beta,r}(a\leftrightarrow b)\P_{\beta,r}(y\leftrightarrow b)\P_{\beta,r}(b\leftrightarrow z),
\end{align*}
where the final estimate follows from the tree-graph inequality. The claim follows since
\begin{multline*}
\P_{\beta_c,r}(0\leftrightarrow x)\P_{\beta_c,r}(y \leftrightarrow z) - \P_{\beta_c,r}(0\leftrightarrow x \nleftrightarrow y \leftrightarrow z)
\\=
\P(x\in K_0, z\in \tilde K_y) - \P(x\in K_0, y\notin K_0, z\in K_y) = \P(x \in K_0, z \in \tilde K_y \setminus K_y)
\end{multline*}
by construction of the coupling.
\end{proof}

\begin{proof}[Proof of \cref{lem:E1_triangle}]
We can write
\begin{multline*}
|B_r|(\E_{\beta_c,r}|K|)^2 - \E_{\beta_c,r} \left[|K| \sum_{y\in B_r} \mathbbm{1}(y\notin K) |K_y|\right] 
\\= 
\sum_{y\in B_r} \sum_{x,z\in \Z^d} \left[\P_{\beta_c,r}(0\leftrightarrow x) \P_{\beta_c,r}(y\leftrightarrow z) - \P_{\beta_c,r}(0\leftrightarrow x \nleftrightarrow y\leftrightarrow z)\right],
\end{multline*}
and applying \cref{lem:disjoint_connections_triangle} we obtain that
\[
|B_r|(\E_{\beta_c,r}|K|)^2 - \E_{\beta_c,r} \left[|K| \sum_{y\in B_r} \mathbbm{1}(y\notin K) |K_y|\right]  \leq \sum_{y\in B_r}\nabla(0,y)(\E_{\beta_c,r}|K|)^2,
\]
which implies the claimed bound on $\cE_{1,r}$.
\end{proof}

The proof of \cref{prop:first_moment} will also make use of the following lemma on the preservation of logarithmic integrability under certain integral transformations.

\begin{lemma}
\label{lem:log_integrability_integration}
If $h:(0,\infty)\to \R$ is a logarithmically integrable error function then so are the functions $r^a \int_r^\infty s^{-a-1} h(s) \dif s$ and $r^{-a}\int_1^r s^{a-1} h(s)\dif s$ for every $a> 0$.
\end{lemma}

(Note that if $h$ is slowly varying then both integrals are in fact asymptotic to multiples of $h$ by \cref{lem:regular_variation_integration}.)

\begin{proof}[Proof of \cref{lem:log_integrability_integration}]
If $h$ is logarithmically integrable then we have by Tonelli's theorem that
\[
\int_1^\infty \left[r^{a} \int_r^\infty |h(s)|s^{-a-1} \dif s\right] \frac{\dif r}{r} = \int_1^\infty |h(s)|s^{-a-1} \int_1^s r^{a-1} \dif r \dif s \leq \frac{1}{a} \int_1^\infty \frac{|h(s)|}{s} \dif s < \infty
\]
and similarly that
\[
\int_1^\infty \left[r^{-a} \int_1^r |h(s)|s^{a-1} \dif s\right] \frac{\dif r}{r} = \int_1^\infty |h(s)|s^{-a-1} \int_s^\infty r^{-a-1} \dif r \dif s = \frac{1}{a} \int_1^\infty \frac{|h(s)|}{s} \dif s < \infty
\]
as required. The fact that these functions also tend to zero as $r\to \infty$ follows easily from the assumption that $h(r)\to 0$ as $r\to \infty$.
\end{proof}

\begin{proof}[Proof of \cref{prop:first_moment}]
We have shown that $\cE_{1,r}$ tends to zero in all the claimed cases, with $\cE_{1,r}$ logarithmically integrable when $d>3\alpha$ or when $r^{-d}\sum_{x\in B_r}\nabla(0,x)$ is logarithmically integrable. Moreover, $r^{-d}\sum_{x\in B_r}\nabla(0,x)$ is logarithmically integrable whenever $\nabla_r$ is by a simple variation on \cref{lem:log_integrability_integration}. The analogous claims for $\overline{\cE}_{1,r}$ hold since $\cE_{0,r}$ is logarithmically integrable by our standing assumptions on $J$ and \cref{lem:ball_regularity}. Thus, the claimed first-order asymptotics follow from \eqref{eq:first_moment_E1_formula}, with the resulting error $\cH_r$ being logarithmically integrable whenever $\cE_{1,r}$ is logarithmically integrable by \cref{lem:log_integrability_integration}.
\end{proof}

It remains to prove \cref{lem:log_integrable_triangle}, which shows that \cref{prop:first_moment} applies under \eqref{HD}. 

\begin{proof}[Proof of \cref{lem:log_integrable_triangle}]
We can write $\nabla$ in terms of the function $\tau(x)=\P_{\beta_c}(0\leftrightarrow x)$ as $\nabla(0,x)=\tau*\tau*\tau(x)$. It is a standard fact (see \cite[Proposition 1.7]{MR1959796}) that if $|f(x)|\preceq (1+\|x\|)^{-d+a}$ and $|g(x)| \preceq (1+\|x\|)^{-d+b}$ with $a,b>0$ and $a+b<d$ then $|f*g(x)| \preceq (1+\|x\|)^{-d+a+b}$, and applying this twice yields in Case 2 that
\begin{equation}
  \nabla(0,x) \preceq \|x\|^{-d+6},
\label{eq:triangle_case_2}
\end{equation}
which is sufficient for logarithmic integrability of $\nabla_r$ in this case since $d>6$. For Case 3 we will need an analogous fact allowing for logarithmic corrections:
If $|f(x)|\preceq (1+\|x\|)^{-d+a} (\log (2+\|x\|))^{-\hat a}$ and $|g(x)| \preceq (1+\|x\|)^{-d+b} (\log (2+\|x\|))^{-\hat b}$ with $a,b>0$, $\hat a,\hat b \geq 0$, are such that either $a+b<d$ or $a+b=d$ and $\hat a + \hat b>1$ then
\begin{equation}
\label{eq:convolution_estimate_logs}
  |f*g(x)| \preceq_{a,b,\hat a,\hat b} \begin{cases} (1+\|x\|)^{-d+a+b} (\log (2+\|x\|))^{-\hat a -\hat b} & a+b <d \\
  (\log (2+\|x\|))^{-\hat a -\hat b+1} & a+b=d.
  \end{cases}
\end{equation}
To prove this estimate, we decompose $\Z^d$ into three sets: $A_1=A_1(x)=\{y\in \Z^d : \|y\| \leq \|x-y\| \text{ and } \|y\|\leq 2 \|x\| \}$,  $A_2=A_2(x)=\{y\in \Z^d : \|y-x\| < \|x\| \text{ and } \|y\|\leq 2 \|x\| \}$, and $A_3 =\{y\in \Z^d:\|y\|\geq 2\|x\|\}$. Writing $\langle x\rangle=2+\lceil \|x\|\rceil$ to avoid division by zero and writing the contribution of each of these sets to the convolution as a sum over the boundaries of boxes, we obtain that 
\begin{align*}
  |f*g(x)| &\preceq_{a,b,\hat a,\hat b} \langle x\rangle^{-d+b}(\log \langle x\rangle)^{-\hat b}\sum_{r=2}^{\langle x\rangle} r^{d-1}r^{-d+a} (\log r)^{-\hat a}
  \\&\hspace{1cm}+\langle x\rangle^{-d+a}(\log \langle x\rangle)^{-\hat a}\sum_{r=2}^{\langle x\rangle} r^{d-1}r^{-d+b} (\log r)^{-\hat b}
  +\sum_{r=\langle x\rangle}^{\infty} r^{d-1} r^{-d+a} r^{-d+b} (\log r)^{-\hat a -\hat b}
  \\
  &\asymp_{a,b,\hat a,\hat b} \langle x \rangle^{-d+a+b} (\log \langle x \rangle)^{-\hat a -\hat b} + \sum_{r=\langle x\rangle}^{\infty} r^{-d+a+b-1} (\log r)^{-\hat a -\hat b},
\end{align*}
which is easily seen to imply the claimed inequality \eqref{eq:convolution_estimate_logs}. Under Case 3 of \eqref{HD} we may apply this estimate twice to obtain that
\begin{equation}
\label{eq:triangle_case_3}
  \tau*\tau(x) \preceq \frac{\|x\|^{-d+4}}{(\log(1+\|x\|))^{2}} \qquad \text{ and } \qquad \nabla(0,x) \preceq \frac{1}{(\log(1+ \|x\|)^2}
\end{equation}
for every $x\in \Z^d\setminus \{0\}$.
\end{proof}

Before moving on, let us note that the consequences of \cref{prop:first_moment} for the hydrodynamic condition.

\begin{corollary}[The hydrodynamic condition]
\label{cor:HD_hydro}
If at least one of the three hypotheses of \eqref{HD} holds then \eqref{Hydro} also holds.
\end{corollary}

\begin{proof}[Proof of \cref{cor:HD_hydro}]
It follows from \cref{prop:first_moment} and the tree-graph inequality applied as in \cite[Eq. 6.99]{grimmett2010percolation} that
\[
  \P_{\beta_c,r}(|K|\geq n) \leq \frac{\sqrt{2}\E_{\beta_c,r}|K|}{n} \exp\left[-\frac{n}{4(\E_{\beta_c,r}|K|)^2}\right]
\]
for every $n\geq 1$. This implies by a union bound that
\[
  M_r \preceq (\E_{\beta_c,r}|K|)^2 \log r
\]
for every $r\geq 2$, 
which is sufficient to establish the hydrodynamic condition when $d>3\alpha$ by \cref{prop:first_moment} (since $2\alpha<(d+\alpha)/2$ when $d>3\alpha$).
On the other hand, we have by \cref{lem:moments_bounded_below_by_M} that
\[
M_r^2 \preceq r^{-d} \E_{\beta_c,r}|K \cap B_{2r}| \leq r^{-d} \sum_{x\in B_{2r}} \P_{\beta_c}(0\leftrightarrow x),
\]
which gives that $M_r=o(r^{(d+\alpha)/2})$ when hypotheses 2 or 3 of \eqref{HD} hold.
\end{proof}

\subsection{Second moment}

We now analyze the second moment $\E_{\beta_c,r}|K|^2$ under the same assumptions as \cref{prop:first_moment}. For this quantity, the different cases treated by that proposition can have different first-order asymptotics. Indeed, we will eventually prove in \cref{III-thm:critical_dim_moments_main} that this quantity has a logarithmic correction to $r^{3\alpha}$ scaling when $d<6$ and $\alpha=d/3$. 

\begin{prop}[Second moment asymptotics]
\label{prop:second_moment}
If any of the three conditions from \cref{prop:first_moment} hold, then $r\mapsto \E_{\beta_c,r}|K|^2$ is a regularly varying function of index $3\alpha$. If moreover $d>3\alpha$ or the triangle condition holds and $\nabla_r$ is logarithmically integrable then there exists a constant $A$ such that $\E_{\beta_c,r}|K|^2\sim Ar^{3\alpha}$. 
\end{prop}

The proof of \cref{prop:second_moment} will rely on \cref{lem:ODE_self_referential}, which we recall stated that functions satisfying $f'\sim a r^{-1} f$ must be regularly varying of index $a$ and moreover that if the relevant (multiplicative) error is logarithmically integrable then there exists a constant $A$ such that $f\sim Ar^a$. The fact that such an estimate might \emph{not} hold when the error is not logarithmically integrable is important as it allows for the logarithmic corrections to scaling at the critical dimension.

\begin{proof}[Proof of \cref{prop:second_moment}, Cases 1 and 2]
To begin, we rewrite the $p=2$ case of the equality \eqref{lem:moment_derivative} as
\begin{align}
\frac{d}{dr}\E_{\beta_c,r}|K|^2 &= \beta_c |J'(r)| \E_{\beta_c,r}\left[\sum_{y\in B_r}\mathbbm{1}(y\notin K) |K||K_y|^2 \right] 
+  2\beta_c |J'(r)| \E_{\beta_c,r}\left[\sum_{y\in B_r}\mathbbm{1}(y\notin K) |K|^2|K_y| \right] 
\nonumber\\ &= 3\beta_c |J'(r)| \E_{\beta_c,r}\left[\sum_{y\in B_r}\mathbbm{1}(y\notin K) |K||K_y|^2 \right] 
\nonumber\\&= 3(1-\cE_{2,r}) \beta_c |J'(r)| |B_r| \E_{\beta_c,r}|K|\E_{\beta_c,r}|K|^2
\label{eq:second_moment_ODE_error_def}
\end{align}
where the error term $\cE_{2,r}$ is defined by
\begin{multline*}
\cE_{2,r} := 
\frac{|B_r|\E_{\beta_c,r}|K|\E_{\beta_c,r}|K|^2 - \E_{\beta_c,r} \left[|K| \sum_{y\in B_r} \mathbbm{1}(y\notin K) |K_y|^2\right]}{|B_r| \E_{\beta_c,r}|K|\E_{\beta_c,r}|K|^2}
\\
=
\frac{|B_r|\E_{\beta_c,r}|K|\E_{\beta_c,r}|K|^2 - \E_{\beta_c,r} \left[|K| \sum_{y\in B_r} |K_y|^2\right]
+\E_{\beta_c,r}[|K|^3|K\cap B_r|]
}{|B_r| \E_{\beta_c,r}|K|\E_{\beta_c,r}|K|^2}.
\end{multline*}
Recalling the definitions of $\cH_r$ and $\cE_{0,r}$ from \cref{subsec:first_moment}, we also define $\overline{\cE}_{2,r}$ via
\[
(1-\overline{\cE}_{2,r}) := (1-\cE_{2,r})\frac{|J'(r)||B_r|}{r^{-\alpha-1}} \frac{\E_{\beta_c,r}|K|}{\frac{\alpha}{\beta_c}r^\alpha} = (1-\cE_{2,r})(1+\cH_r)(1-\cE_{0,r})
\]
so that
\begin{equation}
\label{eq:barE11_def}
\frac{d}{dr}\E_{\beta_c,r}|K|^2
= 3 \alpha (1-\overline{\cE}_{2,r}) r^{-1} \E_{\beta_c,r}|K|^2.
\end{equation}
In order to apply \cref{lem:ODE_self_referential}, we need to prove that $\lim_{r\to \infty}\overline{\cE}_{2,r}=0$ under each of the first two assumptions of \cref{prop:first_moment}, and is logarithmically integrable under the stronger assumption that $d>3\alpha$. (The third assumption, involving the triangle condition, is treated below.) Since the relevant properties have already been established for $\cE_{0,r}$ and $\cH_r$ in \cref{lem:ball_regularity} and \cref{prop:first_moment}, it suffices to bound $\cE_{2,r}$. This is relatively straightforward since we already have good bounds on $\E_{\beta_c,r}|K|$ in each case, and do not need to start from scratch using bootstrapping arguments etc.\ as in the previous section. 
As before, it follows from \cref{lem:BK_disjoint_clusters_covariance} that
\begin{equation}
\label{eq:E2_bound}
0\leq \cE_{2,r} \leq \frac{\E_{\beta_c,r}[|K|^3|K\cap B_r|]
}{|B_r| \E_{\beta_c,r}|K|\E_{\beta_c,r}|K|^2}.
\end{equation}
The right hand side of this inequality will be bounded in a slightly different way in each of the three cases.
Cases 1 and 2 will both apply the following slight generalization of the tree-graph inequality, which we prove after concluding the proof of the proposition. The bound on moments of order $p\geq 3$ provided by this lemma is a significant improvement to the usual tree-graph inequality when $\E_\beta|K|^2$ is much smaller than $(\E_\beta|K|)^3$.

\begin{lemma}[Generalized tree-graph inequality]
\label{lem:generalized_tree_graph}
Consider percolation on any transitive weighted graph. For each integer $p\geq 2$ the inequality
\[
\E_\beta|K|^p \leq (2p-3)!! (\E_\beta|K|^2)^{k} (\E_\beta|K|)^{2p-1-3k}
\]
holds for every $0\leq k \leq \lceil \frac{p-1}{2}\rceil$ and $\beta>0$.
\end{lemma}


\noindent \textbf{Case 1: $d>3 \alpha$.} In this case we use the generalized tree-graph inequality of \cref{lem:generalized_tree_graph} with $k=1$ to bound
\[
\E_{\beta_c,r}[|K|^3|K\cap B_r|] \leq \E_{\beta_c,r}[|K|^4] \preceq \E_{\beta_c,r}|K|^2 \left(\E_{\beta_c,r}|K|\right)^4,
\]
so that
\[
0\leq \cE_{2,r} \preceq \frac{\left(\E_{\beta_c,r}|K|\right)^3}{|B_r|} \preceq r^{3\alpha-d}
\]
is logarithmically integrable as required.

\medskip

\noindent \textbf{Case 2: $d=3\alpha$ $+$ \eqref{Hydro}.} In this case we use Cauchy-Schwarz and the generalized tree-graph inequality of \cref{lem:generalized_tree_graph} with $p=6$ and $k=2$ to bound
\begin{equation}
\label{eq:E2_bound_hydro_1}
\E_{\beta_c,r}[|K|^3|K\cap B_r|] \leq \sqrt{\E_{\beta_c,r}|K|^6\E_{\beta_c,r}|K\cap B_r|^2} \preceq \sqrt{(\E_{\beta_c,r}|K|^2)^2 \left(\E_{\beta_c,r}|K|\right)^5 r^{\alpha}M_r},
\end{equation}
where we used the universal tightness theorem via \cref{cor:universal_tightness_moments} and the two-point upper bound of \cref{thm:two_point_spatial_average_upper} to bound $\E_{\beta_c,r}|K\cap B_r|^2 \preceq r^\alpha M_r$. Using this together with the hydrodynamic condition we obtain that
\begin{equation}
\label{eq:E2_bound_hydro_2}
\cE_{2,r} = o\left(r^{2\alpha-d + \frac{d+\alpha}{4}}  \right) = o(1),
\end{equation}
where the final equality holds since $d=3\alpha$.
%
%
\end{proof}

We now owe the reader the proof of the generalized tree-graph inequality \cref{lem:generalized_tree_graph}.

\begin{proof}[Proof of \cref{lem:generalized_tree_graph}]
We write $\E=\E_\beta$ to lighten notation. 
Recall  that the usual tree-graph bound is proven by taking a union bound over possible combinatorial structures of a tree connecting the origin to $p$ labelled vertices, then using the BK inequality to bound the probability of each relevant event as a product over connection probabilities between vertices in the tree. The $(2p-1)!!$ term arises as the number of binary trees rooted at a leaf with $p$ labelled non-root leaves. In the second step, one can instead bound the relevant probability by first partitioning the internal vertices of the binary tree into bipartite classes: since there are $p-1$ internal vertices, the larger of these bipartite classes much have size at least $\lceil (p-1)/2\rceil$. Taking a subset of this class of size $k$, this lets us partition the edges of the tree into the $3$-stars around each vertex in this set and the remaining edges of the tree that are not adjacent to a vertex in this set. 
When we apply the BK inequality and sum over relevant vertices (the vertices in the distinguished set are \emph{not} summed over), we get $k$ copies of the three-point function and $2p-1-3k$ copies of the two-point function, leading to the desired bounnd.
For example, if $x,y,z$ are three points, then there must exist a point $w$ such that at least one of the events $\{0\leftrightarrow x,w\} \circ \{w\leftrightarrow y\}\circ \{w\leftrightarrow z\}$, 
$\{0\leftrightarrow y,w\} \circ \{w\leftrightarrow x\}\circ \{w\leftrightarrow z\}$,
or 
$\{0\leftrightarrow z,w\} \circ \{w\leftrightarrow x\}\circ \{w\leftrightarrow y\}$
holds, and applying a union bound and the BK inequality yields that
\[
\E|K|^3 = \sum_{x,y,z}\P(0\leftrightarrow x,y,z) \leq 3\sum_{x,y,z,w} \P(0\leftrightarrow x,w) \P(w\leftrightarrow y)\P(w\leftrightarrow z) = 3\E|K|^2(\E|K|)^2.
\]
Diagrammatically,
\[
\E|K|^3 \leq 3 \begin{array}{l}\includegraphics{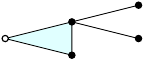}\end{array} = 3\E|K|^2(\E|K|)^2,
\]
where the white vertex denotes the origin, black vertices are summed over, the blue shaded triangle represents a three-point function, and each of the two lines not incident to a triangle represent copies of the two-point function.
Similarly, we have both that 
\[
\E|K|^4 \leq 12 \begin{array}{l}\includegraphics{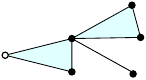}\end{array} + 3
\begin{array}{l}\includegraphics{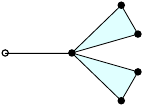}\end{array} = 15 (\E|K|^2)^2 \E|K|
\]
and
\[
\E|K|^4 \leq 12 \begin{array}{l}\includegraphics{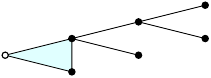}\end{array} + 3
\begin{array}{l}\includegraphics{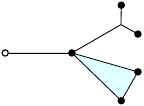}\end{array} = 15 \E|K|^2 (\E|K|)^4.
\]
(We will later see how to get sharper upper bounds on the higher moments in the setting of \cref{prop:first_moment}.)
\end{proof}

It remains to analyze the second moment under the assumption that the triangle condition holds. We will prove the estimates we need for general moments so that we can reuse the same estimates in the next section. We begin by proving the following extension of \cref{lem:disjoint_connections_triangle} to higher moments.

\begin{lemma}
\label{lem:disjoint_connections_triangle_second} 
Consider percolation on a transitive weighted graph at some parameter $\beta>0$. There exist universal integer constants $C_{p,q,\ell}$ such that the estimate
\begin{equation*}
0\leq \E|K|^p \E |K|^q-\E\left[|K_x|^p |K_y|^q \mathbbm{1}(x \nleftrightarrow y)\right] \leq \sum_{\ell=3}^{p+q+1} C_{p,q,\ell} T^\ell(x,y) (\E|K|)^{2p+2q+1-\ell}
\end{equation*}
holds for every pair of vertices $x$ and $y$, where $T^\ell$ denotes the $\ell$-fold power of the two-point matrix $T(x,y):=\P(x\leftrightarrow y)$.
\end{lemma}

\begin{proof}[Proof of \cref{lem:disjoint_connections_triangle_second}]
The fact that the quantity we wish to estimate is non-negative follows immediately from the BK inequality.
Let $\omega$, $\omega'$, $\tilde \omega$, $K_x$ and $\tilde K_y$ be as in the proof of \cref{lem:disjoint_connections_triangle}. As before, on the event that $y\notin K_x$, we have that $K_y \subseteq \tilde K_y$ and that a point $z\in \tilde K_y$ belongs to $K_y$ if and only if there is an open path connecting $y$ to $z$ in $\tilde K_y$ that does not visit $K_x$. Thus, we can write
\begin{align*}
&\E|K|^p \E |K|^q-\E\left[|K_x|^p |K_y|^q \mathbbm{1}(x \nleftrightarrow y)\right] 
\\
&\hspace{0.5cm}= \sum_{x_1,\ldots,x_p} \sum_{y_1,\ldots,y_q} \P(x_1,\ldots,x_p \in K_x, y_1,\ldots,y_q\in \tilde K_y, \text{ and either $y\in K_x$ or at least one $y_i\notin K_y$})\\
&\hspace{0.5cm} \leq 
\sum_w \sum_{x_1,\ldots,x_p} \sum_{y_1,\ldots,y_q} \P(w,x_1,\ldots,x_p \in K_x, y_1,\ldots,y_q\in \tilde K_y,
\\&\hspace{6cm} \text{$w$ \text{on a simple open path connecting $y$ to $\{y_1,\ldots,y_q\}$ in $\tilde K_y$}}).
\end{align*}
Since the events defining the conditions on $K_x$ and $\tilde K_y$ in the last line are independent, we obtain the bound
\begin{align*}
&\E|K|^p \E |K|^q-\E\left[|K_x|^p |K_y|^q \mathbbm{1}(x \nleftrightarrow y)\right]  \leq 
\sum_w \sum_{x_1,\ldots,x_p} \sum_{y_1,\ldots,y_q} \P(w,x_1,\ldots,x_p \in K_x)\\
&\hspace{2.5cm}\cdot\P(y_1,\ldots,y_q\in \tilde K_y,
\text{$w$ \text{on a simple open path connecting $y$ to $\{y_1,\ldots,y_q\}$ in $\tilde K_y$}}).
\end{align*}
Expanding each of these probabilities out as in the proof of the tree-graph inequality yields the claim: We obtain a bound in terms of tree diagrams with two pinned vertices at $x$ and $y$, with $2p+2q+1$ edges in total, some number $3\leq \ell \leq p+q+1$ of which belong to the path connecting $x$ and $y$.
For example, when $p=2$ and $q=1$ we can express this bound diagrammatically as
\begin{align*}
\E|K|^2\E|K| - \E[|K_x|^2|K_y|\mathbbm{1}(x\nleftrightarrow y)]
&\leq \begin{array}{l}\includegraphics{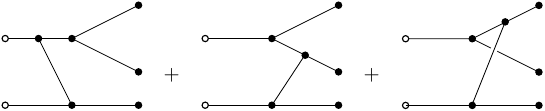}\end{array}\\
&\leq T^3(x,y) (\E|K|)^4 + 2 T^4(x,y) (\E|K|)^3.
\end{align*}
In each case, $C_{p,q,\ell}$ counts the number of diagrams in the bound in which the path connecting $x$ to $y$ has length $\ell$. 
\end{proof}

To make use of this lemma, we will need to have a reasonable bound on how the higher polygon diagrams blow up as $r\to\infty$.

\begin{lemma}
\label{lem:higher_polygon_diagrams} Suppose that $\nabla<\infty$. Then there exist positive constants $c$ and $C$ such that the polygon diagram $T^\ell_{\beta_c,r}$ can be bounded
\[
\sup_x T^\ell_{\beta_c,r}(0,x) = T^\ell_{\beta_c,r}(0,0)\leq C \left(\nabla_{r^{c}} + r^{-c}\right) (\E_{\beta_c,r}|K|)^{\ell-3}
\]
for each integer $\ell\geq 3$.
\end{lemma}

\begin{proof}[Proof of \cref{lem:higher_polygon_diagrams}]
Write $T=T_{\beta_c,r}$ and $\E=\E_{\beta_c,r}$. We have by translation-invariance and positive-definiteness of $T$ that $T^\ell$ is maximized on the diagonal, so that it suffices to bound $T^\ell(0,0)$. 
Our assumptions guarantee that there exists a constant $C$ such that
\begin{align*}
T^\ell(0,0)&= \sum_{y} \nabla(0,y) T^{\ell-3}(y,0)= \sum_{\|y\| \geq R} \nabla(0,y) T^{\ell-3}(y,0)
+
\sum_{\|y\|< R} \nabla(0,y) T^{\ell-3}(y,0)
\\
&\leq \nabla_R (\E|K|)^{\ell-3}
+
\nabla |B_R| (\E|K|)^{\ell-4} \leq \left(\nabla_R + C\frac{R^d}{r^\alpha}\right)(\E|K|)^{\ell-3}
\end{align*}
for every $r>0$, where in the last line we used that $\E|K| \succeq r^\alpha$ and that
\[\sum_{\|y\|< R} T^{\ell-3}(y,0) = \sum_{x_1,\ldots,x_{\ell-4}} T(0,x_1)\prod_{i=1}^{\ell-5} T(x_i,x_{i+1})\sum_{\|y\|<R} T(x_{\ell-4},y) \leq |B_R| (\E|K|)^{\ell-4}.\]
The claim follows by taking, say, $R=r^{\alpha/2d}$, so that $R^d/r^\alpha$ is polynomially small in $r$ as required.
\end{proof}

We will apply this estimate with the aid of the following elementary lemma.

\begin{lemma}
\label{lem:h(r^c)_log_integrable}
If $h(r)$ is logarithmically integrable then so is $h(Cr^c)$ for every $c,C>0$.
\end{lemma}


\begin{proof}[Proof of \cref{lem:h(r^c)_log_integrable}]
Using the change of variables $s=Cr^c$, so that $\frac{dr}{r}=\frac{1}{cC^{1-c}}\cdot \frac{ds}{s}$, yields that if $h$ is logarithmically integrable then
\[\int_1^\infty \frac{|h(Cr^c)|}{r} \dif r = \frac{1}{cC^{1-c}}\int_C^\infty \frac{|h(s)|}{s} \dif s<\infty\]
as required.
\end{proof}

We are now ready to prove \cref{prop:second_moment} in the case that the triangle condition holds.

\begin{proof}[Proof of \cref{prop:second_moment}, Case 3]
As before, the desired bounds on $\cE_{1,r}$ and $\cH_r$ have already been established in \cref{prop:first_moment}, so that it suffices to show that $\cE_{2,r}\to 0$ as $r\to \infty$ and that $\cE_{2,r}$ is logarithmically integrable whenever $\nabla_r$ is.
It follows from \cref{lem:disjoint_connections_triangle_second} and \cref{lem:higher_polygon_diagrams} that there exists a positive constant $c$ such that
\[
\cE_{2,r} \preceq \frac{(\nabla_{r^c} + r^{-c}) (\E_{\beta_c,r}|K|)^{4}}{\E_{\beta_c,r}|K|\E_{\beta_c,r}|K|^2}.
\]
If the triangle condition holds then $(\nabla_{r^c} + r^{-c})$ converges to zero as $r\to \infty$, and is logarithmically integrable whenever $\nabla_r$ is by \cref{lem:h(r^c)_log_integrable}. Thus, it suffices to prove that if the triangle condition holds then $\E_{\beta_c,r}|K|^2 \succeq (\E_{\beta_c,r}|K|)^3\succeq r^{3\alpha}$ as $r\to \infty$, the complementary upper-bound on the second moment following from the tree-graph inequality. Since the triangle condition implies the volume-tail upper bound $\P_{\beta_c}(|K|\geq n)\preceq n^{-1/2}$ \cite{MR1127713,HutchcroftTriangle}, we can bound the first moment from above by
\begin{align*}
\E_{\beta_c,r}|K| &= \sum_{n=1}^\infty \P_{\beta_c,r}(|K|\geq n) \leq \sum_{n=1}^N \P_{\beta_c}(|K|\geq n) + \sum_{n=N}^\infty \frac{\E_{\beta_c,r}|K|^2}{n^2} 
\\ &\preceq \sum_1^N n^{-1/2} + N^{-1}\E_{\beta_c,r}|K|^2  \preceq N^{1/2} + N^{-1}\E_{\beta_c,r}|K|^2
\end{align*}
for every $r>0$ and integer $N\geq 1$. Optimizing by taking $N=\lceil (\E_{\beta_c,r}|K|^2)^{2/3} \rceil$ yields the claim.
\end{proof}

\subsection{Higher moments}

We now analyze the higher moments in terms of the first and second moment. We will see that, due to the different form of the relevant asymptotic ODEs, there is no further distinction between the various cases treated by \cref{prop:first_moment,prop:second_moment}: We can always obtain a first-order asymptotic expression for higher moments in terms of the first and second moments whenever one of the hypotheses of \cref{prop:first_moment} holds, whether or not the errors in \cref{prop:first_moment,prop:second_moment} are logarithmically integrable. In other words, under the hypotheses of \cref{prop:first_moment}, any logarithmic corrections that occur in the asymptotics of the cluster volume moments must appear in the second moment, and there is no further generation of ``new'' logarithmic corrections at higher moments.

\begin{prop}[Higher moment asymptotics]
\label{prop:higher_moments}
If any of the three conditions from \cref{prop:first_moment} holds, then for each integer $p\geq 1$ the asymptotic estimate
\[
\E_{\beta_c,r}|K|^p \sim (2p-3)!! \left(\frac{\E_{\beta_c,r}|K|^2}{\E_{\beta_c,r}|K|}\right)^{p-1}\E_{\beta_c,r}|K|
\]
holds as $r\to \infty$. In particular, $\E_{\beta_c,r}|K|^p$ is regularly varying of index $(2p-1)\alpha$.
\end{prop}

The proof of \cref{prop:higher_moments} will rely on the following analytic lemma about approximate ODEs $f'\sim r^{-1}(a f+h)$ in which the term $h$, which we often refer to as the \textbf{driving term}, is regularly varying of index strictly larger than $a$: Such ODEs have good stability properties, and the function $f$ is forced to scale as a specific constant multiple of $h(r)$. (In particular, solutions to the asymptotic ODE are asymptotic to solutions to the exact ODE.) We will study the cases in which $h$ has index less than or equal to $a$ in \cref{sec:superprocesses} and the case in which $f$ is of uncertain sign in \cref{lem:signed_ODE_analysis}.

\begin{lemma}[$f'\sim r^{-1}(af+h)$ with $h$ non-negligible]
\label{lem:ODE_with_driving_term}
Let $a>0$, let $h:(0,\infty)\to (0,\infty)$ be a measurable, regulary varying function of index $b>0$, and suppose that $f$ is a differentiable function such that
\[
f'(r)\sim \frac{af(r)+h(r)}{r}
\]
as $r\to \infty$. 
If $b>a$ then 
\[f(r)\sim \frac{h(r)}{b-a}\] as $r\to\infty$. In particular, $f$ is regularly varying of index $b$.
\end{lemma}

\begin{proof}[Proof of \cref{lem:ODE_with_driving_term}]
Since $f'(r)\sim ar^{-1}f(r)+r^{-1}h(r)$, there exists $r_0<\infty$ and a (not necessarily positive) measurable function $\delta:[r_0,\infty)\to \R$ with $|\delta_r|\to 0$ as $r\to \infty$ such that $f'-(1-\delta_r)ar^{-1}f(r)=(1-\delta_r)r^{-1}h(r)$ for every $r\geq r_0$. Since $h$ is measurable and regularly varying, we may take $r_0$ sufficiently large that $h$ is locally integrable on $[r_0,\infty)$. Recognizing $f'-(1-\delta_r)ar^{-1}f(r)=(1-\delta_r)r^{-1}h(r)$  as a first-order linear ODE, we can write down the explicit solution
\begin{multline*}
f(r) = \exp\left[\int_{r_0}^r (1-\delta_s)as^{-1} \dif s\right] f(r_0) \\+ \exp\left[\int_{r_0}^r (1-\delta_s)as^{-1} \dif s\right]\int_{r_0}^r \frac{h(s)}{s}\exp\left[-\int_{r_0}^s (1-\delta_t)at^{-1} \dif t\right] \dif s.
\end{multline*}
Since $\exp \int_{r_0}^r (1-\delta_s)as^{-1} \dif s$ is regularly varying of index $a$, the integrand $s^{-1}h(s)\exp[-\int_{r_0}^s (1-\delta_t)at^{-1} \dif t]$ is regularly varying of index $b-a-1>-1$. Applying \cref{lem:regular_variation_integration}, this implies that
\[
  \int_{r_0}^r \frac{h(s)}{s}\exp\left[-\int_{r_0}^s (1-\delta_t)at^{-1} \dif t\right] \dif s \sim \frac{1}{b-a} h(r)\exp\left[-\int_{r_0}^r (1-\delta_t)at^{-1} \dif t\right],
\]
and the claim follows easily since the other term is of lower order. \qedhere

\end{proof}

It will also be useful to have the following inequality version of \cref{lem:ODE_with_driving_term}.




\begin{lemma}
\label{lem:ODE_with_driving_term_Gronwall}
Let $a,b>0$, let $h:(0,\infty)\to(0,\infty)$ be a measurable function that is regularly varying of index $b$, and suppose that $f$ is a positive function such that
\[
f'(r)\leq (1+o(1))\frac{af(r)+h(r)}{r}
\]
as $r\to \infty$.
If $b>a$ then 
\[f(r) \leq (1+o(1)) \frac{h(r)}{b-a}\] as $r\to\infty$. 
\end{lemma}

\begin{proof}[Proof of \cref{lem:ODE_with_driving_term_Gronwall}]
Gr\"onwall's lemma \cite[1.III]{walter2012differential} states that if $f$ is differentiable and satisfies the first-order linear differential inequality
\[
f'(r) \leq g(r)f(r)+h(r)
\]
for every $r\geq r_0$ and
for some locally integrable functions $g$ and $h$ then
\[
f(r) \leq \exp\left[\int_{r_0}^r g(s)\dif s\right] \left(f(r_0)+\int_{r_0}^r  \exp\left[-\int_{r_0}^s g(t)\dif t\right] h(s)\dif s\right)
\]
for every $r\geq r_0$. Using this inequality, the proof of the lemma follows by exactly the same calculation as the proof of \cref{lem:ODE_with_driving_term}. \qedhere

\end{proof}

The proof will also rely on the following identity for double factorials. This identity is closely related to that of \cite[Lemma 4.18]{hutchcroft2022critical}, and is proven using a similar method. See \cite{MR2924154,callan2009combinatorial} for broader discussions of identities of this kind.

\begin{lemma} The identity
\label{lem:double_fun}
$\sum_{k=0}^{n-1}\binom{n}{k}(2k-1)!!(2n-2k-3)!!
= (2n-1)!!$
holds for every $n\geq 2$.
\end{lemma}

\begin{proof}[Proof of \cref{lem:double_fun}] It follows from the generalized binomial theorem that the exponential generating function of $(2n-1)!!$ and $(2n-3)!!$ are given by
\[\sum_{n =0}^\infty \frac{(2n-1)!!}{n!}x^n = \frac{1}{\sqrt{1-2x}} \qquad \text{ and } \qquad \sum_{n =1}^\infty \frac{(2n-3)!!}{n!}x^n = 1-\sqrt{1-2x}. \]
 The product formula for exponential generating functions (Rule 3$'$ of \cite{MR2172781}) therefore yields that
\begin{multline*}
\sum_{n=1}^\infty \frac{x^n}{n!} \sum_{k=0}^{n-1} \binom{n}{k} (2k-1)!!(2n-2k-3)!!
=\left(\sum_{n =0}^\infty \frac{(2n-1)!!}{n!}x^n\right)\left(\sum_{n =1}^\infty \frac{(2n-3)!!}{n!}x^n\right) 
\\= \frac{1-\sqrt{1-2x}}{\sqrt{1-2x}} = \sum_{n =0}^\infty \frac{(2n-1)!!}{n!}x^n-1,
\end{multline*}
and the claim follows by comparing coefficients.
\end{proof}

\begin{proof}[Proof of \cref{prop:higher_moments}]
We lighten notation by writing $\E_r=\E_{\beta_c,r}$.
 The cases $p=1$ and $p=2$ (in which the asymptotic formula is vacuous but the claim concerning regular variation is not) follow from \cref{prop:first_moment} and \cref{prop:second_moment} respectively, so it suffices to consider $p\geq 3$. 
We will prove the claim by a double induction on $p$, in which we first prove that
\begin{equation}
\label{eq:higher_moments_upper}
\E_{\beta_c,r}|K|^p \leq (1+o(1))(2p-3)!! \left(\frac{\E_{\beta_c,r}|K|^2}{\E_{\beta_c,r}|K|}\right)^{p-1}\E_{\beta_c,r}|K|
\end{equation}
and then prove the asymptotic formula by inducting on $p$ a second time.
Suppose to this end that $p\geq 3$ and that the claimed upper bound \eqref{eq:higher_moments_upper}  has been established for all smaller $p$. 
We have by \cref{lem:moment_derivative} that
\begin{multline*}
\frac{d}{dr}\E_{r}|K|^p = \beta |J'(r)| \sum_{\ell=0}^{p-1}\binom{p}{\ell}\E_{r}\left[|K|^{\ell+1}\sum_{y\in B_r}\mathbbm{1}(y\notin K)|K_y|^{p-\ell} \right]
\\=
(p+1)\beta |J'(r)| \E_{r}\left[|K|^{p}\sum_{y\in B_r}\mathbbm{1}(y\notin K)|K_y| \right]
+ \beta |J'(r)| \sum_{\ell=1}^{p-2}\binom{p}{\ell}\E_{r}\left[|K|^{\ell+1}\sum_{y\in B_r}\mathbbm{1}(y\notin K)|K_y|^{p-\ell} \right],
\end{multline*}
where in the second line we used that the expectations appearing in the terms indexed by $\ell=0$ and $\ell=p-1$ are equal by the mass-transport principle. Using \cref{lem:BK_disjoint_clusters_covariance} to bound each of the expectations appearing on the right hand side from above, it follows that
\begin{equation*}
\frac{d}{dr}\E_{r}|K|^p \leq 
(p+1)\beta_c |J'(r)| |B_r| \E_r|K| \E_r|K|^{p}
+ \beta_c |J'(r)| |B_r|\sum_{\ell=1}^{p-2}\binom{p}{\ell}\E_r|K|^{\ell+1}\E_r|K|^{p-\ell},
\end{equation*}
and hence by \cref{prop:first_moment} and the induction hypothesis that
\begin{multline*}
\frac{d}{dr}\E_{r}|K|^p \leq (1+o(1)) 
(p+1) \alpha r^{-1} \E_r|K|^{p}
\\+ (1+o(1))\alpha \left[ \sum_{\ell=1}^{p-2}\binom{p}{\ell}(2\ell-1)!!(2p-2\ell-3)!!\right] r^{-1} \left( \frac{\E_r|K|^2}{\E_r|K|}\right)^{p-1} \E_r|K|.
\end{multline*}
Using \cref{lem:double_fun} to simplify the constant on the right hand side yields that
\[
\sum_{\ell=1}^{p-2}\binom{p}{\ell}(2\ell-1)!!(2p-2\ell-3)!! = (2p-1)!! - (2p-3)!! - p (2p-3)!! = (p-2)\cdot(2p-3)!!
\]
and hence that
\begin{equation*}
\frac{d}{dr}\E_{r}|K|^p \leq \frac{1+o(1)}{r} \left[
(p+1) \alpha  \E_r|K|^{p}
+ \alpha (p-2)(2p-3)!!  \left( \frac{\E_r|K|^2}{\E_r|K|}\right)^{p-1} \E_r|K|\right].
\end{equation*}
Since the driving term is regularly varying of index $(2p-1)\alpha>(p+1)\alpha$, we may apply \cref{lem:ODE_with_driving_term_Gronwall} to deduce that
\[
\E_{r}|K|^p \leq (1+o(1)) \frac{\alpha(p-2)(2p-3)!!}{\alpha(p-2)} \left( \frac{\E_r|K|^2}{\E_r|K|}\right)^{p-1} \E_r|K|
\]
as $r\to \infty$. Noting that the two $\alpha(p-2)$ factors cancel completes the induction step for the upper bound \eqref{eq:higher_moments_upper}.

\medskip

We now perform a second induction on $p$ to establish the claimed first-order asymptotic estimate on $\E_{r}|K|^p$. (Note that if $d>3\alpha$ or the triangle condition holds with $\nabla_r$ logarithmically integrable then the following argument is capable of proving the first-order asymptotic estimate directly, without the aid of the upper bound. Indeed, in these cases one can easily get bounds of the correct order on $\E_{r}|K|^p$ using the tree-graph inequality. The case $d=3\alpha$ is more delicate and will make use of the upper bound.)
To proceed, we need to prove  either that
\begin{equation}
\label{eq:higher_moment_decorrelation1}
\E_r\left[|K|^{\ell+1}\sum_{y\in B_r}\mathbbm{1}(y\notin K)|K_y|^{p-\ell} \right] \sim |B_r| \E_r|K|^{\ell+1} \E_r|K|^{p-\ell} 
\end{equation}
or that
\begin{equation}
\label{eq:higher_moment_decorrelation2}
 \E_r|K|^{\ell+1} \E_r|K|^{p-\ell}  - \frac{1}{|B_r|}\E_r\left[|K|^{\ell+1}\sum_{y\in B_r}\mathbbm{1}(y\notin K)|K_y|^{p-\ell} \right] = o\left(\left(\frac{\E_r|K|^2}{\E_r|K|}\right)^{p-1} (\E_r|K|)^2\right)
\end{equation}
for every $0\leq \ell \leq p-1$. Indeed, once this is done it will follow from the induction hypothesis that
\begin{multline}
\label{eq:either_estimate_suffices}
\frac{d}{dr}\E_{r}|K|^p \sim (p+1)\alpha r^{-1} \E_{r}|K|^p \\ + \alpha r^{-1} \left(\frac{\E_{r}|K|^2}{\E_{r}|K|}\right)^{p-1} \E_{r}|K| \sum_{\ell=1}^{p-2} \binom{p}{\ell} (2\ell-1)!! (2p-2\ell-3)!!.
\end{multline}
(Either estimate \eqref{eq:higher_moment_decorrelation1} or \eqref{eq:higher_moment_decorrelation2} is sufficient to prove this estimate since we can absorb the errors in approximating the relevant expectations into either the first or second term on the right hand side of \eqref{eq:either_estimate_suffices}.) Once this is established, we can conclude the induction step as we did with the upper bound but using \cref{lem:ODE_with_driving_term} instead of \cref{lem:ODE_with_driving_term_Gronwall}.

\medskip

It remains to prove one of the two decorrelation estimates \eqref{eq:higher_moment_decorrelation1} or \eqref{eq:higher_moment_decorrelation2} in each relevant case. As usual, we will take a slightly different approach in each case.

\medskip

\noindent
\textbf{Case 1: $d>3\alpha$.} 
First note that, since $\E_{r}|K|\asymp r^\alpha$ and $\E_{r}|K|^2 \asymp r^{3\alpha}$, it follows from the tree-graph inequality and H\"older's inequality that $\E_{r}|K|^p \asymp_p r^{(2p-1)\alpha}$ for every integer $p\geq 1$.
It follows from \cref{lem:BK_disjoint_clusters_covariance} that
\[
 0\leq |B_r| \E_{r}|K|^{\ell+1}\E_{r}|K|^{p-\ell} - \E_{r}\left[|K|^{\ell+1}\sum_{y\in B_r}\mathbbm{1}(y\notin K)|K_y|^{p-\ell} \right] \leq \E_{r} \left[|K|^{p+1} |K \cap B_r| \right].
\]
Using the tree-graph inequality we can bound
\[\E_{r} \left[|K|^{p+1} |K \cap B_r| \right] \leq \E_{r} |K|^{p+2} \preceq_p (\E_{r}|K|)^{2p+3} \asymp_p r^{(2p+3)\alpha},\]
and the claimed asymptotics follow since
\[
|B_r| \E_{r}|K|^{\ell+1}\E_{r}|K|^{p-\ell} \asymp_p r^{d+(2\ell+2-1)\alpha+(2p-2\ell-1)\alpha}=r^{d+2p\alpha}
\]
and $d+2p\alpha > (2p+3)\alpha$ when $d>3\alpha$.

\medskip

\noindent
\textbf{Case 2: $d=3\alpha$ $+$ \eqref{Hydro}.} 
In this case we can use Cauchy-Schwarz to bound  
\[
\E_{r}\left[|K|^{p+1} |K \cap B_r| \right] \leq \sqrt{\E_r|K|^{2p+2} \E_r [|K\cap B_r|^2]}.
\]
Using \eqref{eq:higher_moments_upper} to bound $\E_r|K|^{2p+2}$ and the universal tightness theorem as in \cref{cor:universal_tightness_moments} to bound $\E_r |K\cap B_r|^2 \preceq M_r \E_r |K\cap B_r| \leq M_r \E_r |K|$ yields that
\[
\E_{r}\left[|K|^{p+1} |K \cap B_r| \right] \preceq \left[ \left(\frac{\E_r|K|^2}{\E_r|K|}\right)^{2p+1}(\E_r|K|)^2 M_r \right]^{1/2}=o\left(r^{3\alpha} r^{\frac{d+\alpha}{4}} \left(\frac{\E_r|K|^2}{\E_r|K|}\right)^{p-1} \E_r|K|\right),
\]
where we used the tree-graph bound $\E_r|K|^2/\E_r|K| \preceq (\E_r|K|)^2 \preceq r^{2\alpha}$ when reducing the power of $\E_r|K|^2/\E_r|K|$ from $(2p+1)/2$ to $p-1$ in the second estimate.
Since $d=3\alpha$ it follows that
\[
\E_{r}\left[|K|^{p+1} |K \cap B_r| \right]  = o\left(r^d \left(\frac{\E_r|K|^2}{\E_r|K|}\right)^{p-1} (\E_r|K|)^2\right)
\]
as required.

\medskip

\noindent \textbf{Case 3: The triangle condition.} We have by \cref{lem:disjoint_connections_triangle_second} and \cref{lem:higher_polygon_diagrams} that there exists a constant $c>0$ such that
 \begin{multline*}
 \E_{r}|K|^{\ell+1} \E_{r}|K|^{p-\ell}  - \frac{1}{|B_r|}\E_{r}\left[|K|^{\ell+1}\sum_{y\in B_r}\mathbbm{1}(y\notin K)|K_y|^{p-\ell} \right]
\\\preceq_p (\nabla_{r^c} +r^{-c}) (\E_{r}|K|)^{2p} = o\left(\left(
\frac{\E_r|K|^2}{\E_r|K|}\right)^{p-1} (\E_r|K|)^2
\right)
 \end{multline*}
 as required, where in the final inequality we used that $\E_{r}|K|^2\asymp (\E_{r}|K|)^3$ by \cref{prop:second_moment}. This concludes the proof.
\end{proof}

We are now ready to conclude the proof of \cref{thm:hd_moments_main}.

\begin{proof}[Proof of \cref{thm:hd_moments_main}]
The claimed asymptotics on the moments $\E_{\beta_c,r}|K|^p$ follow immediately from \cref{prop:first_moment,prop:second_moment,prop:higher_moments} (which apply in cases 2 and 3 of \eqref{HD} by \cref{lem:log_integrable_triangle}). 
 Once these moment estimates have been established, the asymptotic estimate on the volume tail follows from \cref{prop:first_order_volume_tail} as explained in \cref{subsec:tauberian}.
\end{proof}


\section{Superprocess limits}
\label{sec:superprocesses}

In this section we prove our results concerning the superprocess scaling limits of the model, \cref{thm:superprocess_main,thm:superprocess_main_regularly_varying}. We begin by studying the radius of gyration, and more generally moments of the form $\E_{\beta_c,r}\sum_{x\in K}\langle x,u\rangle^2$, in \cref{subsec:the_radius_of_gyration}. We then study the moments $\E_{\beta_c,r}\sum_{x\in K}\langle x,u\rangle^{2p}$ for general $p\geq 1$ in \cref{subsec:the_full_displacement_distribution}, and explain how these moments can be interpreted in terms of the superprocess scaling limit. In \cref{subsec:the_full_scaling_limit_of_the_size_biased_model_with_cut_off} we establish the full scaling limit (as a random measure) of the cluster of the origin under the size-biased cut-off measure $\hat \E_{\beta_c,r}$ by establishing the first-order asymptotics of a more general class of moments. Finally, in \cref{subsec:scaling_limits_without_cut_off} we show how the results of \cref{subsec:the_full_scaling_limit_of_the_size_biased_model_with_cut_off} can be used to deduce \cref{thm:superprocess_main,thm:superprocess_main_regularly_varying}, which concern scaling limits for the measure $\P_{\beta_c}$ without cut-off.

\begin{defn} 
The results of this section will be proven under the assumption that at least one of the following four hypotheses hold, which together cover the hypotheses of both \cref{thm:superprocess_main} (as listed in \eqref{HD}) and \cref{thm:superprocess_main_regularly_varying}. We refer to this list of hypotheses as \mytag[HD+]{{\color{blue}HD+}}\!\!.
\begin{enumerate}
  \item $d>3\alpha$.
    \item $d>6$ and  $\P_{\beta_c}(x\leftrightarrow y) \preceq \|x-y\|^{-d+2}$ for all $x\neq y$.
    \item $d=6$, $\alpha=2$, and  $\P_{\beta_c}(x\leftrightarrow y) \preceq (\log \|x-y\|)^{-1}\|x-y\|^{-d+2}$ for all $x\neq y$.
    \item $d=3\alpha<6$ and the hydrodynamic condition holds.
  \end{enumerate}
  We will often refer to these different hypotheses as Cases 1, 2, 3, and 4 of each relevant theorem, lemma, or proposition.
\end{defn}

\subsection{The radius of gyration}
\label{subsec:the_radius_of_gyration}

Recall from the introduction that the \textbf{radius of gyration} (a.k.a. the correlation length of order $2$) is defined to be 
\[
\xi_{2}(r):= \sqrt{\frac{\E_{\beta_c,r}\left[\sum_{x\in K} \|x\|_2^2 \right]}{\E_{\beta_c,r}|K|}}.
\]
The radius of gyration is a measure of the distance between the origin and a typical point in the cluster of the origin when this cluster has its ``typical large size'' (i.e., is drawn from the size-biased distribution). Our first step towards the computation of the scaling limit is to compute the asymptotics of the radius of gyration.

\begin{prop}[The radius of gyration]
\label{prop:radius_of_gyration}
Suppose that at least one of the hypotheses of \eqref{HD+} holds. 
For each $u\in \R^d$, the function $r\mapsto \E_{\beta_c,r}[\sum_{x\in K} \langle x,u\rangle^2]$ is 
regularly varying of index $\max\{2\alpha,2+\alpha\}$.
 Furthermore, under the same assumptions:
\begin{enumerate}
  \item If $\alpha<2$ then 
  \[
\E_{\beta_c,r}\left[\sum_{x\in K} \langle x,u\rangle^2 \right] 
  \sim \left(\frac{\alpha}{2-\alpha} \int_{B} \langle y,u\rangle^2 \dif y \right) r^2 \E_{\beta_c,r}|K|
\]
as $r\to \infty$ for each $u\in \R^d$.
  \item If $\alpha=2$ 
  then
  \[
\E_{\beta_c,r}\left[\sum_{x\in K} \langle x,u\rangle^2 \right] 
\sim  \left(\alpha \int_B \langle y,u\rangle^2 \dif y \right) (r^2 \log r) \E_{\beta_c,r}|K|
\]
as $r\to \infty$ for each $u\in \R^d$.
  \item If $\alpha>2$ then there exists a positive constant $A$ and a non-negative, unit trace, positive-definite matrix $\Sigma$ such that
  \[
\xi_2^2(r) \sim A r^{2\alpha} \qquad \text{ and } \qquad
\E_{\beta_c,r}\left[\sum_{x\in K} \langle x,u\rangle^2 \right] \sim \langle u,\Sigma u\rangle \xi_2^2(r) \E_{\beta_c,r}|K|
\]
as $r\to \infty$ for each $u\in \R^d$.
\end{enumerate}
In particular, the radius of gyration is regularly varying of index 
 $\max\{\alpha/2,1\}$ and is given by
\[
\xi_2(r) \sim \left\{\begin{array}{rll}
\sqrt{\frac{\alpha}{2-\alpha} \int_B \|y\|_2^2 \dif y} & r& \alpha <2 \\
\sqrt{\alpha \int_B \|y\|_2^2 \dif y} & r \sqrt{\log r} & \alpha =2 \\
\sqrt{A} & r^{\alpha/2} & \alpha>2.
\end{array}\right.
\]
\end{prop}

We will also prove the following weaker claim concerning the case $d=3\alpha\geq 6$ that does not assume a two-point function estimate as in Cases 2 and 3 of \eqref{HD+} and will therefore apply without any perturbative assumptions once \cref{III-thm:critical_dim_hydro} is proven.

\begin{prop}
\label{prop:radius_of_gyration_d=3alpha>6}
If $d=3\alpha \geq 6$ and the hydrodynamic condition holds then $\xi_2(r)$ is regularly varying of index $\alpha/2$ with $r^{-1}\xi_2(r)\to \infty$ as $r\to \infty$.
\end{prop}


It is a Corollary of \cref{prop:radius_of_gyration} that \emph{Sak's prediction} \cite{sak1973recursion} (introduced in detail in \cref{II-sec:introduction}) is valid in a spatially averaged sense for the effectively long-range models in high and critical effective dimensions satisfying \eqref{HD+}. (Of course this was already known for models where the lace expansion converges \cite{MR3306002,MR4032873,liu2025high}.) Let us now record this formally:

\begin{corollary}
\label{cor:HD_Sak}
Suppose that at least one of the hypotheses of \eqref{HD+} holds. If $\alpha<2$ then
\[
  \frac{1}{r^2}\sum_{x\in B_r}\P_{\beta_c}(0\leftrightarrow x) \asymp r^{-d+\alpha}
\]
for all $r\geq 1$. In particular, the critical exponent $\eta$ is well-defined  and equal to $2-\alpha$ in the spatially-averaged sense.
\end{corollary}

We improve this to a \emph{pointwise} estimate of the same order in \cref{II-thm:CL_Sak}.

\begin{proof}[Proof of \cref{cor:HD_Sak}]
The upper bound always holds (regardless of the value of $d$ and $\alpha>0$) by the main result of \cite{hutchcroft2022sharp} as stated in \cref{thm:two_point_spatial_average_upper}. When \eqref{HD+} holds and $\alpha<2$, it follows from \cref{prop:radius_of_gyration} that 
there exists a constant $C$ such that 
\[
  \E_{\beta_c,r} |K \cap B_{Cr}| \geq \frac{1}{2}\E_{\beta_c,r} |K|
\]
for every $r\geq 1$, and hence by \cref{prop:first_moment} (or \cref{cor:mean_lower_bound}) that
\[
  \sum_{x\in B_{Cr}}\P_{\beta_c}(0\leftrightarrow x) = \E_{\beta_c} |K \cap B_{Cr}| \geq \E_{\beta_c,r} |K \cap B_{Cr}|
  \geq \frac{1}{2}\E_{\beta_c,r} |K| \succeq r^\alpha
\]
for every $r\geq 1$. This is easily seen to imply the claimed lower bound.
\end{proof}

The proof of \cref{prop:radius_of_gyration} will rely on some more ODE lemmas.

\begin{lemma}[$f'\sim b r^{-1}f+h$ with $h$ of general order]
\label{lem:ODE_with_possibly_negligible_driving_term}
Let $a,b>0$, let $h$ be a positive, measurable, regulary varying function of index $b$, and suppose that $f$ is a positive, differentiable function such that
\[
f'(r)\sim \frac{af+h}{r}
\]
as $r\to \infty$. 
Then $f$ is regularly varying of index $\max\{a,b\}$. Moreover, if $a\geq b$ then $h(r)=o(f(r))$ and $f'(r)\sim ar^{-1}f(r)$ as $r\to \infty$.
\end{lemma}

\begin{proof}[Proof of \cref{lem:ODE_with_possibly_negligible_driving_term}]
The case in which $b>a$ has already been treated in \cref{lem:ODE_with_driving_term}, so it suffices to consider the case $b \leq a$. As in the proof of \cref{lem:ODE_with_driving_term}, 
there exists $r_0<\infty$ and a (not necessarily positive) measurable function $\delta:[r_0,\infty)\to \R$ with $|\delta_r|\to 0$ as $r\to \infty$ such that that $h$ is locally integrable on $[r_0,\infty)$ and $f'-(1-\delta_r)ar^{-1}f(r)=(1-\delta_r)r^{-1}h(r)$ for every $r\geq r_0$. Solving this first-order linear ODE explicitly between $r$ and $\lambda r$ yields that 
\begin{multline*}
f(\lambda r) = \exp\left[\int_{r}^{\lambda r} (1-\delta_s)as^{-1} \dif s\right] f(r) \\+ \exp\left[\int_{r}^{\lambda r} (1-\delta_s)as^{-1} \dif s\right]\int_{r}^{\lambda r} \frac{h(s)}{s}\exp\left[-\int_{r}^s (1-\delta_t)at^{-1} \dif t\right] \dif s.
\end{multline*}
when $r\geq r_0$ and $\lambda\geq 1$, and hence 
 that
\begin{equation}
\label{eq:f(lambda r)_small_a}
f(\lambda r) \sim \lambda^a f(r)+\lambda^a \frac{h(r)}{r} \int_r^{\lambda r} \left(\frac{s}{r}\right)^{b-a-1} \dif s = \begin{cases}
\lambda^a f(r) + \frac{\lambda^{a}-\lambda^{b}}{a-b} h(r)  & a+1 < b\\
\lambda^b f(r) + (\log \lambda) h(r) & a+1=b
\end{cases}
\end{equation}
 as $r\to \infty$,  uniformly for bounded values of $\lambda \geq 1$. In particular,  the claim that $h(r)=o(f(r))$ implies the claim that $f$ is regularly varying of index $a$. In the case that $a>b$ we can conclude by noting that, by \eqref{eq:f(lambda r)_small_a} and the assumption that $h$ is positive, $f(\lambda r) \geq (1-o(1)) \lambda^a f(r)$ as $r\to \infty$ so that $f(r)\geq r^{a-o(1)}$ and hence that $h(r)=o(f(r))$ as $r\to \infty$ as desired. 
We now consider the case $a=b$. We have by \eqref{eq:f(lambda r)_small_a} that
\[
\frac{f(2^{k+1})}{h(2^{k+1})} \sim \frac{2^a f(2^k)}{h(2^{k+1})} + \frac{(\log 2) h(2^k)}{ h(2^{k+1})} \sim \frac{f(2^k)}{h(2^{k})} + \frac{\log 2}{2^{a}}
\]
as $k\to \infty$. Letting $(\delta_k)_{k\geq 1}$ be a sequence with $|\delta_k|\to 0$ such that 
\[
\frac{f(2^{k+1})}{h(2^{k+1})} = (1-\delta_k) \frac{f(2^k)}{h(2^{k})} + (1-\delta_k)\frac{\log 2}{2^{a}},
\]
it follows by induction that
\[
\frac{f(2^{k+1})}{h(2^{k+1})} = 
\frac{\log 2}{2^{b}} \sum_{j=k_0}^k \prod_{i=j}^k (1-\delta_i)  + \prod_{i={k_0}}^k (1-\delta_i) \frac{f(2^{k_0})}{h(2^{k_0})} 
\]
for every sufficiently large $k_0$ and every $k\geq k_0$. This is easily seen to imply that $\frac{f(2^{k})}{h(2^{k})} \to \infty$ as $k\to \infty$, and together with \eqref{eq:f(lambda r)_small_a} this yields that $h(r)=o(f(r))$ as claimed.
\end{proof}

We will also need to understand more precisely what happens in the critical case $a=b$ under appropriate assumptions on the order of the error terms.

\begin{lemma}[$f'\sim b r^{-1}f+h$ with critical effects]
\label{lem:ODE_with_critical_driving_term}
Let $a>0$, let $h$ be a positive, measurable, regulary varying function of index $a$, and suppose that $f$ is a  positive, differentiable function such that
\[
f'(r) =\frac{(1-\delta_{1,r}) af(r)+(1-\delta_{2,r})h(r)}{r}
\]
where $\delta_{1,r}$ is a logarithmically integrable error function. If $r^{-a} h(r)$ is \emph{not} logarithmically integrable, then
\[
f(r) \sim r^{a} \int_{r_0}^r \frac{h(s)}{s^{a+1}} \dif s
\]
as $r\to\infty$, where $r_0$ is sufficiently large that $h$ is locally bounded on $[r_0,\infty)$. Otherwise, if $r^{-a}h(r)$ is logarithmically integrable, then there exists a constant $A>0$ such that $f(r)\sim Ar^{a}$.
\end{lemma}

It follows in particular that if $h(r)\sim A r^a$ and the error $\delta_{1,r}$ is logarithmically integrable then $f(r)\sim A r^{a} \log r$.

\begin{proof}[Proof of \cref{lem:ODE_with_critical_driving_term}]
It suffices without loss of generality to consider the case $\delta_{2,r} \equiv 0$, since $(1-\delta_{2,r})h(r)$ is regularly varying of the same index as $h$ and shares the logarithmic integrability properties as $h$ after multiplication by $r^{-a}$. Making this assumption and writing $\delta_r=\delta_{1,r}$, we have as in the proof of \cref{lem:ODE_with_driving_term} that if $r_0$ is sufficiently large then
 \begin{multline*}
f(r) = \exp\left[\int_{r_0}^r (1-\delta_s)as^{-1} \dif s\right] f(r_0) \\+ \exp\left[\int_{r_0}^r (1-\delta_s)as^{-1} \dif s\right]\int_{r_0}^r \frac{h(s)}{s}\exp\left[-\int_{r_0}^s (1-\delta_t)at^{-1} \dif t\right] \dif s
\end{multline*}
for every $r\geq r_0$ and hence, simplifying, that
 \begin{equation*}
f(r) = \exp\left[-a\int_{r_0}^r \frac{\delta_s}{s}\dif s\right] (r/r_0)^{a} f(r_0) + \exp\left[-a\int_{r_0}^r \frac{\delta_s}{s}\dif s\right] r^{a} \int_{r_0}^r \frac{h(s)}{s^{a+1}} \exp\left[a\int_{r_0}^s\frac{\delta_t}{t} \dif t\right] \dif s
\end{equation*}
for every $r\geq r_0$. Since $\delta_r$ is logarithmically integrable, each of the integrals appearing inside an exponential converges to a constant as $r\to \infty$ or $s\to \infty$ as appropriate. If $r^{-a}h(r)$ is logarithmically divergent, then the integral involving $h$ is dominated by the contribution from large values of $s$, so that if we set $A=\exp[-a\int_{r_0}^\infty \frac{\delta_s}{s}\dif s]$ then
\begin{equation*}
f(r) \sim A (r/r_0)^{a+1} f(r_0) + A r^{a} \int_{r_0}^r \frac{h(s)}{As^{a+1}} \dif s \sim r^{a+1} \int_{r_0}^r \frac{h(s)}{s^{a+1}} \dif s
\end{equation*}
as claimed. On the other hand, if $r^{-a}h(r)$ is logarithmically integrable, then
\[
f(r) \sim \left(A r_0^{-a-1}f(r_0)+A \int_{r_0}^\infty \frac{h(s)}{s^{a+1}} \exp\left[a\int_{r_0}^s\frac{\delta_t}{t} \dif t\right] \dif s \right)r^{a} =: A' r^{a}
\]
as $r\to \infty$ as claimed.
\end{proof}

With these ODE lemmas in hand, the next step is to write down an ODE satisfied by appropriate spatial moments.

\begin{lemma}
\label{lem:gyration_derivative}
Suppose that at least one of the four hypotheses of \eqref{HD+} holds. For each $u\in \R^d$ we have that
\begin{equation*}
 \frac{d}{dr} \E_{\beta_c,r} \left[\sum_{x\in K} \langle x,u\rangle^2 \right]
\sim  \frac{2 \alpha}{r} \E_{\beta_c,r} \left[\sum_{x\in K} \langle x,u\rangle^2 \right] + \frac{\alpha^2}{\beta_c} r^{2+\alpha-1} \int_{B} \langle y,u\rangle^2 \dif y
\end{equation*}
as $r\to \infty$. Moreover, in Cases 1, 2, and 3 all relevant errors are logarithmically integrable.
\end{lemma}

Compared to the similar estimates we proved in \cref{sec:analysis_of_moments}, the main additional difficulty here is that we are dealing with signed quantities, some of which approximately cancel with each other, and we need to be able to control this cancellation when proving that the error terms in this ODE are negligible. \cref{lem:gyration_derivative} easily implies \cref{prop:radius_of_gyration} in light of our above ODE lemmas, as we now explain in detail.


\begin{proof}[Proof of \cref{prop:radius_of_gyration} given \cref{lem:gyration_derivative}]
We write $\E_r=\E_{\beta_c,r}$.
It follows from \cref{lem:ODE_with_possibly_negligible_driving_term,lem:gyration_derivative} that the function $r\mapsto \E_{r} [\sum_{x\in K} \langle x,u\rangle^2 ]$ is regularly varying of index $\max\{2\alpha,2+\alpha\}$. As we have seen in \cref{lem:ODE_with_driving_term,lem:ODE_with_possibly_negligible_driving_term,lem:ODE_with_critical_driving_term}, the asymptotic ODE of \cref{lem:gyration_derivative} leads to different asymptotic behaviours of $\E_{r} [\sum_{x\in K} \langle x,u\rangle^2 ]$ according to whether $\alpha<2$, $\alpha=2$, or $\alpha>2$. When $\alpha<2$, it follows from \cref{lem:ODE_with_driving_term} that
\[
\E_{r}\left[\sum_{x\in K} \langle x,u\rangle^2 \right] \sim \frac{1}{2-\alpha}\left(\frac{\alpha^2}{\beta_c} \int_{B} \langle y,u\rangle^2 \dif y\right) r^{2+\alpha} \sim \left(\frac{\alpha}{2-\alpha} \int_B \langle y,u\rangle^2 \dif y\right) r^2 \E_{r}|K|
\]
as $r\to \infty$ for each $u\in \R^d$. 
On the other hand, when $\alpha=2$ and the errors are logarithmically integrable, we obtain from \cref{lem:ODE_with_critical_driving_term} that
\[
\E_{r}\left[\sum_{x\in K} \langle x,u\rangle^2 \right] \sim  \left(\frac{\alpha^2}{\beta_c} \int_{B} \langle y,u\rangle^2 \dif y\right) r^{2+\alpha} \log r \sim  \left(\alpha \int_B \langle y,u\rangle^2 \dif y\right) (r^2 \log r) \E_{r}|K|
\]
as $r\to \infty$ for each $u\in \R^d$.
Finally, when $\alpha>2$ and the errors are logarithmically integrable, we obtain from \cref{lem:ODE_with_possibly_negligible_driving_term} that for each $u\in\R^d \setminus \{0\}$ there exists a positive constant $A(u)$ such that
\[
\E_{r}\left[\sum_{x\in K} \langle x,u\rangle^2 \right] 
 \sim A(u) r^{\alpha} \E_{r}|K|
\]
as $r\to \infty$. The relevant bounds on the radius of gyration follow immediately in both cases. Finally, for each $r \geq 0$, the bilinear form $\Sigma_r$ defined by
\[
  \langle u,\Sigma_r v\rangle := \frac{1}{\xi_2^2(r)\E_r|K|}\E_{r}\left[\sum_{x\in K} \langle x,u\rangle \langle x,v\rangle \right] = \frac{1}{\xi_2^2(r)} \hat \E_r [\langle X,u\rangle \langle X,v\rangle] 
\] 
is positive semi-definite since it arises from the covariances of a random variable, where we write $X$ for a uniform random point of the cluster $K$, so that the limiting bilinear form is also positive semi-definite; in fact this limit is positive \emph{definite} since the relevant limiting constants are positive for each $u\in \R^d \setminus \{0\}$ as we saw above. (This is important to ensure that our limiting super-Brownian motion is not supported on a strict linear subspace of $\R^d$.)
\end{proof}

We will first prove \cref{lem:gyration_derivative} in Cases 1 and 2, with Cases 3 and 4 (which involve assumed two-point upper bounds) requiring a somewhat different treatment as usual.
We begin by writing down an estimate for the derivative that holds without any assumptions on $d$ or $\alpha$; we will then need to argue that the error terms appearing in this estimate are small under the assumptions of \cref{prop:radius_of_gyration}. 
Note that this form of the estimate will only be useful for proving \cref{prop:radius_of_gyration} in Cases 1 and 2; as usual, a slightly different approach will be needed when working from two-point function estimates in Cases 3 and 4.

\begin{lemma}
\label{lem:gyration_derivative2}
Let $u\in \R^d$. The estimate
\begin{multline}
\frac{1}{\beta_c |J'(r)|} \frac{d}{dr} \E_{\beta_c,r} \left[\sum_{x\in K} \langle x,u\rangle^2 \right]
=  2 |B_r| \E_{\beta_c,r}|K| \E_{r} \left[\sum_{x\in K} \langle x,u\rangle^2 \right] + (\E_{\beta_c,r}|K|)^2 \sum_{y\in B_r} \langle y,u\rangle^2
 \\
 \pm O\left(\|u\|^2_2 r^2 \E_{\beta_c,r}\left[|K|^2|K\cap B_r|\right]\right) \pm O\left(\E_{\beta_c,r}\left[|K||K\cap B_r| \sum_{x\in K}\langle x,u \rangle^2\right]\right).
\label{eq:x12_derivative_asymptotic}
\end{multline}
holds for all $r\geq 1$.
\end{lemma}

\begin{proof}[Proof of \cref{lem:gyration_derivative2}]
We may assume that $u\neq 0$, the claim holding vacuously otherwise. To lighten notation we will write $\E_r=\E_{\beta_c,r}$.
We begin by writing down the derivative formula
\begin{align*}
 \frac{d}{dr} \E_{r} \left[\sum_{x\in K} \langle x,u\rangle^2\right] &= \beta_c|J'(r)| \cdot \E_{r}\left[\sum_{x\in K} \sum_{y\in B_r(x)} \mathbbm{1}(y\notin K) \sum_{z\in K_y} \langle z,u\rangle^2\right] \\&=  \beta_c|J'(r)| \cdot\E_{r} \left[\sum_{x\in K} \sum_{y\in B_r} \mathbbm{1}(y\notin K) \sum_{z\in K_y} \langle z-x,u \rangle^2\right],
\end{align*}
where the second equality follows by using the mass-transport principle to exchange the roles of $0$ and $x$. 
Writing $z-x=(z-y)+(y-0)+(0-x)$ and expanding out the product, we get that
\begin{align*}
&\frac{1}{\beta_c |J'(r)|} \frac{d}{dr} \E_{r} \left[ \sum_{x\in K} \langle x,u\rangle^2 \right] =  \E_{r}\left[ |K| \sum_{y\in B_r} \mathbbm{1}(y\notin K) |K_y| \langle y,u\rangle^2 \right]
\\
&\hspace{0.5cm}+ \E_{r}\left[ |K| \sum_{y\in B_r} \mathbbm{1}(y\notin K) \sum_{z\in K_y} \langle z-y,u\rangle^2\right] 
+ \E_{r}\left[\sum_{x\in K} \langle x,u\rangle^2 \sum_{y\in B_r} \mathbbm{1}(y\notin K) |K_y|\right]\\
&\hspace{0.5cm}+ 2\E_{r}\left[|K| \sum_{y\in B_r} \langle y,u\rangle \mathbbm{1}(y\notin K) \sum_{z\in K_y} \langle z-y,u\rangle \right] 
 - 2\E_{r}\left[\sum_{x\in K} \langle x,u\rangle \sum_{y\in B_r} \mathbbm{1}(y\notin K) \sum_{z\in K_y} \langle z-y,u\rangle \right] 
 \\
&\hspace{0.5cm}
-2\E_{r}\left[  \sum_{x\in K} \langle x,u\rangle \sum_{y\in B_r} \langle y,u\rangle \mathbbm{1}(y\notin K)|K_y|\right].
\end{align*}
Using the mass-transport principle again to exchange the roles of $0$ and $y$ we see that some of these terms are equal, allowing us to write more simply that
\begin{align}
&\frac{1}{\beta_c |J'(r)|} \frac{d}{dr} \E_{r} \left[\sum_{x\in K} \langle x,u\rangle^2 \right]
=  \E_{r} \left[|K| \sum_{y\in B_r} \mathbbm{1}(y\notin K) |K_y| \langle y,u\rangle^2 \right] \nonumber\\
&\hspace{1.5cm}+2 \E_{r}\left[\sum_{x\in K} \langle x,u\rangle^2 \sum_{y\in B_r} \mathbbm{1}(y\notin K) |K_y|\right]
-2 \E_{r}\left[\sum_{x\in K} \langle x,u \rangle \sum_{y\in B_r} \mathbbm{1}(y\notin K) \sum_{z\in K_y} \langle z-y,u\rangle\right]\nonumber\\
&\hspace{1.5cm}-
4 \E_{r} \left[ \sum_{x\in K} \langle x,u \rangle \sum_{y\in B_r} \langle y,u\rangle \mathbbm{1}(y\notin K)|K_y|\right].
\label{eq:x12_derivative}
\end{align}
 For the first expectation on the right hand side of \eqref{eq:x12_derivative}, we have by \cref{lem:BK_disjoint_clusters_covariance} that
 \begin{align*}
&\left|(\E_{r}|K|)^2 \sum_{y\in B_r} \langle y,u\rangle^2 -
  \E_{r} \left[|K| \sum_{y\in B_r} \mathbbm{1}(y\notin K) |K_y| \langle y,u\rangle^2\right]\right| 
  \\&\hspace{6.5cm}\leq \sum_{y\in B_r} \langle y,u\rangle^2\left((\E_{r}|K|)^2 -
  \E_{r} \left[\mathbbm{1}(y\notin K) |K| |K_y| \right]\right)
  \\
  &\hspace{6.5cm}\leq\sum_{y\in B_r} \langle y,u\rangle^2
  \E_r \left[ |K|^2 \mathbbm{1}(y\in K)\right] \preceq r^{2} \E_r\left[|K|^2|K\cap B_r|\right],
\end{align*}
where we used that $(\E_{r}|K|)^2 -
  \E_{r} \left[\mathbbm{1}(y\notin K) |K| |K_y| \right]$ is non-negative in the first inequality. For the second expectation we have by similar reasoning that
 \begin{equation*}
\left| |B_r| (\E_{r}|K|)(\E_r\sum_{x\in K}\langle x,u\rangle^2)  -
  \E_{r} \left[\sum_{x\in K}\langle x,u\rangle^2 \sum_{y\in B_r} \mathbbm{1}(y\notin K) |K_y| \right] \right| \leq \E_r \left[ |K| |K\cap B_r| \sum_{x\in K}\langle x,u\rangle^2 \right].
\end{equation*}
 Let us now consider the last three terms, which we expect should approximately cancel to zero. For each $y\in \Z^d$ we can write
\begin{multline*}\E_r\left[ \mathbbm{1}(y\notin K) \sum_{x\in K} \langle x,u \rangle \sum_{z\in K_y} \langle z-y,u\rangle\right]
=\sum_{x,z} \P_r(0\leftrightarrow x \nleftrightarrow y \leftrightarrow z) \langle x, u \rangle \langle z-y,u\rangle
\\
=\sum_{x,z} \left[\P_r(0\leftrightarrow x \nleftrightarrow y \leftrightarrow z)-\P_r(0\leftrightarrow x)\P_r(y \leftrightarrow z)\right]\langle x,u\rangle\langle z-y,u\rangle,
\end{multline*}
where the second equality follows since $\sum_{x,z} \P_r(0\leftrightarrow x)\P_r(y \leftrightarrow z)\langle x,u\rangle\langle z_1-y_1,u\rangle=0$ by symmetry. 
Using \cref{lem:disjoint_connections}, we can therefore bound
\begin{align*}
\left|\E_r\sum_{x\in K}  \sum_{y\in B_r} \mathbbm{1}(y\notin K) \sum_{z\in K_y} \langle x,u\rangle\langle z-y,u\rangle\right|&\leq
\sum_{y\in B_r} \sum_{x,z\in \Z^d} \P_{r}(0 \leftrightarrow x \leftrightarrow y \leftrightarrow z) |\langle x,u\rangle||\langle z-y,u\rangle| \\
& \leq  \frac{1}{2}\E_r\left[|K||K\cap B_r| \sum_{x\in K}\langle x,u\rangle^2 \right],
\end{align*}
where in the second inequality we used that $|\langle x,u\rangle||\langle z-y,u\rangle| \leq \frac{1}{2}(\langle x,u\rangle^2+\langle z-y,u\rangle^2)$ and used the mass-transport principle to write both sums in terms of $x$.
We also have similarly that
\begin{align*}\E_r\left[ \mathbbm{1}(y\notin K)  |K_y| \sum_{x\in K} \langle x,u\rangle\langle y,u\rangle  \right]
&=\sum_{x,z} \P_r(0\leftrightarrow x \nleftrightarrow y \leftrightarrow z)\langle x,u\rangle\langle y,u\rangle
\\
&=\sum_{x,z} \left[\P_r(0\leftrightarrow x \nleftrightarrow y \leftrightarrow z)-\P_r(0\leftrightarrow x)\P_r(y \leftrightarrow z)\right]\!\langle x,u\rangle\langle y,u\rangle.
\end{align*}
and can therefore bound
\begin{multline*}
\left|\E_r\sum_{x\in K} \langle x,u\rangle \sum_{y\in B_r} \langle y,u\rangle \mathbbm{1}(y\notin K) |K_y|\right|\leq
\sum_{y\in B_r} \sum_{x,z\in \Z^d} \P_{r}(0 \leftrightarrow x \leftrightarrow y \leftrightarrow z) |\langle x,u\rangle||\langle y,u\rangle| \\
 \leq  \frac{1}{2}\E_r\left[|K||K\cap B_r| \sum_{x\in K}\langle x,u\rangle^2 \right] + \frac{1}{2}\E_r\left[|K|^2 \sum_{y\in K \cap B_r} \langle y,u\rangle^2  \right].
\end{multline*}
Putting all these estimates together yields the claim.
\end{proof}

To bound these errors, we will make use of the following easy consequence of the tree-graph inequalities.

\begin{lemma}[Spatially-weighted tree-graph inequalities]
\label{lem:spatial_tree_graph}
Let $u_1,\ldots,u_n\in \R^d$ and $p_1,\ldots,p_n \geq 0$ and let $p=\sum_i p_i$. 
The inequality
\[
  \E_{\beta,r}\left[ \sum_{x_1,\ldots,x_n\in K} \prod_{i=1}^n|\langle x_i,u_i\rangle|^{p_i} \right] \leq (2n-3)!! n^{\max\{p,1\}}(\E_{\beta,r}|K|)^{2n-2} \max_{1\leq i \leq n} \E_{\beta,r}\left[ |K|^{n-1}\sum_{x\in K}|\langle x,u_i\rangle|^p\right]
\]
holds for every $\beta,r\geq 0$.
\end{lemma}

\begin{proof}[Proof of \cref{lem:spatial_tree_graph}]
It follows from the AM-GM inequality and linearity of expectation that
\begin{multline}
\label{eq:spatial_AM_GM}
  \E_{\beta,r} \left[\sum_{x_1,\ldots,x_n\in K} \prod_{i=1}^n |\langle x_i,u_i\rangle|^{p_i} \right]\leq \E_{\beta,r} \left[\sum_{i=1}^n \frac{p_i}{p} |K|^{n-1}\sum_{x\in K} |\langle x,u_i\rangle|^{p} \right] 
  \\\leq 
  \max_{1\leq i \leq n} \E_{\beta,r} \left[|K|^{n-1} \sum_{x_i\in K}|\langle x_i,u_i\rangle|^{p} \right].
\end{multline}
Applying the tree-graph inequality \eqref{tree_graph_general}, we obtain that
\begin{equation}
\label{eq:tree_graph_spatial_explicit}
 \E_{\beta,r}\left[ \sum_{x_1,\ldots,x_n}  \prod_{i=1}^n|\langle x_i,u_i\rangle|^{p_i}\right]
 \leq
\max_{1\leq i_0 \leq n} \sum_{T\in \mathbb{T}_n} \sum_{x_{1},\ldots,x_{2n-1}} \prod_{\substack{i<j \\ i\sim j}} \P_{\beta,r}(x_i \leftrightarrow x_j) |\langle x_{i_0},u_{i_0}\rangle|^p,
\end{equation}
 where, as usual, the sum is taken over isomorphism classes of trees with leaves labelled $0,1,\ldots,n$ and unlabelled non-leaf vertices all of degree $3$. For each tree $T$ in this sum, let $0=j_0,\ldots j_\ell=i_0$ be the unique simple path from $0$ to $i_0$ in $T$. Applying the elementary inequality $(\sum_{i=1}^\ell a_i)^p \leq \ell^{\max\{p-1,0\}} \sum_{i=1}^\ell a_i^p$, which holds for any $p\geq 0$ and $a_1,\ldots a_\ell \geq 0$, yields that
 \[
   |\langle x_{i_0},u_{i_0}\rangle|^p \leq \ell^{\max\{p-1,0\}} \sum_{m=1}^\ell |\langle x_{j_m}-x_{j_{m-1}},u_{i_0}\rangle|^p.
 \]
Substituting this estimate into
\eqref{eq:tree_graph_spatial_explicit} we obtain that
\begin{align*}
 &\E_{\beta,r} \left[\sum_{x_1,\ldots,x_n\in K}  \prod_{i=1}^n|\langle x_i,u_i\rangle|^{p_i}\right]
 \\&\hspace{3.5cm}\leq
\max_{1\leq i_0 \leq n} \sum_{T\in\bbT_n} \sum_{x_1,\ldots,x_{2n-1}} \prod_{\substack{i<j \\ i\sim j}} \P_{\beta,r}(x_i \leftrightarrow x_j) \ell^{\max\{p-1,0\}} \sum_{m=1}^\ell |\langle x_{j_m}-x_{j_{m-1}},u_{i_0}\rangle|^p
\\
&\hspace{3.5cm}= 
\max_{1\leq i_0 \leq n} \sum_{T\in\bbT_n} \ell^{\max\{p,1\}} (\E_{\beta,r}|K|)^{2n-2}\E_{\beta,r}\sum_{x\in K} |\langle x,u_{i_0}\rangle|^p 
\\
&\hspace{3.5cm}\leq (2n-3)!! n^{\max\{p,1\}} \max_{1\leq i_0 \leq n}  (\E_{\beta,r}|K|)^{2n-2}\E_{\beta,r}\sum_{x\in K} |\langle x,u_{i_0}\rangle|^p 
\end{align*}
as claimed, where we used transitivity in the second line and used that $\ell \leq n$ for every index $i_0$ and tree $T$  in the final inequality.
\end{proof}

In the case $d=3\alpha$, we will also make use of the following four-point generalization of the Gladkov inequality \cite{gladkov2024percolation}, which is a special case\footnote{We hope the reader is not too aggrieved by our decision to defer the proof of the higher Gladkov inequality to \cref{II-subsec:higher_gladkov_inequalities_and_the_k_point_function}. These inequalities play a much more important role in the low-effective-dimensional regime than they do in the analysis of this paper, so that it is most natural to discuss them at length in the second paper. The proof of \cref{II-thm:higher_Gladkov}, which applies to percolation on arbitrary weighted graphs, is a straightforward generalization of the proof of \cite{gladkov2024percolation} and does not rely on the results of this paper in any way.} of \cref{II-thm:higher_Gladkov} (in fact the slightly stronger version of the theorem discussed in \cref{II-rmk:improved_higher_Gladkov}).

\begin{lemma}
\label{lem:four_point_Gladkov}
There exists a universal constant $C$ such that the inequality
\begin{multline*}
  \P_{\beta,r}(x,y,z,w \text{ \emph{all connected}})^2 \leq C \max\{ \P_{\beta,r} (x \leftrightarrow y \leftrightarrow z)\P_{\beta,r} (z \leftrightarrow w)
  \P_{\beta,r} (x \leftrightarrow y \leftrightarrow w),\\
   \P_{\beta,r} (x \leftrightarrow y \leftrightarrow z)\P_{\beta,r} (y \leftrightarrow w)
  \P_{\beta,r} (x \leftrightarrow z \leftrightarrow w)
  \}
\end{multline*}
holds for all $x,y,z,w\in \Z^d$ and $\beta,r\geq 0$.
\end{lemma}

%
%
We now apply these inequalities to prove the first two cases of \cref{lem:gyration_derivative}.

\begin{proof}[Proof of \cref{lem:gyration_derivative}, Cases 1 and 2]
We continue to write $\E_r=\E_{\beta_c,r}$. Fix $u\in \R^d\setminus \{0\}$, the case $u=0$ of all relevant estimates holding vacuously. It suffices by \cref{lem:gyration_derivative2} to prove that the error terms
\[
\frac{r^2\E_r\left[|K|^2 |K\cap B_r|\right]}{r^{d+2}(\E_r|K|)^2} \qquad \text{ and } \qquad \frac{\E_r\left[|K| |K\cap B_r| \sum_{x\in K} \langle x,u\rangle^2\right]}{r^{d}\E_r|K| \max\{\E_r \sum_{x\in K} \langle x,u\rangle^2,r^2\E_r|K|\}}
\]
both converge to zero and moreover are logarithmically integrable when $d>3\alpha$. The relevant estimates for the first of these errors was already proven in the proof of \cref{prop:first_moment}. For the second, we break the analysis into cases:

\medskip

\noindent \textbf{Case 1: $d>3\alpha$.} In this case, we have by \cref{lem:spatial_tree_graph} that
\[
\E_r\left[|K| |K\cap B_r| \sum_{x\in K} \langle x,u\rangle^2\right] \leq \E_r\left[|K|^2\sum_{x\in K} \langle x,u\rangle^2\right] \preceq (\E_r|K|)^4 \E_r\sum_{x\in K} \langle x,u\rangle^2,
\]
so that
\[
\frac{\E_r\left[|K| |K\cap B_r| \sum_{x\in K} \langle x,u\rangle^2\right]}{r^{d}\E_r|K| \E_r \sum_{x\in K} \langle x,u\rangle^2} \preceq \frac{(\E_r|K|)^3}{r^d} \asymp r^{3\alpha-d}
\]
is logarithmically integrable as desired, where we used \cref{thm:hd_moments_main} in the final estimate.

\medskip

\noindent \textbf{Case 2: $d=3\alpha$ \& \eqref{Hydro}.}  In this case, we use Cauchy-Schwarz to bound 
\begin{equation}
\label{eq:critical_dim_gyration_error_Cauchy_Schwarz}
\E_r\left[|K| |K\cap B_r| \sum_{x\in K} \langle x,u\rangle^2\right]^2 \leq \E_r \left[|K|^2 \sum_{x\in K} \langle x,u\rangle^2 \right] \E_r \Bigl[|K\cap B_r|^2 \sum_{x\in K} \langle x,u\rangle^2\Bigr].
\end{equation}
For the first term, we can use \cref{lem:spatial_tree_graph} as above to bound 
\begin{equation}
  \E_r \left[|K|^2 \sum_{x\in K} \langle x,u\rangle^2 \right]  \preceq (\E_r |K|)^4 \E_r \sum_{x\in K} \langle x,u\rangle^2.
  \label{eq:critical_dim_gyration_error_loose}
\end{equation}
For the second, we can use the four-point Gladkov inequality (\cref{lem:four_point_Gladkov}) to bound
\begin{align}
  &\E_r \Bigl[|K\cap B_r|^2 \sum_{x\in K} \langle x,u\rangle^2\Bigr] 
  \label{eq:critical_dim_gyration_error_Gladkov}\\&\hspace{0.75cm}\preceq \sum_{y,z\in B_r} \sum_{x\in \Z^d} \langle x,u\rangle^2 \P_r(0\leftrightarrow y \leftrightarrow z)^{1/2}\P_r(0\leftrightarrow x)^{1/2}\P_r(y\leftrightarrow x)^{1/4}\P_r(z\leftrightarrow x)^{1/4}\P_r(y\leftrightarrow z)^{1/4}
  \nonumber\\
  &\hspace{1.5cm}+\sum_{y,z\in B_r} \sum_{x\in \Z^d} \langle x,u\rangle^2 \P_r(0\leftrightarrow y \leftrightarrow z)^{1/2}\P_r(z\leftrightarrow x)^{1/4}\P_r(y\leftrightarrow x)^{1/2}\P_r(z\leftrightarrow x)^{1/4}\P_r(0\leftrightarrow z)^{1/4},
  \nonumber
\end{align}
where we also applied the three-point Gladkov inequality (\cref{thm:Gladkov}) to bound the $\P_{r}(x\leftrightarrow y \leftrightarrow w)$ terms that appear.
Noting that $\langle x,u\rangle^2 \preceq \max\{\langle x-y,u\rangle^2 + r^2,\langle x-z,u\rangle^2 + r^2\}$, we can use H\"older's inequality to bound the sum over $x$ in the first term by
\begin{multline*}
\sum_{x\in \Z^d} \langle x,u\rangle^2 \P_r(0\leftrightarrow x)^{1/4}\P_r(y\leftrightarrow x)^{1/2}\P_r(z\leftrightarrow x)^{1/4}
\\\leq \left[\sum_{x\in \Z^d} \langle x,u\rangle^2 \P_r(0\leftrightarrow x) \right]^{1/4}\left[\sum_{x\in \Z^d} \langle x,u\rangle^2 \P_r(y\leftrightarrow x) \right]^{1/2}\left[\sum_{x\in \Z^d} \langle x,u\rangle^2 \P_r(z\leftrightarrow x) \right]^{1/4}
\\ 
\preceq \max\{\E_r \sum_{x\in K}\langle x,u\rangle^2, r^2 \E_r |K|\}
\end{multline*}
for every $y,z\in B_r$, 
with a similar bound holding for the sum over $x$ in the other term. Substituting these estimates into \eqref{eq:critical_dim_gyration_error_Gladkov} yields the bound
\begin{align}
&\E_r \Bigl[|K\cap B_r|^2 \sum_{x\in K} \langle x,u\rangle^2\Bigr] 
\label{eq:critical_dim_gyration_error_post_Gladkov}
\\&\hspace{1cm}\preceq \max\left\{r^2\E_r|K|,\E_r\sum_{x\in K} \langle x,u\rangle^2\right\}\sum_{y,z\in B_r} \P_r(0\leftrightarrow y \leftrightarrow z)^{1/2}\left[\P_r(y\leftrightarrow z)^{1/4}+\P_r(0\leftrightarrow z)^{1/4}\right].
\nonumber
\end{align}
To bound this expression, we use H\"older's inequality again to write
\begin{align}
&\sum_{y,z\in B_r} \P_r(0\leftrightarrow y \leftrightarrow z)^{1/2}\P_r(y\leftrightarrow z)^{1/4} 
\nonumber\\&\hspace{5cm}\leq \left[\sum_{y,z\in B_r} \P_r(0\leftrightarrow y \leftrightarrow z)\right]^{1/2}\left[\sum_{y,z\in B_r} \P_r(y \leftrightarrow z)\right]^{1/4}\left[\sum_{y,z\in B_r} 1\right]^{1/4}\nonumber\\
&\hspace{5cm}\leq (\E_r |K\cap B_r|^2)^{1/2} (|B_r|\E_r|K|)^{1/4} |B_r|^{1/2}\nonumber\\
&\hspace{5cm}\preceq M_r^{1/2} |B_r|^{3/4} (\E_r|K|)^{3/4} = o(r^{4\alpha}),
\label{eq:critical_dim_gyration_error_almost_done}
\end{align}
where we used the universal tightness theorem via \cref{cor:universal_tightness_moments} and the hydrodynamic condition in the last line. A similar bound holds on the other summand in \eqref{eq:critical_dim_gyration_error_post_Gladkov} by the same argument. Putting together \eqref{eq:critical_dim_gyration_error_Cauchy_Schwarz}, \eqref{eq:critical_dim_gyration_error_loose}, \eqref{eq:critical_dim_gyration_error_post_Gladkov}, and \eqref{eq:critical_dim_gyration_error_almost_done}, we obtain that
\[
  \E_r\left[|K| |K\cap B_r| \sum_{x\in K} \langle x,u\rangle^2\right] = o\left(r^{4\alpha} \max\left\{\E_r  \sum_{x\in K} \langle x,u\rangle^2,r^2\E_r|K|\right\}\right),
\]
which ensures that the relevant error goes to zero as $r\to \infty$ as desired.
\end{proof}

Before moving on to Cases 2 and 3, let us briefly note how \cref{prop:radius_of_gyration_d=3alpha>6} can be proven using the same method as Case 4 of \cref{prop:radius_of_gyration}.

\begin{proof}[Proof of \cref{prop:radius_of_gyration_d=3alpha>6}]
The proof of Case 4 of \cref{lem:gyration_derivative} extends without modification to the more general setting that $d=3\alpha$ and the hydrodynamic condition holds. When $d=3\alpha\geq 6$, the driving term $r^{2+\alpha}$ is regularly varying of index greater than or equal to the coefficient $2\alpha$ of the self-referential term, so that \cref{lem:ODE_with_possibly_negligible_driving_term} now yields that $\E_{\beta_c,r}\sum_{x\in K}\|x\|_2^2$ is regularly varying of index $2\alpha$ and satisfies $\lim_{r\to \infty} r^{-\alpha-2}\E_{\beta_c,r}\sum_{x\in K}\|x\|_2^2$. This implies the claim in conjunction with \cref{prop:first_moment}.
\end{proof}

\subsubsection{Cases 2 and 3}
It remains to prove the second and third cases of \cref{lem:gyration_derivative}, in which we work with an assumed upper bound on the two-point function. It will be convenient to prove the lemma for the derivative of $\E_{\beta_c,r}\sum_{x\in K}\|x\|_2^2$ before returning to treat $\E_{\beta_c,r}\sum_{x\in K}\langle x,u\rangle^2$. We begin by obtaining a lower bound of the correct order on the radius of gyration as an immediate consequence of the assumed two-point upper bound and our estimates on the first moment (\cref{prop:first_moment}).


\begin{lemma}
If Case 2 or 3 of \eqref{HD+} holds then
\label{lem:gyration_lower_bound}
\[\E_{\beta_c,r} \sum_{x\in K} \|x\|^{2p}_2 \succeq_p \begin{cases} r^{(p+1)\alpha} & \text{Case 2}
\\
r^{(p+1)\alpha} (\log r)^p & \text{Case 3}
\end{cases} \]
for every $r\geq 1$ and $p\geq 1$.
\end{lemma}

\begin{proof}[Proof of \cref{lem:gyration_lower_bound}]
We can use our assumed bound on the two-point function to estimate
\[
  \sum_{\|x\|\leq R} \P_{\beta_c,r}(0\leftrightarrow x) \leq \sum_{\|x\|\leq R} \P_{\beta_c,r}(0\leftrightarrow x) \preceq \begin{cases} R^{2} & \text{Case 2} \\ \frac{R^2}{\log R} & \text{Case 3}
  \end{cases}
\]
for every $r>0$ and $R \geq 1$. Using that $\E_{\beta_c,r}|K| \asymp r^\alpha$ (\cref{prop:first_moment}), it follows that there exists a positive constant $c$ such that if we define $R(r)=c r^{\alpha/2}$ (in Case 2) or $R(r)=c r \sqrt{\log r}$ (in Case 3) then
\begin{equation}
 \E_{\beta_c,r}|K \setminus B_{R}| \geq c\E_{\beta_c,r}|K|
\label{eq:gyration_lower_proof}
\end{equation}
for every $r\geq 1$ and hence that
\[
  \E_{\beta_c,r} \sum_{x\in K} \|x\|^{2p}_2 \geq R^{2p} \E_{\beta_c,r}|K \setminus B_{R}| \succeq R^{2p} r^\alpha.
\]
This is equivalent to the claim.
\end{proof}

\begin{lemma}[The spatially weighted triangle diagram]
\label{lem:spatial_triangle} If Cases 2 or 3 of \eqref{HD+} hold then
\begin{multline*}
  \sum_{x,y\in \Z^d} \|x\|_2^2 \P_{\beta_c,r}(0\leftrightarrow x) \P_{\beta_c,r}(x\leftrightarrow y) \P_{\beta_c,r}(y \leftrightarrow z),\\\sum_{x,y\in \Z^d} \|y-x\|_2^2 \P_{\beta_c,r}(0\leftrightarrow x) \P_{\beta_c,r}(x\leftrightarrow y) \P_{\beta_c,r}(y \leftrightarrow z)
   \preceq \begin{cases}
  1 & d > 8\\
  \log r & d=8\\
  r^{\alpha/2} & d=7 \\
  \frac{r^\alpha}{(\log r)^2} & d =6 
  \end{cases}
\end{multline*}
for every $r\geq 2$ and $z\in B_r$.
\end{lemma}

\begin{proof}[Proof of \cref{lem:spatial_triangle}]
We will prove the claim for the sum involving $\|x\|_2^2$, the other claim being similar.
Fix $r\geq 2$ and $z\in B_r$ and write $\langle x \rangle=2+\lceil \|x\| \rceil$.
We have by \eqref{eq:convolution_estimate_logs} that
\[
  \sum_{y\in \Z^d} \P_{\beta_c}(0\leftrightarrow y)\P_{\beta_c}(y\leftrightarrow x) \preceq \begin{cases}
  \langle x\rangle^{-d+4} & \text{Case 2}\\
  \frac{\langle x\rangle^{-2}}{(\log \langle x\rangle)^2} & \text{Case 3},
  \end{cases}
\]
so that
\begin{multline*}
  \sum_{x,y\in \Z^d}\|x\|_2^{2}\P_{\beta_c,r}(0\leftrightarrow x)\P_{\beta_c,r}(x\leftrightarrow y)\P_{\beta_c,r}(y\leftrightarrow z) \\\preceq \begin{cases} \sum_x \P_{\beta_c,r}(0\leftrightarrow x) \langle x \rangle^2\langle x-z\rangle^{-d+4} & \text{Case 2}\\
  \sum_x \frac{\langle x \rangle^2 \langle x-z\rangle^{-2}}{(\log \langle x-z\rangle)^{2}}\P_{\beta_c,r}(0\leftrightarrow x) & \text{Case 3.}
  \end{cases}
\end{multline*}
In both cases we will split the sum into two pieces according to whether $\langle x-z\rangle$ is greater or smaller than a parameter $R\geq 2r$ and then optimize over $R$. In Case 2 we have that
\[
  \sum_{x\in \Z^d} \P(0\leftrightarrow x)\langle x \rangle^2 \langle x-z \rangle^{-d+4} \preceq \sum_{\langle x-z\rangle\leq R} \langle x\rangle^{-d+4}\langle x-z\rangle^{-d+4} + R^{-d+6} \E_{\beta_c,r}|K|,
\]
where we used that if $\langle x-z\rangle \geq R$ then $\langle x \rangle \asymp \langle x-z \rangle$. The sum
\[\sum_{\langle x-z\rangle\leq R} \langle x\rangle^{-d+4}\langle x-z\rangle^{-d+4} \preceq \begin{cases}
1 & d>8\\
\log R & d=8\\
R & d =7
\end{cases} 
\]
can be estimated using the same method as in the proof of \eqref{eq:convolution_estimate_logs}, and the claim follows by optimizing over $R$ in each case (taking $R=\infty$ for $d>8$ and $R= r^{\alpha/2}$ for $d=7,8$).
Similarly, in Case 3 we have that
\begin{multline*}
  \sum_{x\in \Z^d} \frac{\langle x \rangle^2 \langle x-z\rangle^{-2}}{(\log \langle x-z\rangle)^{2}}\P_{\beta_c,r}(0\leftrightarrow x) \preceq \sum_{\|x\|\leq R} \frac{\langle x\rangle^{-d+4} \langle x-z\rangle^{-2}}{(\log \langle x\rangle)(\log \langle x-z\rangle)^2} + (\log R)^{-2} \E_{\beta_c,r}|K|\\
\asymp 
R^2 (\log R)^{-3} + (\log R)^{-2} r^\alpha
\end{multline*}
and we can take $R=r^{\alpha/2} (\log r)^{1/2}$ to get the claimed estimate.
\end{proof}

We next prove the following estimate, which is very close to that of \cref{lem:gyration_derivative} except that the error is expressed in terms of $\E_{\beta_c,r}\sum_{x\in K}\|x\|_2^2$ rather than $\E_{\beta_c,r}\sum_{x\in K}\langle x,u \rangle^2$.

\begin{lemma}
\label{lem:gyration_derivative3}
If Case 2 or 3 of \eqref{HD+} holds then for each $u\in \R^d$ we have that
\begin{multline*}
 \frac{d}{dr} \E_{\beta_c,r} \left[\sum_{x\in K} \langle x,u\rangle^2 \right]
=  \frac{2 \alpha}{r}  \E_{\beta_c,r} \left[\sum_{x\in K} \langle x,u\rangle^2 \right] + \frac{\alpha^2}{\beta_c} r \int_{B} \langle y,u\rangle^2 \dif y 
\\\pm O\left(\frac{\|u\|^2_2}{r(\log r)^2} \E_{\beta_c,r}\sum_{x\in K}\|x\|_2^2\right)
\end{multline*}
as $r\to \infty$.
\end{lemma}

The error estimate appearing here has not been optimized; the important thing is that $(\log r)^{-2}$ is logarithmically integrable.

\begin{proof}[Proof of \cref{lem:gyration_derivative3}]
It suffices to consider the case that $u$ is a unit vector; fix one such $u$. As in the proof of \cref{lem:gyration_derivative2}, we prove that the first two terms on the right hand side of \eqref{eq:x12_derivative} approximately factor and that the last two terms are negligible, taking care to ensure that all errors are logarithmically integrable. For the first term, we can use \cref{lem:disjoint_connections_triangle} and \cref{lem:spatial_triangle} to bound
\begin{align*}
&\left|(\E_{r}|K|)^2 \sum_{y\in B_r} \langle y,u\rangle^2 -
  \E_{r} \left[|K| \sum_{y\in B_r} \mathbbm{1}(y\notin K) |K_y| \langle y,u\rangle^2\right]\right| \\&\hspace{4cm}\leq \sum_{x,z,a,b\in \Z^d}\sum_{y\in B_r} \P_r(0\leftrightarrow a)\P_r(a\leftrightarrow x)\P_r(a\leftrightarrow b)\P_r(b\leftrightarrow y) \P_r(b\leftrightarrow z) \| y\|^2_2 
  \\ &\hspace{4cm}\preceq \frac{r^{d+3\alpha}}{(\log r)^2},
\end{align*}
where before applying \cref{lem:spatial_triangle} in the last line we bounded $\|y\|_2^2 \preceq \|a\|_2^2+\|b-a\|_2^2+\|y-b\|_2^2$. For the second term, we have similarly that
\begin{align*}&\left| |B_r| (\E_{r}|K|)(\E_r\sum_{x\in K}\langle x,u\rangle^2)  -
  \E_{r} \left[\sum_{x\in K}\langle x,u\rangle^2 \sum_{y\in B_r} \mathbbm{1}(y\notin K) |K_y| \right] \right|\\
  & \hspace{4cm}\leq \sum_{x,z,a,b\in \Z^d}\sum_{y\in B_r} \P_r(0\leftrightarrow a)\P_r(a\leftrightarrow x)\P_r(a\leftrightarrow b)\P_r(b\leftrightarrow y) \P_r(b\leftrightarrow z) \| x\|^2_2 
  \\ & \hspace{4cm} \preceq \frac{r^{d+3\alpha}}{(\log r)^2} + \left[\sum_{y\in B_r} \nabla(0,y)\right] r^\alpha  \E\sum_{x\in K}\|x\|_2^2 \preceq \frac{r^{d+\alpha}}{(\log r)^2} \E\sum_{x\in K}\|x\|_2^2.
\end{align*}
where again we used \cref{lem:spatial_triangle,lem:gyration_lower_bound} and the triangle bounds of \eqref{eq:triangle_case_2} and \eqref{eq:triangle_case_3} in the last line. For the cross-terms, we can do the same algebraic manipulations as in the proof of Cases 1 and 4 and use \cref{lem:disjoint_connections_triangle} to obtain that
\begin{align*}
&\left|\E_r\sum_{x\in K}  \sum_{y\in B_r} \mathbbm{1}(y\notin K) \sum_{z\in K_y} \langle x,u\rangle\langle z-y,u\rangle\right|\\&\hspace{2.7cm}\preceq
\sum_{y\in B_r} \sum_{x,z\in \Z^d} \P_r(0\leftrightarrow a)\P_r(a\leftrightarrow x)\P_r(a\leftrightarrow b)\P_r(b\leftrightarrow y) \P_r(b\leftrightarrow z) (\|x\|_2^2+\|z-y\|_2^2)
\end{align*}
and
\begin{align*}
&\left|\E_r\sum_{x\in K} \langle x,u\rangle \sum_{y\in B_r} \langle y,u\rangle \mathbbm{1}(y\notin K) |K_y|\right|
\\&\hspace{2.7cm}\preceq
\sum_{y\in B_r} \sum_{x,z\in \Z^d} \P_r(0\leftrightarrow a)\P_r(a\leftrightarrow x)\P_r(a\leftrightarrow b)\P_r(b\leftrightarrow y) \P_r(b\leftrightarrow z) (\|x\|_2^2+\|y\|_2^2).
\end{align*}
We can then get bounds of the desired order for both expressions using \cref{lem:spatial_triangle,lem:gyration_lower_bound} and the triangle bounds of \eqref{eq:triangle_case_2} and \eqref{eq:triangle_case_3} via the same argument used above.
\end{proof}

\begin{proof}[Proof of \cref{lem:gyration_derivative}, Cases 2 and 3]
Summing the estimate \eqref{lem:gyration_derivative3} over an orthonormal basis yields that
\begin{equation*}
 \frac{d}{dr} \E_{\beta_c,r} \left[\sum_{x\in K} \|x\|_2^2 \right]
=  \frac{2 \alpha}{r} \left(1\pm O\left(\frac{1}{(\log r)^2}\right)\right) \E_{\beta_c,r} \left[\sum_{x\in K} \|x\|_2^2 \right] + \frac{1}{r}\frac{\alpha^2}{\beta_c} r^{2+\alpha} \int_{B} \| y\|_2^2 \dif y ,
\end{equation*}
where we used \cref{lem:gyration_lower_bound} to absorb the error into the first term. Since $(\log r)^{-2}$ is logarithmically integrable, it follows by the same analysis as in the proof of \cref{prop:radius_of_gyration} given \cref{lem:gyration_derivative} that 
\begin{equation}
\label{eq:gyration_cases2_and_3}
  \E_{\beta_c,r}\sum_{x\in K} \|x\|_2^2 \sim \text{const.} \begin{cases} r^{2\alpha} & \text{Case 2}\\
  r^{4} \log r & \text{Case 3}.
  \end{cases}
\end{equation}
To conclude the proof, it suffices to prove that
\begin{equation}
  \E_{\beta_c,r}\sum_{x\in K} \langle x,u \rangle^2 \succeq \E_{\beta_c,r}\sum_{x\in K} \|x\|_2^2
\label{eq:gyration_isotropy}
\end{equation}
for every $r\geq 1$ and unit vector $u\in \R^d$. Let $R=R(r)=r^\alpha$ in Case 2 and $R=R(r)=r \sqrt{\log r}$ in Case 3. It follows from \eqref{eq:gyration_lower_proof} and \eqref{eq:gyration_cases2_and_3} that there exists a constant $\lambda$ such that
\[
  \E_{\beta_c,r} |\{x\in K : \lambda^{-1}R \leq \|x\|\leq \lambda R \}| \geq \frac{1}{2}  \E_{\beta_c,r} |K| \succeq r^\alpha
\]
for every $r\geq 1$. Fixing this constant $\lambda$, we have by our assumed two-point function bounds that
\begin{align}
  \E_{\beta_c,r} |\{x\in K : \lambda^{-1}R \leq \|x\|\leq \lambda R, \langle x,u\rangle \leq \eps R \}| &\preceq \eps R^d \begin{cases} R^{-d+2} & \text{Case 2}\\
  (\log R)^{-1} R^{-d+2} & \text{Case 3}
  \end{cases}\nonumber\\
  & \preceq \eps r^\alpha
  \label{eq:gyration_isotropy2.5}
\end{align}
for every $\eps>0$ and $r\geq 1$, which is easily seen to imply \eqref{eq:gyration_isotropy} in conjunction with \cref{prop:first_moment} and \eqref{eq:gyration_cases2_and_3}. We note for future application that \eqref{eq:gyration_isotropy2.5} also implies the more general bound
\begin{equation}
  \E_{\beta_c,r}\sum_{x\in K} \langle x,u \rangle^{p} \succeq_p \begin{cases}
  r^{(1+p/2)\alpha} & \text{Case 2}\\
  r^{2+p} (\log r)^p & \text{Case 3}
  \end{cases} 
\label{eq:gyration_isotropy2}
\end{equation}
for every $r,p\geq 1$.
\end{proof}



\subsection{The full displacement distribution}
\label{subsec:the_full_displacement_distribution}

In this section we extend the analysis of the previous section to compute the full limiting distribution of a uniform random point of the cluster $K$ under the size-biased measure $\hat\E_{\beta_c,r}$ as $r\to \infty$. The following proposition characterises this limiting distribution in terms of its moments; we will give a more probabilistic description of the relevant distributions afterward the statement.

\begin{prop}
\label{prop:displacement_moments}
Suppose that at least one of the four hypotheses of \eqref{HD+} holds.
\begin{enumerate}
  \item If $\alpha \geq 2$ 
then 
\[
  \E_{\beta_c,r}\sum_{x\in K}\langle x,u\rangle^{2p} \sim \frac{(2p)!}{2^p} \left(\frac{\E_{\beta_c,r}\sum_{x\in K}\langle x,u\rangle^{2}}{\E_{\beta_c,r}|K|}\right)^p \E_{\beta_c,r}|K| 
\]
as $r\to \infty$ for each $p\geq 0$ and $u\in \R^d$.
\item If $\alpha<2$ then for each $u\in \R^d\setminus \{0\}$ 
there exists 
 a sequence of positive coefficients $A_{2p}(u)$ such that
\[
  \E_{\beta_c,r}\sum_{x\in K}\langle x,u \rangle^{2p} \sim A_{2p}(u) r^{2p} \E_{\beta_c,r}|K|
\]
as $r\to \infty$ for each $p\geq 0$. Moreover, these coefficients satisfy
\[
  \sum_{p=0}^\infty \frac{A_{2p}(u)}{(2p)!} =  \left(1-\sum_{p=1}^\infty \frac{\alpha \int_B \langle u,y\rangle^{2p} \dif y}{(2p)!(2p-\alpha)}\right)^{-1}\]
  for all $u$ in an open neighbourhood of the origin in which all relevant sums converge.
\end{enumerate}
\end{prop}

Let us now interpret these asymptotic formulae in terms of their consequences for the scaling limit.
In the case $\alpha\geq 2$, the prefactors $p!/2^{p/2} \mathbbm{1}(p$ even$)$ are equal to the moments of a \emph{Laplace distribution} of parameter $\sqrt{2}$, that is, the product of an exponential random variable of rate $\sqrt{2}$ with an independent uniform random element of $\{+1,-1\}$. The Laplace distribution also arises as a Gaussian random variable with exponential random variance, or equivalently as the product of a standard Gaussian with the square root of an independent exponential random variable. Indeed, if $N$ is a standard Gaussian and $E$ is an independent exponential random variable of rate $1$ then
\[
  \E [(E^{1/2}N)^p] = \E E^{p/2} \E N^{p}  = \begin{cases} (p/2)! (p-1)!! &p \text{ even}\\
  0&p \text{ odd}\end{cases} = \begin{cases} 2^{-p/2} p! &p \text{ even}\\
  0&p \text{ odd}\end{cases},
\]
where the second equality follows from the identity $(2k-1)!!=(2k)!/2^k k!$, so that $E^{1/2}N$ is distributed as a Laplace random variable of parameter $\sqrt{2}$. It follows similarly that if $N$ is a centered $n$-dimensional Gaussian with covariance matrix $\Sigma$ and $E$ is an independent exponential random variable of rate $1$ then
\[
  \E[\langle E^{1/2} N, u \rangle^{p}] = \langle u, \Sigma u \rangle^{p/2} \frac{p!}{2^{p/2}} \mathbbm{1}(p \text{ even})
\]
for every $u\in \R^d$ and integer $p\geq 0$.
As such, it follows from \cref{prop:displacement_moments} that if at least one of the hypotheses of \eqref{HD+} holds, $\alpha\geq 2$, 
and we write $X$ for a uniform random point of the cluster $K$ then
the law of $\xi_2(r)^{-1} X$ under the size-biased measure $\hat \E_{\beta_c,r}$ converges in distribution to the law of the product $E^{1/2} N$, where $N$ is a centered $d$-dimensional Gaussian with covariance matrix $\Sigma$ and $E$ is an independent exponential random variable of rate $1$. 
This law is sometimes known as a \textbf{multivariate Laplace distribution}, and has density expressible in terms of the modified Bessel function of the second kind \cite{kotz2001laplace}.
To generalize correctly to the case $\alpha<2$ (and get the correct interpretation in terms of the limiting superprocess), we should think of this multivariate Laplace distribution as the distribution of a Brownian motion in $\R^d$ with covariance matrix $\Sigma$ started at $0$ and stopped at an independent $\operatorname{Exp}(1)$-distributed random time. (The fact that we evaluate our Brownian motions at a rate $1$ exponential random time is related to a distributional identity for continuum random trees of chi-squared random mass stated in \cref{thm:chi-squared_CRT_distances}, below.)

\medskip

We now explain the analogous interpretation of \cref{prop:displacement_moments} in the case $\alpha<2$. (Throughout this discussion the reader should bare in mind our running convention \eqref{eq:normalization_conventions} that the unit ball of the norm $\|\cdot\|$ has unit measure.) We begin with some relevant background on L\'evy processes, referring the reader to e.g.\ \cite{ken1999levy} for a more detailed introduction. 
Recall that the distribution of a L\'evy process $(L_t)_{t\geq 0}$ on $\R^d$ is characterized via a pair $(a,\Sigma,\Pi)$ consisting of a constants $b\in \R^d$, a $d$-dimensional covariance matrix $\Sigma$,  and a \textbf{L\'evy measure} $\Pi$ on $\R^d$ via the characteristic function equation
\[
  \E [ e^{i\langle L_1,u\rangle} ] = \exp\left[i\langle u,a\rangle - \frac{1}{2}\langle u,\Sigma u \rangle + \int \bigl(e^{i\langle u,x\rangle} -1 -i \langle u,x\mathbbm{1}(\|x\| \leq 1) \rangle \bigr) \dif \Pi(x)\right],
\]
where the L\'evy measure $\Pi$ must satisfy $\Pi(\{0\})=0$ and 
\[
  \int (\|x\|_2^2 \wedge 1) \dif \Pi(x) < \infty.
\]
(In particular, $\Pi$ must be locally finite on $\R^d\setminus \{0\}$.)
This is known as the \textbf{L\'evy-Khintchine representation} of the L\'evy process $L$.
When the L\'evy measure $\Pi$ is symmetric under the map $x\mapsto -x$ and $a$ and $\Sigma$ are both zero we refer to $L$ as a \textbf{symmetric L\'evy jump process}.
Note that the term $x\mathbbm{1}(\|x\| \leq 1)$ can be replaced by any other compactly supported function that is equal to the identity in a neighbourhood of the origin after possibly changing the constant $a$ (this change of $a$ is not necessary for symmetric L\'evy measures and symmetric choices of function); for our purposes it will be convenient to define it using the same norm used to define our long-range percolation model.
If $\int_{\|x\|\geq 1} e^{\lambda \|x\|_2} \dif \Pi(x)<\infty$ for some $\lambda>0$ then we obtain a similar formula for the moment generating function via analytic continuation, so that
\[
  \E [ e^{\langle L_1,u\rangle} ] = \exp\left[\langle u,b\rangle+ \frac{1}{2}\langle u, \Sigma u \rangle + \int \bigl(e^{\langle u,x\rangle} -1 - \langle u,x\mathbbm{1}(\|x\| \leq 1) \rangle\bigr)  \dif \Pi(x)\right]
\]
for all $u$ in some open neighbourhood of the origin. 
%
Note that the above formulae also determine the distribution of $L_t$ for every $t>0$ via the relationships $\E[e^{i\langle L_t,u\rangle}]=\E[e^{i\langle L_1,u\rangle}]^t$ and $\E[e^{\langle L_t,u\rangle}]=\E[e^{\langle L_1,u\rangle}]^t$.

\medskip

The relevant L\'evy measures for our problem, which we denote by $\Pi_1$ and $\Pi_\infty$, are absolutely continuous with respect to Lebesgue measure with densities given by
\begin{equation}
\label{eq:Levy_measure}
  \frac{d\Pi_1(x)}{dx} = \alpha \mathbbm{1}(\|x\|\leq 1)\int_{\|x\|}^1 s^{-d-\alpha-1}\dif s \qquad \text{ and } \qquad \frac{d\Pi_\infty(x)}{dx} = \alpha \int_{\|x\|}^\infty s^{-d-\alpha-1}\dif s.
\end{equation}
The measures $\Pi_1$ and $\Pi_\infty$ will be used to describe the scaling limits of our model with and without cut-off, respectively.
These measure are indeed L\'evy measures when $0<\alpha<2$, and $\Pi_\infty$ describes the law of a symmetric $\alpha$-stable L\'evy jump process: If $(L_t^\infty)_{t\geq 0}$ is the L\'evy process associated to $\Pi_\infty$ then we can compute that $\lambda L^\infty_{\lambda^{-\alpha} t}$ has the same characteristic function, and hence the same distribution, as $L_t^\infty$ for each $t,\lambda>0$. Similarly, 
if $(L_t^1)_{t\geq0}$ is the symmetric L\'evy jump process with L\'evy measure $\Pi_1$ then we can compute that $(L^\lambda_t)_{t\geq 0} = (\lambda L_{\lambda^{-\alpha} t}^1)_{t\geq 0}$ has L\'evy measure $\Pi_\lambda$ with density
\[
 \frac{d\Pi_{\lambda}}{dx}(x) = \lambda^{d-\alpha} \frac{d\Pi_{1}}{dx}(\lambda^{-1} x)  = \alpha \mathbbm{1}(\|x\|\leq \lambda) \int_{\|x\|}^{\lambda} s^{-d-\alpha-1} \dif s,
\]
which converges in both $L^\infty$ and local $L^1$ to the density of $\Pi_\infty$ as $\lambda \to \infty$.

\medskip

 For $\Pi_1$, the associated symmetric L\'evy jump process has moment generating function defined for all $u\in \R^d$ by
\[
  \E[e^{\langle L_t,u\rangle}] = \exp\left[t\int_{\R^d} \bigl(e^{\langle u,x\rangle} -1 - \langle u,x \rangle\bigr)  \dif \Pi_1(x)\right] = 
  \exp\left[t\int_{\R^d} (\cosh(\langle u,x\rangle) -1) \dif \Pi_1(x)\right],
\]
where the second equality follows by symmetry of $\Pi_1$. As such, if $T$ is an independent rate $1$ exponential random variable then
\begin{multline}
  \E[e^{\langle L_T,u\rangle}] = \int_0^\infty \exp\left[- t+t\int (\cosh(\langle u,x\rangle) -1)  \dif \Pi_1(x)\right] \dif t
  \\
  = 
\left(1-\int_{\R^d} (\cosh(\langle u,x\rangle) -1)  \dif \Pi_1(x)\right)^{-1}
\end{multline}
for $u$ in an open neighbourhood of the origin. Taking the Taylor expansion of the hyperbolic cosine and using Fubini to exchange the order of summation and integration leads to the equality 
\[
  \int_{\R^d} \bigl(e^{\langle u,x\rangle} -1 - \langle u,x \rangle\bigr)  \dif \Pi(x) = \sum_{p=1}^\infty \frac{\alpha \int_B \langle u,y\rangle^{2p} \dif y}{(2p)!(2p-\alpha)},
\]
so that
\[
\E[e^{\langle L_T,u\rangle}]
   =  \left(1-\sum_{p=1}^\infty \frac{\alpha \int_B \langle u,y\rangle^{2p} \dif y}{(2p)!(2p-\alpha)}\right)^{-1}.\]
   Thus, \cref{prop:displacement_moments} shows in this case that if 
   $X$ denotes a uniform random point of the cluster $K$, then
the law of $r^{-1} X$ under the size-biased measure $\hat \E_r$ converges in distribution to the law of $L_T^1$ where $(L^1_t)_{t\geq 0}$ is a L\'evy jump process with L\'evy measure $\Pi_1$ and $T$ is an independent exponential random variable of rate $1$. (Again, the fact that we evaluate our L\'evy process at a rate $1$ exponential random time is related to \cref{thm:chi-squared_CRT_distances}.)

\begin{remark}
The scaling factor $\alpha$ appearing in \eqref{eq:Levy_measure} is needed for the measures we are about to define to coincide with those appearing in \cref{prop:displacement_moments} and, as far as we can tell, does not arise in an obvious way by taking the continuum limit of our kernel $J$. This particular normalization leads to the identity 
\begin{equation}
\Pi_\infty(\R^d)-\Pi_1(\R^d)= \int_{\R^d} \left(\alpha \int_{\|x\|}^\infty s^{-d-\alpha-1}\dif s - \alpha \mathbbm{1}(\|x\|\leq 1)\int_{\|x\|}^1 s^{-d-\alpha-1}\dif s\right) \dif x =1,
\end{equation}
so that the number of additional jumps that appear when passing from this truncated L\'evy process to the associated $\alpha$-stable L\'evy process has mean $1$. These measures also satisfy
\begin{equation}
\label{eq:Pi_derivative_pointwise}
\frac{\partial}{\partial\lambda} \frac{d\Pi_\lambda}{d x} (x) = \alpha \mathbbm{1}(\|x\| \leq \lambda) \lambda^{-d-\alpha-1} 
\end{equation}
in the sense of weak derivatives and
\begin{equation}
\label{eq:Pi_derivative_total}
  \frac{d}{d\lambda} \Pi_\lambda(\R^d) = \int \frac{d}{d\lambda} \Pi_\lambda(x) \dif x = \alpha \lambda^{-\alpha-1}
\end{equation}
for every $\lambda>0$.
\end{remark}

Let us now set some notational conventions regarding these limiting distributions and the associated scaling factors.

\begin{defn}[The scaling factor $\sigma$ and the limiting displacement law]
\label{def:limiting_displacement_law}
Write $\sigma(r)=r$ if $\alpha<2$ and $\sigma(r)=\xi_2(r)$ if $\alpha\geq 2$, where $\xi_2(r)$ denotes the radius of gyration.
We define the \textbf{limiting displacement law} $\nu_\mathrm{disp}$ to be either the multivariate Laplace distribution with covariance matrix $\Sigma$ (i.e., the law of the Brownian motion with covariance matrix $\Sigma$ stopped at an independent Exp$(1)$-distributed random time) if $\alpha\geq 2$ or the law of the symmetric L\'evy jump process with L\'evy measure $\Pi_1$ (defined in \eqref{eq:Levy_measure}) stopped at an independent Exp$(1)$-distributed random time if $\alpha<2$, so that if $X$ denotes a uniform random point of the cluster $K$ then the law of $\frac{1}{\sigma(r)}X$ under $\hat \P_{\beta_c,r}$ converges to $\nu_\mathrm{disp}$ as $r\to \infty$ under the hypotheses of \cref{prop:radius_of_gyration}. We will also write $\kappa$ for the density of $\nu_\mathrm{disp}$, whose existence is ensured by the following lemma.
\end{defn}

\begin{lemma}
\label{lem:cts_density}
The measure $\nu_\mathrm{disp}$ is absolutely continuous with respect to Lebesgue measure.
\end{lemma}

\begin{proof}[Proof of \cref{lem:cts_density}]
When $\alpha\geq 2$ this is an immediate consequence of the explicit description of $\nu_\mathrm{disp}$ as a Laplace distribution.
When $\alpha<2$, the fact that the law of the Levy process $L_t$ associated to the measure $\Pi$ evaluated at a fixed time $t$ is absolutely continuous with respect to Lebesgue follows immediately from the fact that the finite measure $(\|x\|_2\wedge 1) \Pi$ is absolutely continuous with respect to Lebesgue measure \cite[Theorem 27.7]{ken1999levy}, and the claim follows since a convex combination of absolutely continuous measures is absolutely continuous.
\end{proof}


We now turn to the proof of \cref{prop:displacement_moments}. As usual, we begin by establishing a relevant asymptotic ODE.

\begin{lemma}
\label{lem:displacement_distribution_ODE}
Suppose that at least one of the four hypotheses of \eqref{HD+} holds. Then
\begin{multline*}
  \frac{d}{dr} \E_{\beta_c,r} \sum_{x\in K}\langle x,u\rangle^{2p} \sim \frac{2 \alpha}{r} \E_{\beta_c,r} \sum_{x\in K}\langle x,u\rangle^{2p} 
  \\+ \beta_c r^{-\alpha-1}\sum_{\substack{a+b+c=p \\ (a,b,c)\neq (p,0,0),(0,p,0)}} \binom{2p}{2a,2b,2c} \E_{\beta_c,r} \left[\sum_{x\in K}\langle x,u\rangle^{2a}\right]\E_r \left[\sum_{x\in K}\langle x,u\rangle^{2b}\right]  r^{2c} \int_B \langle y,u\rangle^{2c} \dif y
\end{multline*}
as $r\to \infty$ for each $u\in \R^d$ and each integer $p\geq 1$.
\end{lemma}

As in previous sections, Cases 2 and 3 require a slightly different treatment and are deferred until the end of the section.

\begin{proof}[Proof of \cref{lem:displacement_distribution_ODE}, Cases 1 and 4]
This proof is very similar to that of the analogous asymptotic ODE in the proof of \cref{prop:radius_of_gyration}, so we will be brief. We write $\E_r=\E_{\beta_c,r}$ as usual. The case $u=0$ is vacuously true so we may suppose that $u\neq 0$.
We begin by writing down the derivative formula
\begin{multline}\frac{1}{\beta_c |J'(r)|}\frac{d}{dr}\E_r\sum_{x\in K}\langle x,u\rangle^{2p}= 2 \E_r\left[\sum_{y\in B_r}\mathbbm{1}(y\notin K) |K_y| \sum_{x\in K}\langle x,u\rangle^{2p} \right]
 \\
+ \sum_{\substack{a+b+c=p \\ (a,b,c) \neq (p,0,0),(0,p,0)}} \binom{2p}{2a,2b,2c}
\E_r\left[\sum_{y\in B_r} \langle y,u\rangle^{2c}\mathbbm{1}(y\notin K) \sum_{x\in K}\langle x,u\rangle^{2a}\sum_{z\in K_y}\langle z-y,u\rangle^{2b} \right]
\\
+ \sum_{\substack{a+b+c=2p\\
\text{not all even}}}
 (-1)^a \binom{2p}{a,b,c}
\E_r\left[\sum_{y\in B_r} \langle y,u\rangle^{c}\mathbbm{1}(y\notin K) \sum_{x\in K}\langle x,u\rangle^{a}\sum_{z\in K_y}\langle z-y,u\rangle^{b} \right],
\label{eq:Z2p1_derivative_exact}
 \end{multline}
which follows from Russo's formula by using the mass-transport principle, writing $z-x=(z-y)+y+(-x)$ and expanding out each resulting trinomial exactly as in the proof of \cref{lem:gyration_derivative}.
For the terms involving odd exponents, we have as in the proof of \cref{lem:gyration_derivative} that
\begin{align*}
&\E_r\left[\sum_{y\in B_r} \langle y,u\rangle^{c}\mathbbm{1}(y\notin K) \sum_{x\in K}\langle x,u\rangle^{a}\sum_{z\in K_y}\langle z-y,u\rangle^{b} \right]
\\&\hspace{2.8cm}= \sum_{y\in B_r} \sum_{x,z\in \Z^d} \langle x,u\rangle^a \langle z-y,u\rangle^b \langle y,u\rangle^c\, \P_r(0\leftrightarrow x \nleftrightarrow y \leftrightarrow z)
\\&\hspace{2.8cm}=
\sum_{y\in B_r} \sum_{x,z\in \Z^d} \langle x,u\rangle^a \langle z-y,u\rangle^b \langle y,u\rangle^c [\P_r(0\leftrightarrow x \nleftrightarrow y \leftrightarrow z) - \P_r(0\leftrightarrow x)\P_r(y \leftrightarrow z)]
\end{align*}
and hence by \cref{lem:disjoint_connections} that
\begin{align*}
&\left|\E_r\left[\sum_{y\in B_r} \langle y,u\rangle^{c}\mathbbm{1}(y\notin K) \sum_{x\in K}\langle x,u\rangle^{a}\sum_{z\in K_y}\langle z-y,u\rangle^{b} \right]\right|
\\
&\hspace{6cm}\leq 
\sum_{y\in B_r} \sum_{x,z\in \Z^d} |\langle x,u\rangle|^a |\langle z-y,u\rangle|^b |\langle y,u\rangle|^c\, \P_r(0\leftrightarrow x \leftrightarrow y \leftrightarrow z)
\\
&\hspace{6cm}\preceq_p \E_r\left[ |K\cap B_r| |K| \sum_{x\in K} \langle x,u\rangle^{2p} \right]
+
r^{2p}\E_r[ |K\cap B_r| |K|^2].
\end{align*}
On the other hand, we have by a familiar application of \cref{lem:BK_disjoint_clusters_covariance} that 
\begin{multline*}
\left|\E_r\left[\sum_{y\in B_r}\mathbbm{1}(y\notin K) |K_y| \sum_{x\in K}\langle x,u\rangle^{2p} \right] - |B_r| \E_r|K| \E_r \left[\sum_{x\in K}\langle x,u\rangle^{2p}\right]\right| 
\\\leq 
\E_r\left[|K\cap B_r| |K| \sum_{x\in K}\langle x,u\rangle^{2p} \right]
\end{multline*}
and
\begin{align*}
&\Biggl|\E_r\left[\sum_{y\in B_r} \langle y,u\rangle^{2c}\mathbbm{1}(y\notin K) \sum_{x\in K}\langle x,u\rangle^{2a}\sum_{z\in K_y}\langle z-y,u\rangle^{2b} \right] 
\\&\hspace{7.5cm}-
\sum_{y\in B_r} \langle y,u\rangle^{2c}\E_r \left[\sum_{x\in K}\langle x,u\rangle^{2a}\right]
\E_r \left[\sum_{x\in K}\langle x,u\rangle^{2b}\right]\Biggr|
\\&\hspace{3cm}\leq
\E_r\left[\sum_{y\in B_r \cap K} \langle y,u\rangle^{2c} \sum_{x\in K}\langle x,u\rangle^{2a}\sum_{z\in K_y}\langle z-y,u\rangle^{2b} \right]
\\
&\hspace{3cm}\preceq_p 
r^{2c} 
\E_r\left[|K\cap B_r| |K| \sum_{x\in K}\langle x,u\rangle^{2a+2b} \right]
+
r^{2c+2b} 
\E_r\left[|K\cap B_r| |K| \sum_{x\in K}\langle x,u\rangle^{2a} \right],
\end{align*}
where we used the AM-GM inequality and symmetry considerations to obtain the simplified bound in the last line. Thus, to obtain the claimed asymptotic ODE from \eqref{eq:Z2p1_derivative_exact}, it suffices to prove that
\[
  \E_r\left[|K\cap B_r| |K| \sum_{x\in K}\langle x,u\rangle^{2p} \right] =o\left(|B_r| \E_r|K| \max\left\{\E_r\sum_{x\in K}\langle x,u\rangle^{2p}, r^{2p} \E_r|K|\right\}\right)
\]
as $r\to \infty$ for every $p\geq 0$. When $d>3\alpha$ this can be deduced directly from \cref{lem:spatial_tree_graph} as in the proof of \cref{prop:radius_of_gyration}. The case $d=3\alpha$ can be deduced using \cref{lem:spatial_tree_graph} and the four-point Gladkov inequality (\cref{lem:four_point_Gladkov}) via exactly the same calculation performed in the proof of \cref{prop:radius_of_gyration}.
\end{proof}

\begin{proof}[Proof of \cref{prop:displacement_moments} given \cref{lem:displacement_distribution_ODE}]
We continue to write $\E_r=\E_{\beta_c,r}$. Fix $u\in \R^d \setminus \{0\}$, the case $u=0$ holding vacuously. We will first prove that either
\[
  \E_r \sum_{x\in K} \langle x,u\rangle^{2p} \sim A_{2p} \left(\frac{\E_r\sum_{x\in K} \langle x,u \rangle^2}{\E_r|K|}\right)^p \E_r|K|
\]
or 
\[
  \E_r \sum_{x\in K} \langle x,u\rangle^{2p} \sim A_{2p}(u)  r^{2p}\E_r|K|
\]
for some positive constants $A_{2p}$ and $A_{2p}(u)$ as appropriate according to whether $\alpha \geq 2$ or $\alpha<2$, and then give interpretations of these constants by analyzing the recurrence relations they satisfy. These asymptotic relationships have already been established for $p=0,1$ in \cref{prop:first_moment,prop:radius_of_gyration}, so we may prove the claim by induction on $p$ starting with $p=2$.

\medskip

Let $p\geq 2$ and suppose that the claim has been proven for all strictly smaller values of $p$. If $\alpha\geq 2$ then it follows from the induction hypothesis and \cref{prop:radius_of_gyration} that all terms involving $r^{2c}\int_B \langle y,u \rangle^{2c}\dif y$ for $c>0$ are negligible compared with the $c=0$ terms, so that the asymptotic ODE simplifies to
\begin{multline}
  \frac{d}{dr} \E_r \sum_{x\in K}\langle x,u\rangle^{2p} \sim \frac{2 \alpha}{r} \E_r \sum_{x\in K}\langle x,u\rangle^{2p} 
  \\+ \beta_c r^{-\alpha-1}\sum_{a=1}^{p-1} \binom{2p}{2a} \E_r \left[\sum_{x\in K}\langle x,u\rangle^{2a}\right]\E_r \left[\sum_{x\in K}\langle x,u\rangle^{2p-2a}\right].
\end{multline}
Using the induction hypothesis together with the asymptotic relationship $\beta_c r^{-\alpha} \sim \alpha (\E_r|K|)^{-1}$ yields that
\begin{multline}
  \frac{d}{dr} \E_r \sum_{x\in K}\langle x,u\rangle^{2p} \sim \frac{2 \alpha}{r} \E_r \sum_{x\in K}\langle x,u\rangle^{2p} 
  \\+ \frac{\alpha}{r}\sum_{a=1}^{p-1} \binom{2p}{2a} A_{2a}A_{2p-2a} \left(\frac{\E_r\sum_{x\in K} \langle x,u \rangle^2}{\E_r|K|}\right)^p \E_r|K|,
\end{multline}
and since $\left(\E_r\sum_{x\in K} \langle x,u \rangle^2/\E_r|K|\right)^p \E_r|K|$ is regularly varying of index $p\alpha+\alpha$ (by \cref{prop:first_moment,prop:radius_of_gyration}) we deduce from \cref{lem:ODE_with_driving_term} that
\[
  \E_r \sum_{x\in K}\langle x,u\rangle^{2p} \sim A_{2p} \left(\frac{\E_r\sum_{x\in K} \langle x,u \rangle^2}{\E_r|K|}\right)^p \E_r|K| \quad \text{ with } \quad A_{2p}:=\frac{1}{p-1} \sum_{a=1}^{p-1} \binom{2p}{2a} A_{2a}A_{2b},
\]
completing the induction step in this case. If on the other hand $\alpha<2$ then it follows from \cref{lem:displacement_distribution_ODE} and the induction hypothesis that
  \begin{multline*}
  \frac{d}{dr} \E_r \sum_{x\in K}\langle x,u\rangle^{2p} \sim \frac{2 \alpha}{r} \E_r \sum_{x\in K}\langle x,u\rangle^{2p} 
  \\+ \frac{\alpha}{r}\left(\sum_{\substack{a+b+c=p \\ (a,b,c)\neq (p,0,0),(0,p,0)}} \binom{2p}{2a,2b,2c} A_{2a}(u)A_{2b}(u) \int_B \langle y,u\rangle^{2c} \dif y \right) r^{2p} \E_r|K|
\end{multline*}
and hence by \cref{lem:ODE_with_driving_term} that $\E_r\sum_{x\in K}\langle x,u\rangle^{2p} \sim A_{2p}(u)r^{2p}\E_r|K|$ with
\[
  A_{2p}(u):= \frac{\alpha}{2p-\alpha}\sum_{\substack{a+b+c=p \\ (a,b,c)\neq (p,0,0),(0,p,0)}} \binom{2p}{2a,2b,2c} A_{2a}(u)A_{2b}(u) \int_B \langle y,u\rangle^{2c} \dif y,
\]
completing the induction step in this case. 

\medskip

It remains to solve these recurrence relations (which in the case $\alpha<2$ means giving a simple expression for the exponential generating function, rather than a formula for the constants \emph{per se}). In the case $\alpha \geq 2$, we have by definition that $A_0=A_2=1$, with the recurrence relation uniquely determining the remaining coefficients. As such, it suffices to check that
\[
 \frac{1}{p-1} \sum_{a=1}^{p-1} \binom{2p}{2a} \frac{(2a)!}{2^a}\frac{(2p-2a)!}{2^{p-a}} = \frac{1}{p-1} \sum_{a=1}^{p-1} \frac{(2p)!}{(2a)!(2p-2a)!} \frac{(2a)!}{2^a}\frac{(2p-2a)!}{2^{p-a}} = \frac{(2p)!}{2^p}
\]
is indeed a solution to the recurrence relation. In the case $\alpha<2$, we rewrite the recurrence relation as
\begin{multline*}
  A_{2p}(u)= \frac{\alpha}{2p-\alpha}\sum_{\substack{a+b+c=p \\ a,b>0}} \binom{2p}{2a,2b,2c} A_{2a}(u)A_{2b}(u) \int_B \langle y,u\rangle^{2c} \dif y
  \\+ 
  \frac{2\alpha}{2p-\alpha}\sum_{a=1}^{p-1} \binom{2p}{2a} A_{2a}(u) \int_B \langle y,u\rangle^{2p-2a} \dif y
  + \frac{\alpha}{2p-\alpha} \int_B \langle y,u\rangle^{2p}\dif y
\end{multline*}
and then as
\begin{multline*}
  (2p+\alpha)A_{2p}(u)= \alpha \sum_{\substack{a+b+c=p \\ a,b>0}} \binom{2p}{2a,2b,2c} A_{2a}(u)A_{2b}(u) \int_B \langle y,u\rangle^{2c} \dif y
  \\+ 
  2\alpha \sum_{a=1}^{p} \binom{2p}{2a} A_{2a}(u) \int_B \langle y,u\rangle^{2p-2a} \dif y + \alpha \int_B \langle y,u\rangle^{2p}\dif y
\end{multline*}
for $p>0$.
Letting $\mathcal{A}(t) = \sum_{p=1}^\infty \frac{t^{2p}A_{2p}(u)}{(2p)!}$ and $\mathcal{B}(t)=\sum_{p=0}^\infty \frac{t^{2p}}{(2p)!}\int_B \langle y,u \rangle^{2p} \dif y$ be the formal exponential generating functions associated to the two sequences, we can sum this recurrence relation over $p\geq 1$ to obtain the formal equality of exponential generating functions
\begin{equation}
\label{eq:A_B_ODE}
t\mathcal{A}' + \alpha \mathcal{A}
  = \alpha \mathcal{A}^2\mathcal{B}  + 
  2\alpha \mathcal{A}\mathcal{B} + \alpha (\mathcal{B}-1)= (\mathcal{A}+1)^2 \alpha \mathcal{B} -\alpha.
\end{equation}
On the other hand, if we define
\[
  \mathcal{C} = \sum_{p=1}^\infty \frac{ \alpha t^{2p}}{(2p-\alpha)(2p)!}\int_B \langle y,u \rangle^{2p} \dif y
\]
then $\mathcal{B}$ and $\mathcal{C}$ are related via
\[
  \alpha (\mathcal{B}-1) = t\mathcal{C}' - \alpha \mathcal{C},
\]
so that the above equation can be rewritten
\[
t\mathcal{A}' + \alpha \mathcal{A} +\alpha
  = (1+\mathcal{A})^2  (t \mathcal{C}'-\alpha \mathcal{C} + \alpha).
\]
It can be verified by calculus that if two differentiable functions $f$ and $g$ do not take the value $1$ and satisfy $f=g/(1-g)$ (which is equivalent to $g=f/(1-f)$ and to $1+f=(1-g)^{-1}$) on some open interval then they satisfy the differential equation
\[
tf' + \alpha f +\alpha
  = (1+f)^2  (t g'-\alpha g + \alpha)
\]
on the same interval, for any value of $\alpha$.
It follows that the real-analytic function defined on an open neighbourhood of $0$ by
\[
  1 + \mathcal{A}(t) = (1-\mathcal{C}(t))^{-1}
\]
is a solution to the ODE \eqref{eq:A_B_ODE} and hence must equal the exponential generating function of our sequence $A_{2p}(u)$ (since the sequence is uniquely determined by the recurrence and the coefficients of the Taylor expansion of $\mathcal{A}$ must satisfy the same recurrence in order to solve the ODE).
\end{proof}

\begin{remark}
In \cref{prop:recurrence_from_derivative} we give an alternative, arguably less elementary, explanation for why the moments of the superprocess (and hence of the exponentially stopped Brownian motion or L\'evy process) satisfy the recurrence relations derived above. This alternative derivation shows that these recurrence relations are the infinitesimal encoding of the scale-invariance of the limit object: In the Brownian case, they follow from the distributional identity $B_T \eqd \lambda B_{\lambda^2 T}$ by noting that the $\lambda$-derivatives of the moments of $\lambda B_{\lambda^2 T}$ must vanish at $\lambda=1$ and expanding this identity out using the product rule.
\end{remark}

\begin{proof}[Proof of \cref{lem:displacement_distribution_ODE}, Cases 3 and 4] 
We will be brief since the proof is similar to (but slightly simpler than) the proof of the relevant cases of \cref{lem:gyration_derivative}. As in that proof, it suffices to prove that
\begin{multline}
  \frac{d}{dr} \E_{\beta_c,r} \sum_{x\in K}\langle x,u\rangle^{2p} = \frac{2 \alpha}{r} \E_{\beta_c,r} \sum_{x\in K}\langle x,u\rangle^{2p} 
  \\+ \beta_c r^{-\alpha-1}\sum_{\substack{a+b+c=p \\ (a,b,c)\neq (p,0,0),(0,p,0)}} \binom{2p}{2a,2b,2c} \E_{\beta_c,r} \left[\sum_{x\in K}\langle x,u\rangle^{2a}\right]\E_r \left[\sum_{x\in K}\langle x,u\rangle^{2b}\right]  r^{2c} \int_B \langle y,u\rangle^{2c} \dif y
  \\ \pm o\left(r^{-1} \E_{\beta_c,r} \sum_{x\in K}\|x\|^{2p}_2\right)
  \label{eq:displacement_p_moment_derivative_cases_2_and_3}
\end{multline}
as $r\to \infty$ for each fixed $u\in \R^d \setminus \{0\}$ and integer $p\geq 1$. 
Indeed, once this is proven we can sum over an orthonormal basis and inductively apply \cref{lem:ODE_with_driving_term} (using \cref{prop:radius_of_gyration} as the base case) to obtain that
that
\[
  \E_{\beta_c,r} \sum_{x\in K}\langle x,u\rangle^{2p} \sim \begin{cases} 
  C_p r^{(p+1)\alpha} & \text{Case 2}\\
  C_p r^{2(p+1)} (\log r)^p & \text{Case 3}
  \end{cases}
\]
for some positive constants $C_p$.
Together with \eqref{eq:gyration_isotropy2} this implies that $\E_{\beta_c,r} \sum_{x\in K}\langle x,u\rangle^{2p} \asymp \E_{\beta_c,r} \sum_{x\in K}\| x\|_2^{2p}$ for every unit vector $u\in \R^d$, which implies the claim in conjunction with \eqref{eq:displacement_p_moment_derivative_cases_2_and_3}.

We now turn to the proof of \cref{eq:displacement_p_moment_derivative_cases_2_and_3}. This estimate follows by exactly the same calculation as in \cref{lem:gyration_derivative3} except that the error terms are now bounded by constant multiples of
\begin{equation}
\label{eq:cases3and4_error_5}
  r^{\alpha-1} \frac{1}{r^d}\sum_{z\in B_r} \sum_{x,y\in \Z^d} (\|x\|_2^{2p}+\|y-x\|_2^{2p}) \P_{\beta_c,r}(0\leftrightarrow x)\P_{\beta_c,r}(x\leftrightarrow y)\P_{\beta_c,r}(y\leftrightarrow z) 
\end{equation}
and 
\begin{equation}
\label{eq:cases3and4_error_6}
  r^{-1} \left[ \frac{1}{r^d}\sum_{y\in B_r}\nabla_{\beta_c}(0,y)\right] \E_{\beta_c,r}\sum_{x\in K}\|x\|_2^{2p}.
\end{equation}
(In the proof of \cref{lem:gyration_derivative3} the errors were of the same form with $p=1$; they look different here because we have now included the $|J'(r)|\asymp r^{-d-\alpha-1}$ factor.)
We need to show that both of these quantities are negligible compared to $r^{-1} \E_{\beta_c,r}\sum_{x\in K}\|x\|_2^{2p}$ (but are no longer required to check logarithmic integrability). 
We also already know by \cref{prop:radius_of_gyration} and H\"older's inequality that
\begin{equation}
\label{eq:Lp_gyration_lower}
  \E_{\beta_c,r} \sum_{x\in K}\|x\|^{2p}_2 \geq \left(\frac{\E_{\beta_c,r} \sum_{x\in K}\|x\|^{2}_2}{\E_{\beta_c,r}|K|}\right)^p \E_{\beta_c,r}|K| \succeq \begin{cases} r^{(p+1)\alpha} & \text{Case 2}\\
  r^{(p+1)\alpha}(\log r)^p & \text{Case 3},
  \end{cases}
\end{equation}
which will be useful for showing that the error \eqref{eq:cases3and4_error_5} is indeed negligible. 
For \eqref{eq:cases3and4_error_6}, the claimed negligibility follows immediately from e.g.\ \cref{lem:log_integrable_triangle}. For \eqref{eq:cases3and4_error_5}, writing $T_r(x,y)=\P_{\beta_c,r}(x\leftrightarrow y)$, it suffices by commutativity of convolution on $\Z^d$ to prove that
\begin{equation}
\label{eq:spatial_triangle_convolution}
 r^{\alpha-d} \sum_{z\in B_r}\sum_{x\in \Z^d} T_r(0,x)\|x\|_2^{2p} T^2_r(x,z) = o\left( \E_{\beta_c,r}\sum_{x\in K}\|x\|_2^{2p} \right),
\end{equation}
where $T^2_r$ denotes the square of $T_r$ as a matrix.
%
To do this, we can separately consider the contributions to this sum from the sets $B_{2r}$ and $B_{2r}^c=\Z^d\setminus B_{2r}$.
For the first set of points, we have that
\[
  r^{\alpha-d}\sum_{z\in B_r}\sum_{x\in B_{2r}} T_r(0,x)\|x\|_2^{2p}T^2_r(x,z)\preceq 
  r^{\alpha-d+2p} \sum_{z\in B_r} \nabla_{\beta_c}(0,z) = o\left(r^{\alpha+2p}\right)=o\left(\E_{\beta_c,r}\sum_{x\in K}\|x\|^{2p}_2 \right)
\]
by \cref{lem:log_integrable_triangle} and \eqref{eq:Lp_gyration_lower} as required. For the second set of points we have that $\|x-z\|\geq \|x\|/2$  for every $x\in B_{2r}^c$ and $z\in B_r$. Thus, using the estimate on $T_r(x,y)=\tau*\tau(x-y)$ from \cref{lem:log_integrable_triangle}, we obtain that
%
%
%
\begin{equation*}
  \sum_{x\in B_{2r}^c} T_r(0,x)\|x\|_2^{2p} T^2_r(x,z)
   \preceq 
 \begin{cases} \sum_{x\in \Z^d} \langle x\rangle^{2p-d+4} T_r(0,x) & \text{Case 2}\\
 \sum_{x\in \Z^d} \frac{\langle x\rangle^{2p-2}}{(\log \langle x \rangle)^2} T_r(0,x) &\text{Case 3}
 \end{cases},
\end{equation*}
for every $z\in B_r$ and hence that
\[
r^{\alpha-d}\sum_{z\in B_r}\sum_{x\in B_{2r}} T(0,x)\|x\|_2^{2p}T(x,z)^2
%
=o\left(r^{\alpha}\sum_{x\in \Z^d} \langle x\rangle^{2p-2} T_r(0,x)\right)\]
in both Case 2 and Case 3. Now, we also have by H\"older's inequality that
\[
  \frac{\sum_{x\in \Z^d} \langle x\rangle^{2p} T_r(0,x)}{\sum_{x\in \Z^d} \langle x\rangle^{2p-2} T_r(0,x)} \geq \frac{\sum_{x\in \Z^d} \langle x\rangle^{2} T_r(0,x)}{\sum_{x\in \Z^d}  T_r(0,x)} \asymp \begin{cases} r^\alpha & \text{Case 2}\\
  r^2 \log r & \text{Case 3}
  \end{cases}
\]
for every $p\geq 1$, which ensures that
\[
r^{\alpha-d}\sum_{z\in B_r}\sum_{x\in B_{2r}} T_r(0,x)\|x\|_2^{2p}T_r^2(x,z) = o\left( \E_{\beta_c,r}\sum_{x\in K}\|x\|_2^{2p}\right)
\]
as required.
\end{proof}

\begin{remark}
Let us note now for future application that an identical argument (using the upper bound on the three-fold convolution instead of the two-fold convolution of the two-point function from the proof of \cref{lem:log_integrable_triangle}) yields the analogous bound
\begin{multline}
r^{-d}\sum_{z\in B_r}\sum_{x\in \Z^d} T_r(0,x)\|x\|_2^{2p}T_r^3(x,z) \preceq r^{2p-d}\sum_{z\in B_r} T_r^{4}(0,z) + o\left(
  \E_{\beta_c,r}\sum_{x\in K}\|x\|_2^{2p}
\right)
\\= o\left( \E_{\beta_c,r}\sum_{x\in K}\|x\|_2^{2p}\right).
\label{eq:square_diagram_spatial_error}
\end{multline}
Similar estimates for larger diagrams will require a slightly different treatment (since the $k$-fold convolution of the critical two-point function may have infinite entries for $k\geq 4$) and will be returned to in the proof of \cref{lem:higher_spatial_moments_ODE}.
\end{remark}

\subsection{The full scaling limit of the size-biased model with cut-off}
\label{subsec:the_full_scaling_limit_of_the_size_biased_model_with_cut_off}

In this section we prove the following theorem and explain how it leads to a full scaling limit of clusters (considered as random measures) under the size-biased, cut-off measure $\hat \E_{\beta_c,r}$. We will then explain how this leads to scaling limit results for the measure $\E_{\beta_c}$ without cut-off in \cref{subsec:scaling_limits_without_cut_off}.
The statement of the theorem will use the kernel $\kappa:\R^d \to [0,\infty)$ and the scaling factor $\sigma(r)$ as defined in \cref{def:limiting_displacement_law}. We define a \textbf{polynomial} on $\R^d$ to be an element of the algebra generated by the linear functionals $\{x\mapsto \langle x,u\rangle:u\in \R^d\}$ and the constant $1$ (equivalently, a polynomial on $\R^d$ is a polynomial in the coordinates).

\begin{thm}
\label{thm:scaling_limit_diagrams}
Suppose at least one of the hypotheses of \eqref{HD+} holds.
 For each $n\geq 1$ and polynomials $P_1,\ldots,P_n :\R^d\to \R$, the asymptotic formula
\begin{multline*}
  \E_{\beta_c,r} \left[ \sum_{x_1,\ldots,x_n \in K} \prod_{i=1}^n P_i\left(\frac{x_i}{\sigma(r)}\right)\right] \\= 
\left(\sum_{T\in \mathbb{T}_n} \idotsint \prod_{\substack{i<j \\ i\sim j}} \kappa(x_i-x_j) \prod_{i=1}^n P_i(x_i)\dif x_1 \cdots \dif x_{2n-1} \pm o(1)\right) (\hat \E_{\beta_c,r}|K|)^{n-1} \E_{\beta_c,r}|K|
\end{multline*}
holds as $r\to \infty$,
where 
 the sum is taken over isomorphism classes of trees with leaves labelled $0,1,\ldots,n$ and unlabelled non-leaf vertices all of degree $3$. 
\end{thm}

Here we do not write $\sim$ but instead add a $\pm o(1)$ term to the constant prefactor to account for the possibility that this constant is zero, in which case we do not claim the left hand side is identically zero for large $r$. 
Note that if $P_i\equiv 1$ for every $1\leq i \leq n$ then the each of the integrals on the right hand side is $1$ and, since $|\mathbb{T}_n|=(2n-3)!!$, we recover the asymptotic formula \[\E_{\beta_c,r}|K|^n \sim (2n-3)!!(\hat\E_{\beta_c,r}|K|)^{n-1}\E_{\beta_c,r}|K| \] 
established in \cref{prop:higher_moments}. 

\medskip

\cref{thm:scaling_limit_diagrams} is consistent with asymptotic expressions for $k$-point functions under the cut-off measure $\P_{\beta_c,r}$ of the form
\begin{align*}
  \P_{\beta_c,r}(x\leftrightarrow y\leftrightarrow z) &\approx 
  \begin{array}{l}
\includegraphics{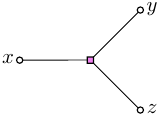}
  \end{array} = V_r \sum_{w\in \Z^d} \P_{\beta_c,r}(x\leftrightarrow w)\P_{\beta_c,r}(w\leftrightarrow y)\P_{\beta_c,r}(w\leftrightarrow z),
\\
   \P_{\beta_c,r}(x\leftrightarrow y\leftrightarrow z\leftrightarrow w) &\approx 
  \begin{array}{l}
\includegraphics{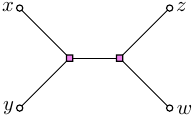}
  \end{array}
  +
    \begin{array}{l}
\includegraphics{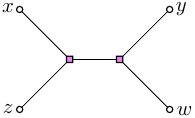}
  \end{array}
  +
    \begin{array}{l}
\includegraphics{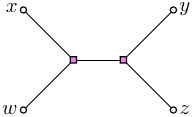}
  \end{array},
\end{align*}
and so on, where each line represents a copy of the two-point function
\begin{equation}
  \P_{\beta_c,r}(x\leftrightarrow y) \approx r^{-d}\E_{\beta_c,r}|K| \kappa\Bigl(\frac{x-y}{r}\Bigr) \approx \frac{\alpha}{\beta_c}r^{-d+\alpha} \kappa\Bigl(\frac{x-y}{r}\Bigr),
\label{eq:two_point_kappa_heuristic}
\end{equation}
internal vertices are summed over, and the violet box at each internal vertex indicates the inclusion of a copy of the \textbf{vertex factor}
\begin{equation}
\label{eq:vertex_factor_def}
  V_r := \frac{\E_{\beta_c,r}|K|^2}{(\E_{\beta_c,r}|K|)^3}. 
\end{equation}
  The main caveat here is that \cref{thm:scaling_limit_diagrams} establishes these asymptotic estimates only in the sense of moments, rather than pointwise, and indeed these estimates might not be accurate for points with distance much smaller than $r$. (We conjecture that these estimates do hold pointwise in the sense of first-order asymptotics when all distances between points are of order $r$.)
  The vertex factor $V_r$ is asymptotic to $\frac{\beta_c^2}{\alpha^2}A_r$ where $A_r$ is the quantity from \cref{thm:hd_moments_main,thm:critical_dim_moments_main_slowly_varying} and accounts for the interactions between the different parts of the cluster that prevent the tree-graph upper bound
$\P_r(x\leftrightarrow y\leftrightarrow z) \leq \sum_{w\in \Z^d} \P_r(x\leftrightarrow w)\P_r(w\leftrightarrow y)\P_r(w\leftrightarrow z)$. In high effective dimensions it follows from \cref{thm:hd_moments_main} that the vertex factor converges to a constant, meaning that the interaction between different parts of the cluster occurs primarily on microscopic scales (i.e., on scales of the same order as the lattice spacing). Meanwhile, when $d=3\alpha<6$ we expect that the vertex factor is divergently small (as we prove in \cref{III-sec:logarithmic_corrections_at_the_critical_dimension}) and that the interactions between different parts of the cluster are spread out over a large range of mesoscopic scales.

\medskip

Before proving \cref{thm:scaling_limit_diagrams}, we first establish its consequences regarding superprocess scaling limits for the size-biased cut-off measure $\hat \P_{\beta_c,r}$.

\begin{corollary}[Scaling limit of the size-biased model with cut-off]
\label{cor:scaling_limit_cut_off}
Suppose at least one of the hypotheses of \eqref{HD+} holds, and for each $r>0$ write $\hat \P_{\beta_c,r}$ for probabilities associated to the size-biased measure $\hat \E_{\beta_c,r}$. Then
\[
  \hat \P_{\beta_c,r}\left( \frac{1}{\hat \E_{\beta_c,r}|K|} \sum_{x\in K} \delta_{x/\sigma(r)} \in \cdot 
  \right) \to \mathbb{K}
\]
weakly as $r\to \infty$, 
where $\mathbb{K}$ is the law of the integrated superprocess whose underlying genealogical tree is a continuum random tree of total mass distributed as the square of a standard Gaussian and whose spatial motion is either Brownian motion with covariance matrix $\Sigma$ defined in \cref{prop:radius_of_gyration} or the L\'evy process with L\'evy measure $\Pi_1$ defined in \eqref{eq:Levy_measure} according to whether $\alpha\geq 2$ or $\alpha<2$.
\end{corollary}

\begin{remark}
We use the scaling convention in which the continuum random tree is encoded by $2e$ where $e$ is a Brownian excursion conditioned to have duration $1$. We warn the reader that at least two other scaling conventions appear in the literature.
\end{remark}

We will deduce \cref{cor:scaling_limit_cut_off} from \cref{thm:scaling_limit_diagrams} together with the following property of the continuum random tree.

\begin{theorem}
\label{thm:chi-squared_CRT_distances}
Let $(T,X_0)$ be the rooted continuum random tree whose mass $\mu$ is distributed as the square of a standard Gaussian and let $X_1,\ldots,X_n$ be uniform random points in $T$. The distribution of the array of distances $(d(X_i,X_j))_{0\leq i, j \leq n}$ under the measure obtained from biasing the law of $(T,X_0)$ by $\mu^{n-1}$ is equal to that obtained by picking a uniform random element of $\mathbb{T}_n$ and setting the edges of the tree to have independent \emph{Exp}$(1)$-distributed random lengths.
\end{theorem}

Equivalently, if $(T,X_0)$ is the rooted continuum random tree whose mass $\mu$ is distributed as a chi-squared distribution with $2n-1$ degrees of freedom and the points $X_1,\ldots,X_n$ are chosen independently at random from the mass measure then the distances $(d(X_i,X_j))_{0\leq i, j \leq n}$ have distribution equal to that obtained by picking a uniform random element of $\mathbb{T}_n$ and setting the edges of the tree to have independent Exp$(1)$-distributed random lengths. This property is implicit in \cite{addario2018voronoi}, where it is mentioned that there \emph{exists} a mass distribution making this distributional identity hold but this distribution is not identified as a chi-squared distribution.

\begin{remark}
The corresponding statement for the canonical measure is that the edge lengths are ``distributed'' according to Lebesgue measure \cite[Chapter III, Theorem 4]{le1999spatial}. This leads to the diagrammatic formula \eqref{eq:canonical_measure_diagrams} in the same way that \cref{thm:chi-squared_CRT_distances} leads to the diagrammatic formula \eqref{eq:bbK_diagram}.
\end{remark}

\begin{proof}[Proof of \cref{thm:chi-squared_CRT_distances}]
As proven by Aldous \cite{aldous1993continuum} and stated clearly in \cite[Eq. (2.1)]{addario2018voronoi}, if we sample a continuum random tree of unit mass and then let $X_1,\ldots,X_n$ be points in the tree chosen independently at random from the mass measure, the distribution of the array of distances $(d(X_i,X_j))_{0\leq i, j \leq n}$ is equal to that obtained by picking a uniform element of $\mathbb{T}_n$, taking the total length of all edges in this tree to have distribution
\[
\frac{2}{2^{(m+1)/2}\Gamma(\frac{m+1}{2})}x^m e^{-\frac{1}{2}x^2} 
\]
where $m=2n-1$ is the number of edges in the tree, then taking a uniform random subdivision of this total length to determine the edges of the tree. (The formula given in \cite{addario2018voronoi} differs from the one given here by various factors of $2^{1/2}$ due to our different scaling convention for the CRT; see \cite[Eq.\ (34)]{aldous1993continuum} and \cite[Page 51]{le1999spatial} for an equivalent density formula established under the same scaling conventions as ours.)
 We need to prove that if we instead take the tree to have mass distributed as a chi-squared distribution with $2n-1$ degrees of freedom (which, as discussed above, is equivalent to biasing our chi-squared distribution with one degree of freedom by the total mass to the power $n-1$) then the total length instead follows a Gamma$(m)$ distribution; it is then a classical fact that a uniform subdivision of a Gamma$(m)$ is equivalent to a sequence of $m$ independent exponential random variables. 
If the mass of the tree is rescaled by a factor $\lambda$ then the law of all distances are rescaled by a factor $\lambda^{1/2}$, so that for a continuum random tree of mass $\lambda$ the relevant density is
\[
 \frac{2}{(2\lambda)^{(m+1)/2} \Gamma(\frac{m+1}{2})}x^m e^{-\frac{1}{2}\lambda^{-1} x^2}.
\]
Thus, to prove the theorem it suffices to verify that
\[
 \int_0^\infty \frac{\lambda^{(m-1)/2}}{\sqrt{2\pi \lambda}}  e^{-\lambda/2} \cdot \frac{2}{(2\lambda)^{(m+1)/2} \Gamma(\frac{m+1}{2})}x^m e^{-\frac{x^2}{2\lambda}} \dif \lambda \propto  
  x^{m-1} e^{-x}
\]
or equivalently that
\[
   \int_0^\infty \lambda^{-3/2} e^{-\frac{\lambda}{2}-\frac{x^2}{2\lambda}} \dif \lambda \propto  
  x^{-1} e^{-x}.
\]
The integral on the left can be recognized as one of the many integral representations of the modified Bessel function of the second kind \cite[\S 7.12, eq.\ (23)]{bateman1953higher}:
\[
  K_{\nu}(z)=\frac{1}{2} z^\nu \int_0^\infty e^{-\frac{1}{2}(s+\frac{z^2}{s})} s^{-\nu-1} \dif s,
\]
which holds for $z$ of positive real part, and the claim follows since 
\[
  K_{1/2}(z) = \sqrt{\frac{\pi}{2z}}e^{-z}
\]
as proven in \cite[\S 7.2.6, eq.\ (42)]{bateman1953higher}.
\end{proof}

We are now ready to deduce \cref{cor:scaling_limit_cut_off} from \cref{thm:scaling_limit_diagrams}.

\begin{proof}[Proof of \cref{cor:scaling_limit_cut_off}]
It follows immediately from \cref{thm:scaling_limit_diagrams} that
\begin{multline*}
  \lim_{r\to \infty} \hat \E_{\beta_c,r} \left[ \frac{1}{(\hat \E_{\beta_c,r}|K|)^{n-1}} \frac{1}{|K|}\sum_{x_1,\ldots,x_n \in K} \prod_{i=1}^n P_i\left(\frac{x_i}{\sigma(r)}\right)\right] \\= 
\sum_{T\in \mathbb{T}_n} \idotsint \prod_{\substack{i<j\\i\sim j}} \kappa(x_i-x_j) \prod_{i=1}^n P_i(x_i)\dif x_1 \cdots \dif x_{2n-1} 
\end{multline*}
for each $n\geq 1$ and collection of polynomials $P_1,\ldots,P_n:\R^d\to \R$. 
Since the limiting (chi-squared) distribution of the volume of the cluster has  moments satisfying the Carleman condition (which is weaker than having an exponential moment), it follows by standard properties of moment measures \cite{zessin1983method} that the weak limit of the law of the random measures $(\hat \E_{\beta_c,r}|K|)^{-1} \sum_{x\in K} \delta_{x/\sigma(r)}$ under the measures $\hat \E_{\beta_c,r}$ is determined by this asymptotic formula. It remains only to prove that this limit coincides with the integrated superprocess given in the statement, i.e., that
\begin{equation}
\label{eq:bbK_diagram}
  \bbK\left[ \frac{1}{\mu(\R^d)} \prod_{i=1}^n \int P_i(x) \dif \mu(x) \right] =
\sum_{T\in \bbT_n} \idotsint \prod_{\substack{i<j\\i\sim j}} \kappa(x_i-x_j) \prod_{i=1}^n P_i(x_i) \dif x_1 \cdots \dif x_{2n-1}
\end{equation}
for every $n\geq 1$ and collection of polynomials $P_1,\ldots,P_n:\R^d\to \R$. This in turn follows from \cref{thm:chi-squared_CRT_distances} and the fact that $\nu_\mathrm{disp}$ is the law of the appropriate L\'evy process stopped at an Exp$(1)$ random time (since the superprocess is defined as the L\'evy process indexed by the continuum random tree).
\end{proof}

We now turn to the proof of \cref{thm:scaling_limit_diagrams}. 
We will make use of a generalization of our ODE lemmas \cref{lem:ODE_with_driving_term} to signed functions.


\begin{lemma}[ODE analysis for signed functions]
\label{lem:signed_ODE_analysis}
Let $a,b>0$, let $h:(0,\infty)\to (0,\infty)$ be a measurable, regularly varying function of index $b$, and suppose that 
$f:(0,\infty)\to \R$ is a (not necessarily positive) differentiable function such that 
\begin{equation}
\label{eq:signed_ODE_f}
f'(r) = \frac{(a \pm o(1)) f(r) + (A \pm o(1)) h(r)}{r}
\end{equation}
as $r\to \infty$ for some constant $A\in \R$. If $b>a$ and $|f(r)|=O(h(r))$ as $r\to \infty$ then
\[
  f(r) = \frac{A \pm o(1)}{b-a}h(r)
\]
as $r\to \infty$.
\end{lemma}


\begin{proof}[Proof of \cref{lem:signed_ODE_analysis}]
Let $\lambda > 0$ be a constant and consider the function defined for $r\geq r_0$ by
\[
 f_\lambda(r)= f(r)+\lambda \int_{r_0}^r \frac{h(s)}{s} \dif s = f(r)+(1\pm o(1))\frac{\lambda  }{b} h(r)
\]
where $r_0$ is sufficiently large that $h$ is locally integrable on $[r_0,\infty)$. The assumption that $f(r)=O(h(r))$ as $r\to \infty$ ensures that we can take $\lambda$ sufficiently large that
$A+\frac{b-a}{b}\lambda>0$ and
 $f_\lambda(r)$ is non-negative for all sufficiently large $r$. Since we also have that
\[
  f_\lambda'(r) = f'(r) + \lambda h(r) = \frac{a\pm o(1)}{r} f_\lambda(r) + \frac{1}{r}\left(A+\frac{b-a}{b} \lambda \pm o(1)\right)h,
\]
it follows from \cref{lem:ODE_with_driving_term} that
\[
  f_\lambda(r) \sim \frac{1}{b-a}\left(A+\frac{b-a}{b} \lambda \right)  h(r) = \frac{(A \pm o(1)) }{b-a}h(r) + \frac{\lambda}{b}  h(r)
\]
as $r\to \infty$, which is equivalent to the claim.
\end{proof}

To apply this lemma, we will need appropriate order estimates on moments, which will be deduced from \cref{prop:higher_moments,prop:displacement_moments}.
 We define a \textbf{monomial} on $\R^d$ to be a product function of the form $x\mapsto \prod_{i=1}^k \langle x,u_i\rangle$ with $u_1,\ldots,u_k$ unit vectors in $\R^d$ (including the case of the empty product $x\mapsto 1$). We refer to $k\geq 0$ as the \textbf{degree} of the monomial.

\begin{lemma}
\label{lem:mixed_moments_order_estimates}
 Suppose that at least one of the four hypotheses of \eqref{HD+} holds. There exists a constant $C$ such that if $P_1,\ldots,P_n$ are monomials of total degree $p$ then
\[
\E_{\beta_c,r} \left[ \sum_{x_1,\ldots,x_n \in K} \prod_{i=1}^n |P_i(x_i)|\right] \leq (1+o(1))C^{n+p} n! p! \left( \sigma(r)^{p} (\hat \E_{\beta_c,r}|K|)^{n-1} \E_{\beta_c,r}|K|\right)
\]
as $r\to \infty$.
\end{lemma}

\begin{proof}[Proof of \cref{lem:mixed_moments_order_estimates}]
 We may suppose that $p>0$, the case $p=0$ having already been treated in \cref{prop:higher_moments}. Write $P_i(x)=\prod_{j=1}^{k_i} \langle x,u_{i,j}\rangle$ for each $1\leq i \leq n$, where the $u_{i,j}$ are all unit vectors.  We have as in \eqref{eq:spatial_AM_GM} that
\begin{multline*}
  \E_{\beta_c,r} \left[ \sum_{x_1,\ldots,x_n \in K} \prod_{i=1}^n |P_i(x_i)|\right] = \E_{\beta_c,r} \left[\sum_{x_1,\ldots,x_n\in K} \prod_{i=1}^n \prod_{j=1}^{k_i} |\langle x_i,u_{i,j}\rangle|^{p_i} \right]
  \\\leq 
  \max_{\substack{1\leq i\leq n  \\1\leq j \leq k_i}} \E_{\beta_c,r} \left[|K|^{n-1} \sum_{x\in K}|\langle x,u_{i,j}\rangle|^{p} \right] 
  = \max_{\substack{1\leq i\leq n \\ 1\leq j \leq k_i}} \E_{\beta_c,r}|K|  \hat\E_{\beta_c,r} \left[|K|^{n-1} |\langle X,u_{i,j}\rangle|^{p} \right],
  \end{multline*}
  where $X$ is a uniform random element of $K$. Applying Cauchy-Schwarz, we obtain that
  \[
   \E_{\beta_c,r} \left[ \sum_{x_1,\ldots,x_n \in K} \prod_{i=1}^n |P_i(x_i)|\right]
  \leq 
  \max_{\substack{1\leq i\leq n \\ 1\leq j \leq k_i}} \E_{\beta_c,r}|K| \sqrt{\hat\E_{\beta_c,r} |K|^{2n-2} \hat \E_{\beta_c,r}|\langle X,u_{i,j}\rangle|^{2p}}.
  \]
  It follows from \cref{prop:higher_moments,prop:displacement_moments} that there exists a constant $C$ such that
  \[
    \hat\E_{\beta_c,r} |K|^{2n-2} \leq (1+o(1))(4n-5)!! (\hat \E_{\beta_c,r}|K|)^{2n-2} \; \text{ and } \; \hat \E_{\beta_c,r}|\langle X,u_i\rangle|^{2p} \leq (1+o(1))C^p (2p)! \sigma(r)^{2p}
  \]
 (in the case $\alpha<2$ this follows from the fact that the exponential generating function in \cref{prop:displacement_moments} converges for all $u$ in some neighbourhood of the origin) and the claim follows since $\sqrt{(4n-5)!! (2p)!}$ is bounded by $\tilde C^{n+p} n!p!$ for some finite constant $\tilde C$. 
\end{proof}

Given a sequence of monomials $P_1,\ldots,P_n$, we define a \textbf{splitting} of $(P_1,\ldots,P_n)$ to be a tuple $S=(I,(P_{i,1},P_{i,2},P_{i,3})_{i\in I})$ consisting of a non-empty subset $I$ of $\{1,\ldots,n\}$ together with a collection of linear monomials $(P_{i,1},P_{i,2},P_{i,3})_{i\in I}$ such that $P_{i,1}P_{i,2}P_{i,3}=P_i$ for each $i\in I$. For each such splitting we write $N(S)$ to be the number of distinct ways of factoring each $P_i$ in terms of the three monomials $P_{i,1},P_{i,2},$ and $P_{i,3}$. We write $\mathscr{S}(P_1,\ldots,P_n)$ for the set of splittings of the sequence $P_1,\ldots,P_n$. 

\begin{lemma}
\label{lem:higher_spatial_moments_ODE_unsimplified}
Let $n\geq 1$, and let $P_1,\ldots,P_n$ be monomials on $\R^d$. Then
\begin{multline*}
\frac{d}{dr}\E_{\beta_c,r} \left[ \sum_{x_1,\ldots,x_n \in K} \prod_{i=1}^n P_i(x_i)\right] 
=\beta_c |J'(r)| \sum_{S\in \mathscr{S}} N(S) 
\\\cdot\E_{\beta_c,r} \left[
  \sum_{x_1,\ldots,x_n \in K} \sum_{x\in K} \sum_{y\in B_r(x)} \mathbbm{1}(0\nleftrightarrow y) \prod_{i \notin I} P_i(x_i)  \prod_{i \in I} P_{i,1}(x) P_{i,2}(x_i-y) P_{i,3}(y-x) \right]
\end{multline*}
as $r\to \infty$.
\end{lemma}

\begin{proof}[Proof of \cref{lem:higher_spatial_moments_ODE_unsimplified}]
This follows by applying Russo's formula, writing each point $x_i$ in the second cluster as $(x_i-y)+(y-x)+x$ and expanding out each resulting trinomial in a similar manner to the proofs of \cref{lem:gyration_derivative,lem:displacement_distribution_ODE}.
\end{proof}

We next show that the derivative formula of \eqref{lem:higher_spatial_moments_ODE_unsimplified} admits an asymptotic simplifcation under each of the hypotheses of \eqref{HD+}.

\begin{lemma}
\label{lem:higher_spatial_moments_ODE}
Suppose at least one of the hypotheses of \eqref{HD+} holds, let $n\geq 1$, and let $P_1,\ldots,P_n$ be monomials on $\R^d$ of total degree $p$. Then
\begin{multline*}
\frac{d}{dr}\E_{\beta_c,r} \left[ \sum_{x_1,\ldots,x_n \in K} \prod_{i=1}^n P_i(x_i)\right] 
=\beta_c r^{-\alpha-1} \sum_{S\in \mathscr{S}} N(S) 
\\\cdot\E_{\beta_c,r} \left[
 \sum_{x\in K} \sum_{\substack{x_i \in K \\ i \notin I}} \prod_{i \notin I} P_i(x_i)  \prod_{j \in I} P_{j,1}(x) \right] \E_{\beta_c,r}\left[
 \sum_{\substack{x_i \in K \\ i \in I}} \prod_{i\in I} P_{i,2}(x_i) \right] \int_B   \prod_{i \in I} P_{i,3}(y) \dif y
 \\\pm o\left(r^{-1}\sigma(r)^p (\hat \E_{\beta_c,r}|K|)^{n-1}\E_{\beta_c,r}|K|\right)
\end{multline*}
as $r\to \infty$.
\end{lemma}

\begin{proof}[Proof of \cref{lem:higher_spatial_moments_ODE}, Cases 1 and 4]
We will keep the proof as a brief sketch since the details are very similar to those of the relevant cases of \cref{lem:gyration_derivative,lem:displacement_distribution_ODE}. (but are simpler since we already have estimates of the correct order on various quantities).
We continue to write $\E_r=\E_{\beta_c,r}$. All implicit constants appearing in this proof may depend on $p$.
In light of \cref{lem:higher_spatial_moments_ODE_unsimplified}, 
it suffices to prove under the hypotheses of Cases 1 and 2 of \cref{prop:radius_of_gyration} that
\begin{multline}
  \Biggl|\E_{r} \Biggl[
  \sum_{\substack{x_i \in K \\ i \notin I}}  \sum_{x\in K} \sum_{y\in B_r(x)} \mathbbm{1}(0\nleftrightarrow y) \sum_{\substack{x_i \in K_y \\ i \in I}}  \prod_{i \notin I} P_i(x_i)  \prod_{i \in I} P_{i,1}(x) P_{i,2}(x_i-y) P_{i,3}(y-x) \Biggr]
\\
-
 |B_r| \E_r \left[
 \sum_{x\in K} \sum_{\substack{x_i \in K \\ i \notin I}} \prod_{i \notin I} P_i(x_i)  \prod_{j \in I} P_j(x) \right] \E_r\left[
 \sum_{\substack{x_i \in K \\ i \notin I}} P_{i,3}(x_i) \right] \int_B \prod_{i\in I} P_{i,2}(y) \dif y\Biggr|
\\= o\left(r^{d+\alpha}\sigma(r)^p (\hat \E_r|K|)^{n-1}\E_r|K|\right)
\label{eq:splitting_factorization}
\end{multline}
as $r\to \infty$ for each fixed non-empty set $I\subset \{1,2,\ldots n\}$ and assignment of non-negative integers $(a_j,b_j,c_j)$ summing to $p_j$ for each $j\in I$. This follows by using \cref{lem:BK_disjoint_clusters_covariance} to bound the left hand side of this estimate by 
\[
\E_{r} \left[\sum_{x_1,\ldots,x_n \in K} \sum_{x\in K} \sum_{y\in B_r(x)} \mathbbm{1}(0\leftrightarrow y) \prod_{i \notin I} |P_i(x_i)|  \prod_{i \in I} |P_{i,1}(x)| |P_{i,2}(x_i-y)| |P_{i,3}(y-x)| \right].
\]
By ignoring the restriction that $y\in B_r(x)$, it follows straightforwardly from \cref{lem:mixed_moments_order_estimates} that this quantity admits a bound of the form
\begin{equation}
\label{eq:higher_spatial_moments_error1}
  O\left(\sigma(r)^p (\hat \E_r|K|)^{n+1}\E_r|K| \right),
\end{equation}
which is sufficient in the case $d>3\alpha$. When $d=3\alpha$ and the hydrodynamic condition holds, one can proceed by applying the mass-transport principle to exchange the roles of $x$ and $0$ and then applying the Cauchy-Schwarz inequality to factor out the term $\sum_{y\in B_r}\mathbbm{1}(0\leftrightarrow y)=|K\cap B_r|$. This yields a bound of the form
\begin{equation}
\label{eq:higher_spatial_moments_error2}
  O\left(\sqrt{\sigma(r)^{2p} (\hat \E_r|K|)^{2n+1}\E_r|K|}\cdot \sqrt{M_r \E_r|K|} \right) = o\left(r^{\alpha} \sigma(r)^p (\hat \E_r|K|)^{n+1/2}\E_r|K|\right).
\end{equation}
Taking the minimum of the two estimates \eqref{eq:higher_spatial_moments_error1} and \eqref{eq:higher_spatial_moments_error2} yields a bound of the desired form; \eqref{eq:higher_spatial_moments_error1} is sufficient when $\hat \E_r|K|$ is much smaller than $r^{2\alpha}$ (e.g., by a factor of $(r^{-(d+\alpha)/2}M_r)^{1/4}$) while \eqref{eq:higher_spatial_moments_error2} handles the remaining case. 
\end{proof}

\begin{proof}[Proof of \cref{lem:higher_spatial_moments_ODE}, Cases 2 and 3]
We continue to write $\E_r=\E_{\beta_c,r}$ and write $T_r(x,y)=\P_{\beta_c,r}(x\leftrightarrow y)$. It suffices to prove that the estimate \eqref{eq:splitting_factorization} continues to hold in this setting. It follows by an argument similar to that used in \cref{lem:disjoint_connections_triangle_second} that the relevant error term can be bounded by a constant multiple of
\begin{multline*}
\sum_{\ell=3}^{n+2} \left[\sum_{y\in B_r} T_r^\ell(0,y)\right] (\E_{r} |K|)^{2n+2-\ell} (\E_r \sum_{x\in K} \|x\|_2^{p}) 
\\+ \sum_{\ell=3}^{n+2} \left[\sum_{\ell_1+\ell_2+1=\ell}\sum_{y\in B_r} \sum_{a,b\in \Z^d}T_r^{\ell_1}(0,a) \|a-b\|_2^{p} T_r(a,b) T_r^{\ell_2}(b,y) \right] (\E_{r} |K|)^{2n+3-\ell}
\end{multline*}
where, as in several other proofs above, we have bounded all diagrammatic sums in terms of diagrammatic sums in which the spatial weight $\|\cdot\|_2^p$ appears on exactly one edge of the diagram: The first term accounts for diagrams in which this spatial weight appears on an edge that does not lie on the path between $0$ and $y$, while the second term accounts for diagrams in which the spatial weight appears on this path.
Applying \cref{lem:higher_polygon_diagrams} together with \cref{prop:higher_moments,prop:displacement_moments} yields that the first term satisfies
\begin{multline*}
  \sum_{\ell=3}^{n+1} \left[\sum_{y\in B_r} T_r^\ell(0,y)\right] (\E_{r} |K|)^{2n+2-\ell} \E_r \sum_{x\in K} \|x\|_2^{p} = o\left(r^d(\E_{r} |K|)^{2n-1} \E_r \sum_{x\in K} \|x\|_2^{p}\right)
  \\
  = o\left(r^{d+\alpha}\sigma(r)^p (\hat \E_r|K|)^{n-1} \E_r|K|\right),
\end{multline*}
which is of the required order after multiplication by $|J'(r)|\sim r^{-d-\alpha-1}$. For the second, it suffices to prove that
\[
r^{-d} \sum_{y\in B_r} \sum_{a,b\in \Z^d} T_r^{\ell_1}(0,a) \|a-b\|_2^{p} T_r(a,b) T_r^{\ell_2}(b,y) = o\left( (\E_{r}|K|)^{\ell-4} \E_{\beta_c,r}\sum_{x\in K}\|x\|_2^p\right)
\]
for every pair of integers $\ell_1,\ell_2\geq 0$ with $\ell_1+\ell_2+1=\ell\geq 3$.
By commutativity of convolution on $\Z^d$, the left hand side does not depend on the choice of $\ell_1,\ell_2$ for a given value of $\ell=\ell_1+\ell_2+1$.
The cases $\ell=3,4$ have already been treated in \eqref{eq:spatial_triangle_convolution} and \eqref{eq:square_diagram_spatial_error}, so that it suffices to consider the case $\ell \geq 5$. In this case, we can take $\ell_1,\ell_2\geq 2$ and  use Cauchy-Schwarz to obtain that
\begin{align*}
&\sum_{a,b\in \Z^d} T_r^{\ell_1}(0,a) \|a-b\|_2^{p} T_r(a,b) T_r^{\ell_2}(b,y) 
\\
&\hspace{2.5cm}\leq \left[\sum_{a,b\in \Z^d} T_r^{\ell_1}(0,a)^2 \|a-b\|_2^{p} T_r(a,b) \right]^{1/2}\left[\sum_{a,b\in \Z^d} T_r^{\ell_2}(y,b)^2 \|b-a\|_2^{p} T_r(b,a) \right]^{1/2}
\\
&\hspace{2.5cm}= \left[T_r^{2\ell_1}(0,0)T_r^{2\ell_2}(0,0)\right]^{1/2} \E_{r}\sum_{x\in K}\|x\|_2^p = o\left((\E_r|K|)^{(2\ell_1-3+2\ell_2-3)/2} \E_{\beta_c,r}\sum_{x\in K}\|x\|_2^p \right)
\\
&\hspace{2.5cm}= o\left((\E_r|K|)^{\ell-4} \E_{\beta_c,r}\sum_{x\in K}\|x\|_2^p \right)
\end{align*}
as claimed, where we applied \cref{lem:higher_polygon_diagrams} in the penultimate estimate. 
%
\end{proof}

We next prove the following proposition, which establishes \emph{in the scaling limit} the recurrence relation that will be derived from \cref{lem:higher_spatial_moments_ODE}. The proof shows that this recurrence relation encodes the scale-invariance of the model infinitesimally around the scale factor $\lambda=1$.

\begin{prop}[Recurrence relations in the scaling limit]
\label{prop:recurrence_from_derivative}
Suppose at least one of the hypotheses of \eqref{HD+} holds, let $\kappa$ be as in \cref{def:limiting_displacement_law}, and let $T$ be a tree with vertex set $\{0,\ldots,k\}$. If $\alpha\geq 2$ then 
\begin{multline}
 \label{eq:recurrence_from_derivative_SR}
(p+2k)  \idotsint \prod_{\substack{i<j \\ i\sim j}} \kappa(x_i-x_j)  \prod_{i=1}^k P_i(x_i)\dif x_1 \cdots \dif x_{k}
\\=
2  \sum_{\substack{i_0<j_0 \\ i_0\sim j_0}}
\idotsint [\kappa*\kappa](x_{i_0}-x_{j_0})\!\!\!\!\!\!\!\! \prod_{\substack{i<j \\ i\sim j \\ (i,j)\neq (i_0,j_0)}} \!\!\!\!\!\!\!\!\kappa(x_i-x_j)  \prod_{i=1}^k P_i(x_i)\dif x_1 \cdots \dif x_{k}
 \end{multline}
for every sequence of monomials $P_1,\ldots,P_k$ of total degree $p$  
while if $\alpha<2$ then
 \begin{multline}
 \label{eq:recurrence_from_derivative_LR}
(p+\alpha k)  \idotsint \prod_{\substack{i<j \\ i\sim j}} \kappa(x_i-x_j)  \prod_{i=1}^n P_i(x_i)\dif x_1 \cdots \dif x_{k}
\\=
\alpha  \sum_{\substack{i_0<j_0 \\ i_0\sim j_0}}
\idotsint [\kappa*\mathbbm{1}_B*\kappa](x_{i_0}-x_{j_0}) \!\!\!\!\!\!\!\!\prod_{\substack{i<j \\ i\sim j \\ (i,j)\neq (i_0,j_0)}} \!\!\!\!\!\!\!\!\kappa(x_i-x_j)  \prod_{i=1}^k P_i(x_i)\dif x_1 \cdots \dif x_{k}.
 \end{multline}
 for every sequence of monomials $P_1,\ldots,P_k$ of total degree $p$.
\end{prop}

\begin{remark}
\label{remark:convolution_recurrence}
Applying \cref{prop:recurrence_from_derivative} in the special case that $T$ is the line graph with vertex set $\{0,\ldots,n\}$ and the sequence of monomials if of the form $(1,1,\ldots,1,P)$, we obtain that if
 $\alpha\geq 2$ then
  \begin{align}
 \label{eq:recurrence_from_derivative_SR3}
 (\deg(P)+2 n) \int_{\R^d} \kappa^{*n}(y)P(y) \dif y 
  =2 n \int_{\R^d} \kappa^{*(n+1)}(y) P(y) \dif y
\end{align}
 for each $n\geq 1$ and each monomial $P$, while if $\alpha<2$ then
 \begin{align}
 \label{eq:recurrence_from_derivative_LR3}
 &(\deg(P)+\alpha n) \int_{\R^d} \kappa^{*n}(y)P(y) \dif y 
=\alpha n \int_{\R^d} [\kappa^{*(n+1)} * \mathbbm{1}_B](y) P(y) \dif y
\end{align}
for each $n\geq 1$ and each monomial $P$.
The identity \eqref{eq:recurrence_from_derivative_LR3} will later be useful in the proof of \cref{III-thm:correction_moments}.  The $n=1$ case of these identities  can be used to give an alternative proof that the moments $\int \langle x,u\rangle^{2p}\kappa(x)\dif x$ satisfy the recurrence relation in the proof of \cref{prop:displacement_moments}.
\end{remark}

\begin{proof}[Proof of \cref{prop:recurrence_from_derivative}]
Fix a tree $T$ with vertex set $\{0,\ldots,k\}$ and a sequence of monomials $P_1,\ldots,P_k$ of total degree $p$.
For each $\lambda>0$, let $(L_t^\lambda)_{t\geq 0}$ be either Brownian motion with covariance matrix $\Sigma$ (if $\alpha \geq 2$) or the symmetric L\'evy jump process with L\'evy measure $\Pi_\lambda$ (if $\alpha<2$), and write $p^\lambda_t(\cdot,\cdot)$ for the associated transition density. (In particular, $L_t^\lambda$ and $p^\lambda_t$ do not depend on the choice of $\lambda>0$ in the case $\alpha \geq 2$.) As discussed in \cref{subsec:the_full_displacement_distribution}, for each $\lambda>0$ the rescaled process $(\lambda L^1_{\lambda^{-\alpha \wedge 2}t})_{t\geq 0}$ has the same distribution as $(L_t^\lambda)_{t\geq 0}$, and since $\lambda L_t^1$ has density $\lambda^d p_t^\lambda (0,\lambda^{-1}x)$ it follows that
\[
\lambda^{-d} p^1_{\lambda^{-(\alpha\wedge 2)}t}(0,\lambda^{-1}x) = p^\lambda_{t}(0,x) 
\]
and hence that
\[
p^1_t(0,x) = \lambda^d p^\lambda_{\lambda^{\alpha\wedge 2} t}(0,\lambda x)
\]
for every $t,\lambda>0$ and $x\in \R^d$. Integrating over $t$, it follows that the kernel $\kappa$ can be expressed as
\begin{multline*}
  \kappa(x) = \int_0^\infty e^{-t}\lambda^{-d} p^1_t(0,\lambda^{-1}x)\dif t =
  \int_0^\infty e^{-t} \lambda^d p^\lambda_{\lambda^{\alpha\wedge 2} t}(0,\lambda x)  \dif t 
  \\=
  \lambda^{d-(\alpha\wedge 2)} \int_0^\infty e^{-\lambda^{-(\alpha\wedge2)}s}  p^\lambda_{s}(0,\lambda x)  \dif s 
\end{multline*}
for every $x\in \R^d$ and $\lambda>0$.
As such, we have also that
\begin{multline*}
 \idotsint \prod_{\substack{i<j \\ i\sim j}} \kappa(x_i-x_j) \prod_{i=1}^k P_i(x_i)\dif x_1 \cdots \dif x_{k}
\\=
 \idotsint \prod_{\substack{i<j \\ i\sim j}} \left[\lambda^{d-(\alpha\wedge 2)} \int_0^\infty e^{-\lambda^{-(\alpha\wedge 2)}s}  p^\lambda_{s}(0,\lambda (x_i-x_j))  \dif s \right]  \prod_{i=1}^k P_i(x_i)\dif x_1 \cdots \dif x_{k}
\end{multline*}
for every $\lambda>0$. Applying the change of variables $(x_1,\ldots,x_{k})=\lambda^{-1}(z_1,\ldots,z_{k})$ and using that $\prod_{i=1}^n P_i(x_i)$ is homogeneous of degree $p$, we obtain that
\begin{multline*}
 \idotsint \prod_{\substack{i<j \\ i\sim j}} \kappa(x_i-x_j) \prod_{i=1}^k P_i(x_i)\dif x_1 \cdots \dif x_{k}
\\=
\lambda^{-p} \idotsint \prod_{\substack{i<j \\ i\sim j}} \left[\lambda^{-(\alpha\wedge 2)} \int_0^\infty e^{-\lambda^{-(\alpha\wedge 2)}s}  p^\lambda_{s}(0,x_i-x_j)  \dif s \right]  \prod_{i=1}^k P_i(x_i)\dif x_1 \cdots \dif x_{k}
\\
=\lambda^{-k(\alpha\wedge 2)-p} \idotsint \prod_{\substack{i<j \\ i\sim j}} \left[\int_0^\infty e^{-\lambda^{-(\alpha\wedge 2)}s}  p^\lambda_{s}(0,x_i-x_j)  \dif s \right]  \prod_{i=1}^k P_i(x_i)\dif x_1 \cdots \dif x_{k}
\end{multline*}
for every $\lambda>0$. Since the left hand side does not depend on $\lambda$, it follows that
\begin{equation}
\label{eq:scaling_to_recurrence}
\frac{d}{d\lambda} \lambda^{-p-k(\alpha\wedge 2)} \idotsint \prod_{\substack{i<j \\ i\sim j}} \left[\int_0^\infty e^{-\lambda^{-(\alpha\wedge 2)}s}  p^\lambda_{s}(0,x_i-x_j)  \dif s \right]  \prod_{i=1}^k P_i(x_i)\dif x_1 \cdots \dif x_{k}
=0\end{equation}
for every $\lambda>0$.

\medskip

We now argue that the fact that this derivative vanishes at $\lambda=1$ exactly encodes the two recurrence relations \eqref{eq:recurrence_from_derivative_SR} and \eqref{eq:recurrence_from_derivative_LR}. 
To proceed, we claim that
\begin{equation}
\label{eq:lambda_derivative_kappa}
  \frac{d}{d\lambda}\int_0^\infty e^{-\lambda^{-(\alpha\wedge 2)}s}  p^\lambda_{s}(0,x)  \dif s \Biggr|_{\lambda=1} = \begin{cases} 2 \kappa*\kappa(x) & \alpha \geq 2 \\
 \alpha \kappa*\mathbbm{1}_B*\kappa(x) & \alpha <2,
  \end{cases}
\end{equation}
where we recall that $B$ denotes the unit ball for the norm $\|\cdot\|$.
If $\alpha \geq 2$ then $p^\lambda_{s}(0,x)$ does not depend on $\lambda$ and we have simply that
\begin{multline*}
   \frac{d}{d\lambda}\int_0^\infty e^{-\lambda^{-2}s}  p^\lambda_{s}(0,x)   \dif s \Biggr|_{\lambda=1}  = 2 \int_0^\infty s e^{-s}  p_{s}(0,x) \dif s 
   \\=  2 \int_0^\infty \int_0^s \int e^{-t}  p_{t}(0,z)e^{-(s-t)}  p_{s-t}(0,x-z) \dif z \dif t \dif s = 2 \kappa*\kappa(x)
\end{multline*}
as claimed. When $\alpha<2$, we get the same $\alpha \kappa*\kappa$ term from differentiating $e^{-\lambda^{\alpha \wedge 2} s}$, but also get a non-zero contribution from differentiating $p^\lambda_{s}(0,x)$. To compute this contribution, we recall that, when $\lambda > 1$, the two L\'evy processes $L^\lambda=(L^\lambda_t)_{t\geq 0}$ and $L^1=(L^1_t)_{t\geq 0}$ can be coupled so that their difference is a compound Poisson process with intensity measure $\Pi_\lambda-\Pi_1$ independent of $L^1$. 
Using the derivative formulae \eqref{eq:Pi_derivative_pointwise} and \eqref{eq:Pi_derivative_total}, this leads easily to the equality
\begin{align*}
  \frac{\partial}{\partial\lambda} p^\lambda_t(0,x) \Biggr|_{\lambda=1} &
  = \alpha \int_0^t [p_s^1*\mathbbm{1}_B * p_{t-s}^1](0,x) \dif s
  - \alpha t  p^1_t(0,x)
  =\alpha \int_0^t [p_s^1*(\mathbbm{1}_B-1) * p_{t-s}^1](0,x) \dif s
\end{align*}
holding in the sense of weak derivatives, so that integrating over time yields that
\begin{align*}
  \int_0^\infty e^{-\lambda^{-\alpha} s}\frac{\partial}{\partial \lambda} p^\lambda_s(0,x) \dif s \Biggr|_{\lambda=1} &
  = \alpha \kappa * \mathbbm{1}_B * \kappa - \alpha \kappa*\kappa.
\end{align*}
Adding this to the $\alpha \kappa*\kappa$ term that comes from differentiating $e^{-\lambda^{-\alpha} s}$ yields the claimed equality \eqref{eq:lambda_derivative_kappa} in this case.

\medskip

 Applying the product rule to the derivative in \eqref{eq:scaling_to_recurrence} at $\lambda=1$ and applying \eqref{eq:lambda_derivative_kappa} as appropriate yields the two claimed identities \eqref{eq:recurrence_from_derivative_SR} and \eqref{eq:recurrence_from_derivative_LR}.
\end{proof}

\begin{proof}[Proof of \cref{thm:scaling_limit_diagrams}]
We first prove that for each sequence of monomials $P_1,\ldots,P_n$ there exists a constant $A(P_1,\ldots,P_n)$ such that 
\begin{equation*}
  \E_{\beta_c,r} \left[ \sum_{x_1,\ldots,x_n \in K} \prod_{i=1}^n P_i(x_i)\right] \\= 
\left(A(P_1,\ldots,P_n) \pm o(1)\right) \sigma(r)^p(\hat \E_{\beta_c,r}|K|)^{n-1} \E_{\beta_c,r}|K|
\end{equation*}
as $r\to \infty$, where $p$ is the total degree of the monomials $P_1,\ldots,P_n$. (Note that the left hand side is identically zero for $p$ odd, but this does not play an important role in the proof.)
 The proof of this claim will also establish a recurrence relation satisfied by these constants, which we will then argue identifies these constants with the diagrammatic integrals from the statement of the theorem.
We prove the claim by a double induction on $n$ and $p$, in which we prove the claim first for $n=1$ and with $p$ increasing from $0$ to $\infty$, then with $n=2$ and $p$ increasing from $0$ to $\infty$, and so on. The boundary cases $n=1$ and $p=0$ follow from \cref{prop:displacement_moments} and \cref{prop:higher_moments} and the definition of $\nu_\mathrm{disp}$ and $\kappa$. Let $n\geq 2$ and $p\geq 1$ and suppose that the claim has been proven for all smaller pairs $(n,p)$ in the lexicographical order.
 It follows from \cref{lem:higher_spatial_moments_ODE} and the induction hypothesis that
\begin{align}
  \frac{d}{dr}\E_{\beta_c,r} \left[ \sum_{x_1,\ldots,x_n \in K} \prod_{i=1}^n P_i(x_i)\right] &= (n+1)\alpha r^{-1} \E_{\beta_c,r} \left[ \sum_{x_1,\ldots,x_n \in K} \prod_{i=1}^n P_i(x_i)\right] 
  \label{eq:tilde_A_recurrence_ODE}
  \\
  &\hspace{3cm}+
  \alpha r^{-1} \sum_{S\in \mathscr{S}'} \tilde A(S) \sigma(r)^{p-q(S)} r^{q(S)} (\hat\E_r|K|)^{n-1} \E_r|K|
  \nonumber\\
  &\hspace{3cm}\pm o\left(r^{-1} \sigma(r)^p (\hat\E_r|K|)^{n-1} \E_r|K|  \right)
  \nonumber
\end{align}
where $q(S)$ denotes the degree of the monomial $\prod_{i \in I} P_{i,3}$, the constant $\tilde A(S)$ is defined by 
\[
\tilde A(S):= N(S) A\Biggl((P_i)_{i\notin I}, \prod_{j\in I} P_{j,1}\Biggr) A\Biggl((P_{i,2})_{i\in I}\Biggr) \int_B   \prod_{i \in I} P_{i,3}(y) \dif y, \]
and where $\mathscr{S}'\subseteq \mathscr{S}$ denotes the set of splittings whose complement $\mathscr{S}\setminus \mathscr{S}'$ contains the splitting in which $I=\{1,\ldots,n\}$ and $P_{2,i}=P_i$ for every $1\leq i \leq n$ and the splittings in which $I=\{j\}$ is a singleton and $P_{j,1}=P_j$. (The set $\mathscr{S}\setminus \mathscr{S}'$ has $n+1$ elements and accounts for the coefficient of the first term on the right hand side of \eqref{eq:tilde_A_recurrence_ODE}. Note that when $n=1$ the combinatorics are slightly different, but this does not concern us in this proof.) When $\alpha<2$ the other terms on the right hand side of \eqref{eq:tilde_A_recurrence_ODE} 
all have the same order (when their associated constant $\tilde A(S)$ does not vanish) and are regularly varying of index $p+(2n-1)\alpha$, while if $\alpha\geq 2$ then the terms with $q(S)>0$ are of lower order than the terms with $q(S)=0$, which have index of regular variation $\alpha p/2 + (2n-1)\alpha$.
Thus, it follows from \cref{lem:signed_ODE_analysis} (which applies in our setting by \cref{lem:mixed_moments_order_estimates}) that
\[
  \E_{\beta_c,r} \left[ \sum_{x_1,\ldots,x_n \in K} \prod_{i=1}^n P_i(x_i)\right] = (A(P_1,\ldots,P_n) \pm o(1)) \sigma(r)^p (\hat \E_r |K|)^{n-1} \E_r|K|
\]
as claimed, where 
\begin{multline}
\label{eq:complicated_recurrence_small_alpha}
  A(P_1,\ldots,P_n) \\= \frac{\alpha}{p+(n-2)\alpha} \sum_{S \in \mathscr{S}'} N(S) A\Biggl((P_i)_{i\notin I}, \prod_{j\in I} P_{j,1}\Biggr) A\Biggl((P_{i,2})_{i\in I}\Biggr) \int_B   \prod_{i \in I} P_{i,3}(y) \dif y
\end{multline}
when $\alpha<2$ and
\begin{equation}
\label{eq:complicated_recurrence_large_alpha}
  A(P_1,\ldots,P_n) = \frac{1}{p/2+(n-2)} \sum_{S \in \mathscr{S}''} N(S) A\left((P_i)_{i\notin I}, \prod_{j\in I} P_{j,1}\right) A\Biggl((P_{i,2})_{i\in I}\Biggr)
\end{equation}
when $\alpha \geq 2$, where $\mathscr{S}''$ denotes the set of splittings in $\mathscr{S}'$ for which $P_{i,3}=1$ for every $i\in I$.

\medskip

It remains to prove that the recurrence relations \eqref{eq:complicated_recurrence_small_alpha} and \eqref{eq:complicated_recurrence_large_alpha} determine the same asymptotics of moments claimed in the theorem. It follows from \cref{prop:recurrence_from_derivative} (applied with the monomial $1$ on each internal vertex of the tree and $P_1,\ldots,P_n$ on the leaves) that if $\alpha\geq 2$ then
 \begin{multline}
 \label{eq:recurrence_from_derivative_SR}
(p+2(2n-1)) \sum_{T\in \mathbb{T}_n} \idotsint \prod_{\substack{i<j \\ i\sim j}} \kappa(x_i-x_j)  \prod_{i=1}^n P_i(x_i)\dif x_1 \cdots \dif x_{2n-1}
\\=
2 \sum_{T\in \mathbb{T}_n} \sum_{\substack{i_0<j_0 \\ i_0\sim j_0}}
\idotsint [\kappa*\kappa](x_{i_0}-x_{j_0})\!\!\!\!\!\!\!\! \prod_{\substack{i<j \\ i\sim j \\ (i,j)\neq (i_0,j_0)}} \!\!\!\!\!\!\!\!\kappa(x_i-x_j)  \prod_{i=1}^n P_i(x_i)\dif x_1 \cdots \dif x_{2n-1}
 \end{multline}
while if $\alpha<2$ then
 \begin{multline}
 \label{eq:recurrence_from_derivative_LR}
(p+\alpha(2n-1)) \sum_{T\in \mathbb{T}_n} \idotsint \prod_{\substack{i<j \\ i\sim j}} \kappa(x_i-x_j)  \prod_{i=1}^n P_i(x_i)\dif x_1 \cdots \dif x_{2n-1}
\\=
\alpha \sum_{T\in \mathbb{T}_n} \sum_{\substack{i_0<j_0 \\ i_0\sim j_0}}
\idotsint [\kappa*\mathbbm{1}_B*\kappa](x_{i_0}-x_{j_0}) \!\!\!\!\!\!\!\!\prod_{\substack{i<j \\ i\sim j \\ (i,j)\neq (i_0,j_0)}} \!\!\!\!\!\!\!\!\kappa(x_i-x_j)  \prod_{i=1}^n P_i(x_i)\dif x_1 \cdots \dif x_{2n-1}.
 \end{multline}
 If we split the tree $T$ into two pieces by cutting the edge $(i_0,j_0)$ in half and expand out each $P_i(x_i)$ with $i$ in the half of the tree not containing $0$ in an exactly analogous manner to how we derived the recurrences \eqref{eq:complicated_recurrence_small_alpha} and \eqref{eq:complicated_recurrence_large_alpha} above, 
 we obtain that the quantities 
 \[
   \tilde A(P_1,\ldots,P_n) := \sum_{T\in \mathbb{T}_n} \idotsint \prod_{\substack{i<j \\ i\sim j}} \kappa(x_i-x_j) \prod_{i=1}^n P_i(x_i)\dif x_1 \cdots \dif x_{2n-1}
 \]
 satisfy the same recurrence relations \eqref{eq:complicated_recurrence_small_alpha} and \eqref{eq:complicated_recurrence_large_alpha} as the quantities $A(P_1,\ldots,P_n)$. Since the two quantities also coincide for the boundary terms where either $n=1$ or $p=0$ by \cref{prop:higher_moments,prop:displacement_moments}, they are the same for every sequence $P_1,\ldots,P_n$ as claimed. \qedhere 
\end{proof}

\subsection{Scaling limits without cut-off}
\label{subsec:scaling_limits_without_cut_off}


In this section deduce our main theorems on scaling limits, \cref{thm:superprocess_main,thm:superprocess_main_regularly_varying}, which concern the critical model without cut-off converging under rescaling to the canonical measure of an integrated superprocess (which is not a probability measure), from \cref{thm:scaling_limit_diagrams,cor:scaling_limit_cut_off}, which concern the scaling limits of the size-biased, cut-off measures $\hat\E_{\beta_c,r}$.
The proof will also make use of our results on the volume tail established in \cref{sec:volume_tail,sec:analysis_of_moments}, and indeed parts of the proof can be thought of as more elaborate versions of the arguments of \cref{sec:volume_tail}.

\medskip

Throughout this section, we will work under the implicit assumption that at least one of the hypotheses of \eqref{HD+} holds.
We write $p_t^\infty(\cdot,\cdot)$ and
$G(x) = \int_0^\infty p_t^\infty(0,x) \dif t$
for the transition kernel and Greens function associated to either the Brownian motion with covariance matrix $\Sigma$ (if $\alpha\geq 2$) or the symmetric $\alpha$-stable L\'evy process with L\'evy measure $\Pi_\infty$ defined in \eqref{eq:Levy_measure} if $\alpha<2$.
We also let $\sigma(r)=r$ if $\alpha<2$ and $\sigma(r)=\xi_2(r)$ if $\alpha\geq 2$, and define the scaling functions
\[
  \eta(R) = 4\cdot\frac{\E_{\beta_c,\sigma^{-1}(R)}|K|}{\hat \E_{\beta_c,\sigma^{-1}(R)}|K|}
   \qquad \text{ and } \qquad \zeta(R) = \frac{1}{4}\hat \E_{\beta_c,\sigma^{-1}(R)}|K|,
\]
which are regularly varying of index $-\alpha$ and $2\alpha$ respectively by \cref{thm:hd_moments_main,thm:critical_dim_moments_main_slowly_varying}.

\medskip

We will deduce \cref{thm:superprocess_main,thm:superprocess_main_regularly_varying} from the following two lemmas: the first provides tightness while the second yields the canonical measure of the integrated superprocess as a double limit of the measures $\E_{\beta_c,\sigma^{-1}(s R)}$. We write 
\[\mu_R:=\frac{1}{\zeta(R)}\sum_{x\in K} \delta_{x/R}\]
for the appropriately normalized measure associated to the cluster on scale $R$. 
 To lighten notation we also write $\mu[\varphi]=\int \varphi(x)\dif \mu(x)$.

\begin{lemma}[Tightness]
\label{lem:tightness_scaling_limits}
For every $\eps,\delta>0$ there exists $s_0,M<\infty$  such that
\begin{multline*}
  \frac{1}{\eta(R)}\P_{\beta_c,\sigma^{-1}(s R)}\Bigl(\mu_R(\{x \in \R^d : \|x\|\geq s_0\}) \geq \eps 
  \text{ \emph{or} } \mu_R(\R^d) \geq M \Bigr) 
  \\\leq
  \frac{1}{\eta(R)}\P_{\beta_c}\Bigl(\mu_R(\{x \in \R^d : \|x\|\geq s_0\}) \geq \eps 
  \text{ \emph{or} } \mu_R(\R^d) \geq M \Bigr)  
\leq 
  \delta 
\end{multline*}
for all $R\geq R_0$ and $s \geq s_0$.
\end{lemma}

\begin{lemma}[Double limits]
\label{lem:scaling_limit_double_limit}
 We have that
\begin{equation*}
 \lim_{s\to \infty}\lim_{R\to \infty}\frac{1}{\eta(R)} \E_{\beta_c,\sigma^{-1}(s R)} \left[g(\mu_R(\R^d))(\mu[\varphi])^n 
 \right] = \N \Bigl[g(\mu(\R^d)) (\mu[\varphi])^n \Bigr]
\end{equation*}
for each compactly supported continuous function $g:[0,\infty)\to [0,1]$ whose support does not contain~$0$, each bounded, non-negative, continuous function $\varphi:\R^d\to [0,\infty)$, and each $n\geq 0$.
\end{lemma}

(The proof of this lemma also shows that each of the internal $R\to\infty$ limits is well-defined when $s>0$ is fixed.)

\medskip

We will also require one further lemma, which shows in light of Carleman's condition (see e.g.\ \cite[Lemma II.5.9]{perkins2002part}) that the canonical measure is determined by the moments $\N[(\mu[\phi])^n]$ for compactly supported non-negative $\phi:\R^d\to \R$.
(This lemma does not really require the full strength of our running assumptions in this subsection, and in fact only requires that $d>2(\alpha \wedge 2)$, where $2(\alpha \wedge 2)$ is the critical dimension for recurrence of the critical branching random walk conditioned to survive forever.)

\begin{lemma}[Moment estimates]
\label{lem:canonical_measure_moment_upper_bound}
For each compactly-supported, continuous function $\varphi:\R^d\to \R$  there exists a finite constant $C(\varphi)$ such that 
\[
  \N\left[ |\mu[\varphi]|^p \right] \leq C(\varphi)^p p!
\]
for every integer $p\geq 1$.
\end{lemma}


\medskip

Before proving these lemmas, let us first explain how they imply \cref{thm:superprocess_main,thm:superprocess_main_regularly_varying}.

\begin{proof}[Proof of \cref{thm:superprocess_main,thm:superprocess_main_regularly_varying}]
To deduce the theorem from \cref{lem:tightness_scaling_limits,lem:scaling_limit_double_limit}, it suffices to prove that if we write $\tilde \P_{R,s}$ for the law of the standard monotone coupling of $\P_{\beta_c,\sigma^{-1}(s R)}$ and $\P_{\beta_c}$ and write $K_1$ and $K_2$ for the two clusters then for each $\eps>0$ there exists $s_0<\infty$ such that
\begin{equation}
\label{eq:final_sprinkle}
  \limsup_{R\to\infty} \frac{1}{\eta(R)}\tilde \P_{R,s}\Bigl(|K_2\setminus K_1| \geq \eps \zeta(R)\Bigr) \leq \eps
\end{equation}
for every $s\geq s_0$. Indeed, together with \cref{lem:scaling_limit_double_limit} this shows (by a variation on the bounded convergence theorem) that
\begin{equation*}
\lim_{R\to \infty}\frac{1}{\eta(R)} \E_{\beta_c} \left[g(\mu_R(\R^d)) (\mu[\varphi])^n
 \right] = \N \Bigl[g(\mu(\R^d)) (\mu[\varphi])^n \Bigr]
\end{equation*}
for 
each compactly supported continuous function $g:[0,\infty)\to [0,1]$ whose support does not contain zero, each bounded, non-negative, continuous function $\varphi:\R^d\to [0,\infty)$, and each $n\geq 0$.
 This implies by standard properties of moment measures (see e.g.\ \cite[Lemma II.5.9]{perkins2002part}) that
\begin{equation*}
\frac{1}{\eta(R)} \E_{\beta_c} \left[g(\mu_R(\R^d)) \mathbbm{1}(\mu_R\in \;\cdot\;)
 \right] \to \N \Bigl[g(\mu(\R^d)) \mathbbm{1}(\mu_R\in \;\cdot\;) \Bigr]
\end{equation*}
weakly as $R\to\infty$ for each compactly supported continuous function $g:[0,\infty)\to [0,1]$ whose support does not contain zero. This in turn is easily seen to imply the claim as written in the theorem in light of the tightness provided by \cref{lem:tightness_scaling_limits} (which lets us replace the assumption that $g$ is compactly supported with the assumption that it is bounded) and the fact that $\N$ gives mass zero to the set $\{\mu:\mu(\R^d)=\lambda\}$ for each $\lambda>0$ (which lets us replace the continuous function $g$ with the indicator $\mathbbm{1}( \;\cdot\; \geq \lambda )$ as in the portmanteau theorem \cite[Theorem 2.1]{billingsley2013convergence}).

\medskip

We now prove \eqref{eq:final_sprinkle}. It suffices to prove that
\[
  \tilde \P_{R}((K_2\setminus K_1) \cap \cG \neq \emptyset ) =o(\eta(R))
\]
as $R\to \infty$, where $\cG$ denotes an independent ghost field of intensity $\zeta(R)^{-1}$. We have by Markov's inequality that
\[
   \tilde \P_{R}((K_2\setminus K_1) \cap \cG \neq \emptyset \mid K_1) \preceq \sigma^{-1}(s_RR)^{-\alpha} |K_1|\P_{\beta_c}(0\leftrightarrow \cG),
\]
where the $\sigma(s_RR)^{-\alpha}|K_1|$ term bounds the expected number of edges in the boundary of $K_1$ that become open in $K_2$ and $\P_{\beta_c}(0\leftrightarrow \cG)$ bounds the probability that the other endpoint of one of these edges is connected to $\cG$ in $K_2$ off $K_1$. Since $\zeta(R)=\frac{1}{4}\hat\E_{\beta_c,\sigma^{-1}(R)}|K|$, we can apply the estimate \eqref{eq:volume_tail_without_inverses} (which is a simple interpretation of the tail estimates of \cref{thm:hd_moments_main,thm:critical_dim_moments_main_slowly_varying} removing the references to inverse functions)
to obtain that
\[
 \tilde \P_{R}((K_2\setminus K_1) \cap \cG \neq \emptyset \mid K_1) \preceq 
 \frac{\sigma^{-1}(R)^\alpha}{\sigma^{-1}(s_R R)^{\alpha}}
  \zeta(R)^{-1}
 |K_1|.
\]
Taking expectations over $K_1$, it follows that
\begin{multline}
 \tilde \P_{R}((K_2\setminus K_1) \cap \cG \neq \emptyset  ) \preceq \E_{\beta_c}\left[\min\left\{
 \frac{\sigma^{-1}(R)^\alpha}{\sigma^{-1}(s_R R)^{\alpha}}
  \zeta(R)^{-1}
 |K|,1\right\}\right] \\= o\left(\E_{\beta_c}\left[\min\left\{
  \zeta(R)^{-1}
 |K|,1\right\}\right]\right)
 \label{eq:K_2_not_K_1_ghost}
\end{multline}
as $R\to \infty$. Using that the the volume tail $\P_{\beta_c}(|K|\geq n)$ is regularly varying of index $-1/2 > -1$ we obtain that
\begin{align*}
\E_{\beta_c}\left[\min\left\{
  \zeta(R)^{-1}
 |K|,1\right\}\right]=\zeta(R)^{-1}\E_{\beta_c}\left[\min\left\{
 |K|,\zeta(R)\right\}\right] \asymp \P_{\beta_c}(|K|\geq \zeta(R)) \asymp \eta(R),
\end{align*}
%
%
%
where the final estimate follows from
 \eqref{eq:volume_tail_without_inverses}, and substituting this estimate into \eqref{eq:K_2_not_K_1_ghost} yields that
\[ \tilde \P_{R}((K_2\setminus K_1) \cap \cG \neq \emptyset  ) =o(\eta(R))\]
as $R\to \infty$ as claimed.
    The claimed asymptotics of $\eta(R)$ and $\zeta(R)$ under the hypotheses of \cref{thm:superprocess_main} follow from \cref{thm:hd_moments_main}.
\end{proof}


It remains to prove \cref{lem:tightness_scaling_limits,lem:scaling_limit_double_limit,lem:canonical_measure_moment_upper_bound}. We begin with \cref{lem:tightness_scaling_limits}, which is an easy consequence of our results on the volume distribution and the radius of gyration.

\begin{proof}[Proof of \cref{lem:tightness_scaling_limits}]
Fix $\eps,\delta>0$. The first inequality is trivial so we focus on the second.
We have by the same argument used to prove \eqref{eq:final_sprinkle} that there exist constants $s_0$ and $R_0$ such that if $(K_1,K_2)$ are the clusters of the origin in the standard monotone coupling of $\P_{\beta_c,\sigma^{-1}(s_0 R)}$ and $\P_{\beta_c}$ then
\begin{equation}
\label{eq:tightness_spatial_1}
  \P(|K_2\setminus K_1| \geq \delta \zeta(R)) \leq \delta \eta(R)
\end{equation}
for every $s\geq s_0$ and $R\geq R_0$. 
On the other hand, we also have that there exists a positive constant $c>0$ (arising from the equivalence of the norms $\|\cdot\|$ and $\|\cdot\|_2$) such that
\begin{align}
  \P_{\beta_c,\sigma^{-1}(s_0 R)}\Bigl(|K\setminus B_{s R}| \geq \eps \zeta(R)\Bigr) 
  &\leq  \P_{\beta_c,\sigma^{-1}(s_0 R)}\left( \sum_{x\in K} \|x\|_2^2 \geq c s^2 \eps r^2 \zeta(r)\right) 
  \nonumber\\
  &\preceq \frac{\E_{\beta_c,\sigma^{-1}(s_0 R)}\sum_{x\in K}\|x\|_2^2}{s^2 \eps R^2 \zeta(r)} =  \frac{4 \xi_2(\sigma^{-1}(s_0 R))^2 \E_{\beta_c,\sigma^{-1}(s_0 R)}|K|}{s^2 \eps R^2 \hat \E_{\beta_c,\sigma^{-1}(R)}|K| }
  \nonumber\\
  &\asymp_\eps 
   \frac{\xi_2(\sigma^{-1}(R))^2 \E_{\beta_c,\sigma^{-1}(R)}|K|}{s^2 \eps r^2 \hat \E_{\beta_c,\sigma^{-1}(R)}|K| } \asymp_\eps \frac{\eta(R)}{s^2},
   \label{eq:tightness_spatial_2}
\end{align}
where we used Markov's inequality in the second line, regular variation of $\zeta_2$, $\sigma^{-1}$, and $\E_{\beta_c,r}|K|$ in the first estimate on the third line, and the relation $\xi_2(\sigma^{-1}(R)) \asymp R$ in the final estimate on the third line. The claim follows easily from \eqref{eq:tightness_spatial_1} and \eqref{eq:tightness_spatial_2}
 together with \eqref{eq:volume_tail_without_inverses}, which implies that $\P_{\beta_c}(|K|\geq M \zeta(R)) \leq \delta \eta(R)$ when $M$ is large.
\end{proof}

We next briefly outline how \cref{lem:canonical_measure_moment_upper_bound} follows from arguments in the literature.

\begin{proof}[Proof of \cref{lem:canonical_measure_moment_upper_bound}]
It suffices to prove that
 for each compact set $\Lambda \subseteq \R^d$ there exists a finite constant $C(\Lambda)$ such that 
\[
  \N\left[ \mu(\Lambda)^p \right] \leq C(\Lambda)^p p!
\]
for every integer $p\geq 1$. Similar estimates have been proven for branching random walk with $d>2(\alpha \wedge 2)$ in \cite{angel2021tail,asselah2024intersection}, with the latter work \cite{asselah2024intersection} even identifying rather precisely the dependency of the constant $C(\Lambda)$ on the set $\Lambda$. These proofs easily extend to superprocesses and we omit the details. (Indeed, the analysis of \cite[Section 5]{angel2021tail} is based on an inductive analysis of diagrams of the same form as \eqref{eq:canonical_measure_diagrams} but with sums and lattice Green's functions instead of integrals and continuum Green's functions; passing from the discrete to the continuous versions of these calculations does not lead to any significant difficulties. The factorial term comes from counting the number of diagrams, while each diagram makes a contribution at most exponentially large in its number of edges.)
\end{proof}

We now begin to work towards the proof of \cref{lem:scaling_limit_double_limit}. We begin by relating the two kernels $G$ and $\kappa$ (where $\kappa$ was defined in \cref{def:limiting_displacement_law}):

\begin{lemma}
\label{lem:Greens_function_limit}
 The two kernels $\kappa$ and $G$ are related via
\[\lim_{\lambda\to \infty}\lambda^{\min\{2,\alpha\}-d} \kappa\left(\lambda^{-1}x\right) = G(x),\]
where the convergence holds in local $L^1$.
\end{lemma}

\begin{proof}[Proof of \cref{lem:Greens_function_limit}]
First suppose that $\alpha \geq 2$. Letting $p_t$ denote the transition kernel for Brownian motion with covariance matrix $\Sigma$, we can write
\[
  \kappa(\lambda^{-1} x) = \int_0^\infty e^{-t}p_t(0, \lambda^{-1}x)\dif t. 
\] 
Using the Brownian scaling property $p_t(0,\lambda^{-1}x)=\lambda^d p_{\lambda^2 t}(0,x)$, we obtain that
\[
  \kappa(\lambda^{-1} x) = \lambda^d \int_0^\infty e^{-t}p_{\lambda^2t}(0,x) \dif t = \lambda^{d-2} \int_0^\infty e^{-\lambda^{-2}t}p_{t}(0,x) \dif t.
\] 
The claim follows easily since $d\geq 6 \geq 3$ so that $\int_D \int_{T}^\infty p_t(0,x)\dif t \dif x\to 0$ as $T\to \infty$ for each compact set $D \subseteq \R^d$. Now suppose $\alpha<2$. In this case, we can write
\[
  \kappa(\lambda^{-1}x) = \int_0^\infty e^{-t} p_t^1(0,\lambda^{-1}x) \dif t
\]
where $p_t^1$ is the transition kernel of the symmetric L\'evy jump process $(L_t^1)_{t\geq 0}$ with L\'evy measure $\Pi_1$. Recognizing $\lambda^{-d} p_t^1(0,\lambda^{-1}x)$ as the density of $\lambda L_t^1$, which is equal to the density of the symmetric L\'evy jump process $(L_{s}^{\lambda})_{s\geq 0}$ associated to the measure $\Pi_{\lambda}$ evaluated at time $s=\lambda^\alpha t$, we have that
\begin{equation}
  \kappa(\lambda^{-1}x) =\lambda^d \int_0^\infty e^{-t} p_{\lambda^\alpha t}^{\lambda^{-\alpha}}(0,x) \dif t = \lambda^{d-\alpha} \int_0^\infty e^{-\lambda^{-\alpha}t} p_{t}^{\lambda}(0,x) \dif t
  \label{eq:kappa_lambda_exact}
\end{equation}
where $p_{\lambda^\alpha t}^{\lambda}$ denotes the transition kernel of this L\'evy process at time $\lambda^\alpha t$. Since the two L\'evy measures $\Pi_\infty$ and $\Pi_{\lambda}$ satisfy $\Pi_\infty \geq \Pi_{\lambda}$ and the $L^1$ difference between the two measures is finite, the two associated symmetric L\'evy jump processes can be coupled so that they differ by a compound Poisson process that makes jumps at rate 
\begin{equation*}
\Pi_\infty(\R^d)-\Pi_{\lambda}(\R^d) = \int_{\R^d} \left(\alpha \int_{\|x\|}^\infty s^{-d-\alpha-1}\dif s - \alpha \mathbbm{1}(\|x\|\leq \lambda)\int_{\|x\|}^{\lambda} s^{-d-\alpha-1}\dif s\right) \dif x 
= \lambda^{-\alpha}.
\end{equation*}
 On the event that this compound Poisson process makes no jumps up to time $t$, which occurs with probability $e^{-\lambda^{-\alpha}t}$ independently of the L\'evy process $(L^\lambda_s)_{s\geq 0}$, the two L\'evy processes $(L^\lambda_s)_{s\geq 0}$ and $(L^\infty_s)_{s\geq 0}$ coincide up to time $t$. This implies that $p_t^\lambda(0,\cdot)$ converges in $L^1$ to $p_t^\infty(0,\cdot)$ as $\lambda\to \infty$ uniformly for compact sets of $t$ and
 \[
   p^\infty_t(0,x) \geq e^{-\lambda^{-\alpha} t}p^\lambda_t(0,x)
 \]
 for every $\lambda,t>0$ and $x\in \R^d$. These two facts together allow us to apply the dominated convergence theorem to deduce the claim from \eqref{eq:kappa_lambda_exact}.
\end{proof}

%
%
%

\begin{lemma}
\label{lem:scaling_limit_diagrams_no_cutoff}
We have
\begin{equation*}
\lim_{s \to \infty} \lim_{R\to \infty} \frac{1}{\eta(R)}\E_{\beta_c,\sigma^{-1}(s R)} \left[ \prod_{i=1}^n\mu_R[\varphi_i]\right] 
= \N\left[ \prod_{i=1}^n \mu[\varphi_i] \right]
\end{equation*}
for each $n\geq 1$ and sequence of compactly supported continuous functions $\varphi_1,\ldots,\varphi_n :\R^d \to \R$.
\end{lemma}

(As before, the proof of this lemma also shows that each of the internal $R\to \infty$ limits is well-defined when $s>0$ is fixed.)

\begin{proof}[Proof of \cref{lem:scaling_limit_diagrams_no_cutoff}]
It suffices to consider the case that $\varphi_1,\ldots,\varphi_n:\R^d\to [0,\infty)$ are  non-negative. 
We can write 
\begin{align*}\E_{\beta_c,\sigma^{-1}(s R)} \left[ \prod_{i=1}^n\mu_R[\varphi_i]\right] &= 2^{2n}\E_{\beta_c,\sigma^{-1}(s R)} \left[ \frac{1}{(\hat \E_{\beta_c,\sigma^{-1}(R)}|K|)^n}\sum_{x_1,\ldots,x_n \in K} \prod_{i=1}^n \varphi_i\left(\frac{x_i}{R}\right)\right]
  \\&=2^{2n}
 \frac{\E_{\beta_c,\sigma^{-1}(s R)}|K| (\hat \E_{\beta_c,\sigma^{-1}(s R)}|K|)^{n-1}}{(\hat \E_{\beta_c,\sigma^{-1}(R)}|K|)^n} \\&\hspace{1.5cm}\cdot \hat \E_{\beta_c,\sigma^{-1}(s R)} \left[ \frac{1}{(\hat \E_{\beta_c,\sigma^{-1}(s R)}|K|)^{n-1}}\frac{1}{|K|}\sum_{x_1,\ldots,x_n \in K} \prod_{i=1}^n \varphi_i\left(s \cdot \frac{x_i}{s R}\right)\right]\end{align*}
and apply \cref{thm:scaling_limit_diagrams}, \cref{cor:scaling_limit_cut_off}, and the dominated convergence theorem to deduce that
\begin{align*}
  &2^{-2n}\E_{\beta_c,\sigma^{-1}(s R)} \left[ \prod_{i=1}^n\mu_R[\varphi_i]\right] 
  \\
  &\hspace{1.7cm}\sim
  s^{(2n-1)\min\{2,\alpha\}}\frac{\E_{\beta_c,\sigma^{-1}(R)}|K|}{\hat \E_{\beta_c,\sigma^{-1}(R)}|K|} \left(\sum_{T\in\bbT_n} \idotsint \prod_{\substack{i<j \\ i\sim j}} \kappa(x_i-x_j) \prod_{i=1}^n \varphi_i(s x_i)\dif x_1 \cdots \dif x_{2n-1}\right)
  \\
  &\hspace{1.7cm}=
  \frac{\E_{\beta_c,\sigma^{-1}(R)}|K|}{\hat \E_{\beta_c,\sigma^{-1}(R)}|K|} \left(\sum_{T\in\bbT_n} \idotsint \prod_{\substack{i<j \\ i\sim j}} s^{\min\{2,\alpha\}-d} \kappa\left(\frac{x_i-x_j}{s}\right) \prod_{i=1}^n \varphi_i(x_i)\dif x_1 \cdots \dif x_{2n-1}\right),
\end{align*}
as $R\to \infty$ for each fixed $s>0$, where we used that $\sigma^{-1}(R)$, $\E_{\beta_c,r}|K|$, and $\hat \E_{\beta_c,r}|K|$ are regularly varying of index $\min\{2/\alpha,1\}$, $\alpha$, and $2\alpha$ respectively in the penultimate line 
and used the change of variables $x_i\mapsto s x_i$ (which gives a factor $s^{-d}$ for each of the $2n-1$ variables of integration) in the final line. Applying \cref{lem:Greens_function_limit} and introducing the additional  factor of $4$ in the definition of $\eta(R)$ yields the claim in conjunction with \eqref{eq:canonical_measure_diagrams}.
\end{proof}

The deduction of \cref{lem:scaling_limit_double_limit} from \cref{lem:scaling_limit_diagrams_no_cutoff} that follows in the remainder of this section was suggested to us by Ed Perkins: The key idea is to rewrite things in such a way that we can reason in terms of probability measures and moment problems, even though $\bbN$ is not a probability measure and $\bbN( \mu(\R^d) \mid \mu(\R^d) \geq 1) = \infty$.


\begin{lemma}
\label{lem:scaling_limit_diagrams_no_cutoff2}
We have that
\begin{equation*}
\lim_{s \to \infty} \lim_{R\to \infty} \frac{1}{\eta(R)}\E_{\beta_c,\sigma^{-1}(s R)} \left[ f(\mu_R[\psi]) \left(\mu_R[\varphi]\right)^n \right] 
= \N\left[f(\mu_R[\psi])(\mu[\varphi])^n\right]
\end{equation*}
for each pair of compactly-supported, continuous, non-negative functions $\varphi,\psi$, integer $n\geq 0$, and bounded continuous function $f:[0,\infty)\to [0,\infty)$ whose support does not contain zero.
\end{lemma}

(Once again, the proof of this lemma also shows that each internal $R\to\infty$ limit is well-defined when $s>0$ is fixed.)


\begin{proof}[Proof of \cref{lem:scaling_limit_diagrams_no_cutoff2}]
Fix $\varphi$ and $\psi$. We may assume that neither $\varphi$ or $\psi$ is identically zero, the claim holding vacuously otherwise.
It follows immediately from \cref{lem:scaling_limit_diagrams_no_cutoff} that
\begin{equation*}
\lim_{s \to \infty} \lim_{R\to \infty} \frac{1}{\eta(R)}\E_{\beta_c,\sigma^{-1}(s R)} \left[ 
(\mu_R[\psi])^k(\mu_R[\varphi])^n \right] 
= \N\left[(\mu[\psi])^k(\mu[\varphi])^n\right]
\end{equation*}
for every pair of non-negative integers $n,k$ with $n+k>0$.
Indeed, this is just \cref{lem:scaling_limit_diagrams_no_cutoff} applied to the functions $\varphi_i$ with $\varphi_i=\varphi$ for $i=1,\ldots,n$ and $\varphi_i=\psi$ for $i=n+1,\ldots,n+k$.  
If $n\geq 1$ then this can be formulated as a statement about convergence of moments of appropriately defined probability measures:
\begin{equation*}
\lim_{s \to \infty} \lim_{R\to \infty} \frac{\E_{\beta_c,\sigma^{-1}(s R)} \left[ 
(\mu_R[\psi])^k (\mu_R[\varphi])^n \right]}{\E_{\beta_c,\sigma^{-1}(s R)} \left[ 
 (\mu_R[\varphi])^n \right]}
\\= \frac{\N\left[(\mu[\psi])^k(\mu[\varphi])^n\right]}{\bbN\left[(\mu[\varphi])^n\right]}.
\end{equation*}
By a standard argument, it follows from this, \cref{lem:canonical_measure_moment_upper_bound}, and Carleman's criterion that
\begin{equation}
\lim_{s \to \infty} \lim_{R\to \infty} \frac{\E_{\beta_c,\sigma^{-1}(s R)} \left[ 
\mathbbm{1}(\mu_R[\psi] \in  \cdot\,) (\mu_R[\varphi])^n \right]}{\E_{\beta_c,\sigma^{-1}(s R)} \left[ 
 (\mu_R[\varphi])^n \right]}
\\= \frac{\N\left[\mathbbm{1}(\mu[\psi] \in \cdot\,)
(\mu[\varphi])^n\right]}{\bbN\left[(\mu[\varphi])^n\right]},
\label{eq:weak_convergence_psi_phi}
\end{equation}
holds as a weak limit of probability measures, 
which implies the $n\geq 1$ case of the claim
 by a further application of \cref{lem:scaling_limit_diagrams_no_cutoff} (to identify the asymptotics of the denominator). (In particular, the internal $R\to \infty$ limit of these measures also exists for each fixed $s>0$ by the same argument, and is described by a measure having a similar diagrammatic formula for moments as $\N$ and $\bbK$ but with the kernel $\kappa_s(x)=s^{(\alpha \wedge 2)-d}\kappa(s^{-1}x)$ instead of $G$ or $\kappa$. The fact that this measure is also determined by these moment formulas when $s>0$ follows again from Carleman's criterion, with the analogue of \cref{lem:canonical_measure_moment_upper_bound} now being trivial since the kernel $\kappa_s$ is in $L^1$.)
It remains to prove the case $n=0$. Applying the weak convergence statement \eqref{eq:weak_convergence_psi_phi} with $n=1$ and $\varphi=\psi$, and considering the function $F(x)=\frac{1}{x}f(x)$ applied to $\mu_R[\psi]$ with $f$ an arbitrary bounded continuous function whose support does not contain zero, it follows that
\begin{equation*}
\lim_{s \to \infty} \lim_{R\to \infty} \frac{\E_{\beta_c,\sigma^{-1}(s R)} \left[f(\mu_R[\psi])\right]}{\E_{\beta_c,\sigma^{-1}(s R)} \left[ 
 \mu_R[\psi] \right]}
\\= \frac{\N\left[f(\mu[\psi])
\right]}{\bbN\left[\mu[\psi]\right]}.
\end{equation*}
The $n=0$ case of the claim follows by another application of \cref{lem:scaling_limit_diagrams_no_cutoff} (again to identify the asymptotics of the denominator) since $f$ was arbitrary (similar considerations apply regarding the internal $R\to \infty$ limit as discussed above).
\end{proof}



\begin{proof}[Proof of \cref{lem:scaling_limit_double_limit}]
Fix $\varphi$ and let $\psi_k$ be an increasing sequence of compactly-supported, continuous, non-negative functions with $\sup_k \psi_k(x) \equiv 1$ such that $\psi_k \equiv 1$ on the ball of radius $k$ around the origin for each $k\geq 0$. It follows from \cref{lem:tightness_scaling_limits} that for each $\eps>0$ there exists $k_0$ such that
\[
\sup_{R,s\geq 1} \frac{1}{\eta(R)}\P_{\beta_c,\sigma^{-1}(sR)}\Bigl(\mu_R(\R^d)-\mu_R[\psi_k] \geq \eps\Bigr) \leq \eps
\]
for all $k\geq k_0$, 
and the claim follows easily from this and \cref{lem:scaling_limit_diagrams_no_cutoff2} by a simple variation on the bounded convergence theorem.
\end{proof}

\subsection*{Acknowledgements}
This work was supported by NSF grant DMS-1928930 and a Packard Fellowship for Science and Engineering. I thank Roland Bauerschmidt, Nicolas Curien, Jean-Fran\c{c}ois Le Gall, Remco van der Hofstad, Ed Perkins, and Gordon Slade for helpful comments and conversations, and also thank Ed Perkins for sharing his ideas on the weak convergence issues treated in \cref{subsec:scaling_limits_without_cut_off}.

\addcontentsline{toc}{section}{References}

 \setstretch{1}
 \footnotesize{
  \bibliographystyle{abbrv}
  \bibliography{unimodularthesis.bib}
  }

\end{document}